\documentclass[11pt,reqno]{amsart}
\usepackage{hyperref}
\usepackage{url}
\usepackage{amssymb}
\usepackage{tikz,enumerate}
\usepackage{diagbox}
\usepackage{appendix}
\usepackage{inputenc}
\usepackage{xcolor}
\allowdisplaybreaks[4]

\newtheorem{theorem}{Theorem}
\newtheorem{corollary}[theorem]{Corollary}

\newtheorem{lemma}[theorem]{Lemma}
\newtheorem{prop}[theorem]{Proposition}
\theoremstyle{definition}
\newtheorem{defn}[theorem]{Definition}
\newtheorem{rem}[theorem]{Remark}
\theoremstyle{remark}
\numberwithin{equation}{section}
\numberwithin{theorem}{section}

\DeclareMathOperator{\spt}{spt}
\DeclareMathOperator{\ord}{ord}
\DeclareMathOperator{\RE}{Re}
\DeclareMathOperator{\IM}{Im}

\DeclareMathOperator{\SL}{SL}
\def\Q{\mathbb Q}

\begin{document}
\title[Quotients of Eisenstein series and the Dedekind eta function]
 {Moments of ranks and cranks, and Quotients of Eisenstein Series and the Dedekind Eta Function}

\author{Liuquan Wang and Yifan Yang}
\address{School of Mathematics and Statistics, Wuhan University, Wuhan 430072, Hubei, People's Republic of China}
\address{Mathematics Division, National Center for Theoretical Sciences, Taipei 10617, Taiwan}
\email{wanglq@whu.edu.cn;mathlqwang@163.com}

\address{Department of Mathematics, National Taiwan University and
  National Center for Theoretical Sciences, Taipei 10617, Taiwan}
\email{yangyifan@ntu.edu.tw}



\begin{abstract}
Atkin and Garvan introduced the functions $N_k(n)$ and $M_k(n)$, which denote the $k$-th moments of ranks and cranks in the theory of partitions.
Let $e_{2r}(n)$ be the $n$-th Fourier coefficient of $E_{2r}(\tau)/\eta(\tau)$, where $E_{2r}(\tau)$ is the classical Eisenstein series of weight $2r$ and $\eta(\tau)$ is the Dedekind eta function. Via the theory of quasi-modular forms, we find that for $k \leq 5$, $N_k(n)$ and $M_k(n)$ can be expressed using $e_{2r}(n)$ ($0\leq r \leq k$), $p(n)$ and $N_2(n)$. For $k>5$, additional functions are required for such expressions. For $r\in \{2, 3, 4, 5, 7\}$, by studying the action of Hecke operators on $E_{2r}(\tau)/\eta(\tau)$, we provide explicit congruences modulo arbitrary powers of primes for $e_{2r}(n)$. Moreover, for $\ell \in \{5, 7, 11, 13\}$ and any $k\geq 1$, we present uniform methods for finding nice representations for $\sum_{n=0}^\infty e_{2r}\left(\frac{\ell^{k}n+1}{24}\right)q^n$, which work for every $r\geq 2$.  These representations allow us to prove congruences modulo powers of $\ell$, and we have done so for $e_4(n)$ and $e_6(n)$ as examples. Based on the congruences satisfied by $e_{2r}(n)$, we establish  congruences modulo arbitrary powers of $\ell$ for the  moments and symmetrized moments of ranks and cranks as well as higher order $\spt$-functions.
\end{abstract}

\maketitle
\tableofcontents
\section{Introduction}
A partition of a positive integer $n$ is a way of writing it as a sum of positive integers in non-increasing order. We denote the number of partitions of $n$ by $p(n)$. For convention,  we agree that $p(0)=1$. The generating function of $p(n)$ can be written as
\begin{align}\label{p(n)-gen}
\sum_{n=0}^\infty p(n)q^n=\frac{1}{(q;q)_\infty}.
\end{align}
Here and throughout this paper, we assume that $|q|<1$ and adopt the $q$ series notation:
\begin{align}
(a;q)_\infty=\prod\limits_{n=0}^\infty (1-aq^n),  \quad (a;q)_n=\frac{(a;q)_\infty}{(aq^n;q)_\infty}.
\end{align}

Let $\lfloor x\rfloor$ denote the integer part of $x$. For $\ell \in \{5,7, 11\}$, let $\delta_{\ell,k}$ be the reciprocal of 24 modulo $\ell^k$, i.e., $24\delta_{\ell,k}\equiv 1$ (mod $\ell^{k}$). For $k \geq 1$ and $n\geq 0$, Ramanujan's congruences assert that
\begin{align}
p(5^kn+\delta_{5,k}) &\equiv 0 \pmod{5^k}, \label{pn-mod5} \\
p(7^kn+\delta_{7,k}) &\equiv 0 \pmod{7^{\lfloor k/2\rfloor+1}}, \label{pn-mod7}\\
p(11^kn+\delta_{11,k}) &\equiv 0 \pmod{11^{k}}. \label{pn-mod11}
\end{align}
Complete proofs of \eqref{pn-mod5} and \eqref{pn-mod7} were first given by Watson \cite{Watson}, and a proof of \eqref{pn-mod11} was first provided by Atkin \cite{Atkin}.

In 1944, in order to explain the congruences \eqref{pn-mod5}--\eqref{pn-mod11} with $k=1$ combinatorially, Dyson \cite{Dyson} introduced the concepts of rank and crank. He defined the rank of a partition as its largest part minus the number of parts. Let $N(m,n)$ be the number of partitions of $n$ with rank $m$, and $N(a, M, n)$ be the number of partitions of $n$ with rank $\equiv a$ (mod $M$). Dyson \cite{Dyson} conjectured that
\begin{align}
N(a,5, 5n+4)&=\frac{1}{5}p(5n+4), \quad 0\leq a \leq 4, \label{rank-5}\\
N(a,7, 7n+5)&=\frac{1}{7}p(7n+5), \quad 0\leq a \leq 6.  \label{rank-7}
\end{align}
By studying the generating functions of $N(a,M,n)$, Atkin and Swinnerton-Dyer \cite{ASD} proved these conjectures in 1954.

The rank successfully explains Ramanujan's congruences modulo 5 and 7. However, it does not explain the congruence modulo 11. Instead, Dyson guessed that there exists another partition statistic,  called  crank, that can explain the congruence modulo 11. The explicit definition of crank was given by Andrews and Garvan \cite{Andrews-Garvan} around 1988. Similar to the notion of rank, let $M(m,n)$ denote the number of partitions of $n$ with crank $m$, and $M(a,M,n)$ be the number of partitions of $n$ with crank $\equiv a$ (mod $M$). Andrews and Garvan proved that
\begin{align}
M(a, 11, 11n+6)=\frac{1}{11}p(11n+6), \quad  0\leq a \leq 10,
\end{align}
which clearly explains \eqref{pn-mod11} for $k=1$.

In 2003,  Atkin and Garvan \cite{Atkin-Garvan} studied the moments of ranks and cranks. For a nonnegative integer $j$ they defined
\begin{align}
N_j(n)=\sum_{k} k^j N(k,n), \label{rank-moment} \\
M_j(n)=\sum_{k} k^j M(k,n). \label{crank-moment}
\end{align}
The odd moments of ranks and cranks are all zero by the fact that
\begin{align}
N(k,n)=N(-k,n) \quad \textrm{and} \quad M(k,n)=M(-k,n). \label{NM-even}
\end{align}
The even moments are positive and possess nice arithmetic properties. From definition we have $N_0(n)=M_0(n)=p(n)$. Dyson \cite{Dyson-JCTA} combinatorially proved that
\begin{align}\label{M2-repn}
M_2(n)=2np(n).
\end{align}
Atkin and Garvan \cite{Atkin-Garvan} found that $N_j(n)$ and $M_j(n)$ are closely related. They proved that for $k=2, 3, 4$ and 5, there are polynomials $\alpha_k(n)$ of degree $k-1$ and $\beta_{k,j}(n)$ of degree $k-j$ ($1\leq j \leq k$) such that
\begin{align}
N_{2k}(n)=\alpha_k(n)N_2(n)+\sum_{j=1}^k \beta_{k,j}(n)M_{2j}(n).  \label{N-M-relation}
\end{align}
Moreover, for $k=6$, they showed that no relation of the form \eqref{N-M-relation} exists and for $k=7$, a similar relation exists but involves an extra term $N_{12}(n)$.

There is one case missing from \cite{Atkin-Garvan}, namely the relation between $N_2(n)$ and $M_2(n)$. This missing relation was later found by Andrews \cite{Andrews-spt} through a study of the smallest parts function. In 2008, Andrews \cite{Andrews-spt} introduced the function $\spt (n)$, which counts the total number of appearances of the smallest parts in all partitions of $n$.
He  proved that
\begin{align}
\spt(n)=np(n)-\frac{1}{2}N_2(n). \label{spt-N2}
\end{align}
In view of Dyson's identity \eqref{M2-repn}, we can rewrite \eqref{spt-N2} as
\begin{align}
\spt(n)=\frac{1}{2}\left(M_2(n)-N_2(n) \right). \label{spt-M2-N2}
\end{align}
This can be regarded as the relation between $M_2(n)$ and $N_2(n)$.

Andrews \cite{Andrews-spt} discovered that $\spt(n)$ satisfies some nice congruences modulo $5$, $7$ and $13$. Namely, for $n\geq 0$,
\begin{align}
\spt (5n+4) &\equiv 0 \pmod{5},  \label{spt-mod5} \\
\spt (7n+5) &\equiv 0 \pmod{7}, \label{spt-mod7} \\
\spt(13n+6) &\equiv 0 \pmod{13}. \label{spt-mod13}
\end{align}
These congruences motivate people to investigate further arithmetic properties of $\spt(n)$. For example,  Garvan \cite{Garvan-IJNT}  found congruences modulo $\ell$ for $\ell \in \{11, 17, 19, 29, 31, 37\}$. Ono \cite{Ono} showed that if $\ell\geq 5$ is a prime and $\left(\frac{1-24n}{\ell}\right)=1$, then
\begin{align}
\spt\left(\ell^2 n-\frac{\ell^2-1}{24}\right)\equiv 0 \pmod{\ell}.
\end{align}
For more detailed introduction to $\spt(n)$, see the recent survey of Chen \cite{Chen}.

In 2007, Andrews \cite{Andrews-2007} gave nice combinatorial interpretation to the moments of ranks. He introduced the $k$-th symmetrized moment function
\begin{align}
\eta_k(n):=\sum_{m=-\infty}^\infty \binom{m+\lfloor \frac{k-1}{2}\rfloor}{k} N(m,n). \label{eta-defn}
\end{align}
Again by \eqref{NM-even}, it is easy to see that $\eta_{2k+1}(n)=0$. From \eqref{eta-defn} we see that $\eta_{2k}(n)$ can be expressed as a linear sum in $N_{2j}(n)$ ($0\leq j \leq k$), and conversely, $N_{2k}(n)$ can be expressed as a linear sum in $\eta_{2j}(n)$ ($0\leq j \leq k$). The generating function of $\eta_{2k}(n)$ was given in \cite{Andrews-2007}:
\begin{align}\label{eta-gen}
\sum_{n=1}^\infty \eta_{2k}(n)q^n&=\frac{1}{(q;q)_\infty} \sum_{n=1}^\infty \frac{(1+q^n)(-1)^{n-1}q^{n(3n-1)/2+kn}}{(1-q^n)^{2k}}.
\end{align}
After introducing the concepts of Durfee symbols and $k$-marked Durfee symbols, Andrews showed that $\eta_{2k}(n)$ equals the number of $(k+1)$-marked Durfee symbols of $n$. He also proved some interesting congruences for $\eta_{2k}(n)$ as well as $N_{2k}(n)$. For instance, it was proved in \cite{Andrews-2007} that
\begin{align}
N_2(5n+r)\equiv \eta_2(5n+r) \equiv 0 \pmod{5}, \quad r\in \{1,4\}, \\
N_2(7n+r) \equiv \eta_2(7n+r) \equiv 0 \pmod{7}, \quad r \in \{1,5\}, \\
N_4(7n+r)\equiv \eta_4(7n+r) \equiv 0 \pmod{7}, \quad r\in \{1,5\}.
\end{align}

In 2011, following the work of Andrews \cite{Andrews-2007}, Garvan \cite{Garvan-AIM} defined the $k$-th symmetrized moment of cranks by
\begin{align}
\mu_k(n):=\sum_{m=-\infty}^\infty \binom{m+\lfloor \frac{k-1}{2}\rfloor}{k} M(m,n). \label{mu-defn}
\end{align}
Again, we know that $\mu_{2k+1}(n)=0$ for all $k\geq 0$. Garvan showed that many of Andrews' result in \cite{Andrews-2007} for $\eta_k(n)$ can be analogously extended to $\mu_k(n)$. For example, similar to \eqref{eta-gen}, he proved that (see \cite[Theorem 2.2]{Garvan-AIM})
\begin{align}\label{mu-gen}
\sum_{n=1}^\infty \mu_{2k}(n)q^n &=\frac{1}{(q;q)_\infty}\sum_{n=1}^\infty \frac{(1+q^n)(-1)^{n-1}q^{n(n-1)/2+kn}}{(1-q^n)^{2k}}.
\end{align}
Furthermore, motivated by the relation \eqref{spt-M2-N2}, Garvan defined a higher order $\spt$-function $\spt_k(n)$ so that
\begin{align}
\spt_k(n)=\mu_{2k}(n)-\eta_{2k}(n) \label{sptk-defn}
\end{align}
for all $k\geq 1$. He gave a combinatorial definition of $\spt_k(n)$ in terms of a weighted sum over the partitions of $n$. In particular, when $k=1$, we have $\spt_1(n)=\spt(n)$.
Garvan \cite{Garvan-AIM} also provided some explicit congruences modulo small primes for $\spt_k(n)$ with $k\in \{2,3,4\}$. A sample of such congruences is
\begin{align}
\spt_2(5n+r) &\equiv 0 \pmod{5}, \quad r \in \{0,1,4\}, \label{spt2-mod5}\\
\spt_2(7n+r) &\equiv 0 \pmod{7}, \quad r \in \{0,1,5\}, \label{spt2-mod7} \\
\spt_3(7n+r) &\equiv 0 \pmod{7}, \quad r\in \{0,1,2,4,5\}. \label{spt3-mod7}
\end{align}

From the above works, we see that the moments of ranks and cranks possess rich combinatorial and arithmetic properties. Investigating  the congruences satisfied by them becomes one of the central objects in research related to this topic. In this direction, Bringmann \cite{Bringmann} gave a remarkable result. By regarding the generating function of the second moment as a quasi-weak Maass form, she showed that given any prime $\ell>3$ with $k$ and $j$ fixed, there are infinitely many arithmetic progressions $An+B$ such that
\begin{align}
\eta_{2}(An+B) \equiv 0 \pmod{\ell^j}.
\end{align}
Later Bringmann, Garvan and Mahlburg \cite{BGM} studied the modularity of the generating functions of $\eta_{2k}(n)$. As a consequence, they showed that congruences of the above form also exist for $\eta_{2k}(n)$ for all $k\geq 1$.

As the existence of such congruences is now known, the next question is to discover and prove explicit congruences of this form.    As one of the open problems in  \cite{Andrews-2007}, Andrews asked the reader to prove that
\begin{align}
\eta_4(25n+24) &\equiv 0 \pmod{5}. \label{eta4-mod5}
\end{align}
This congruence was proved by Garvan in his review \cite{Garvan-review} of Andrews' paper. Meanwhile, Garvan also discovered companion congruences
\begin{align}
\eta_6(49n+r) &\equiv 0 \pmod{7}, \quad r\in \{19,33,40,47\}. \label{eta6-mod7}
\end{align}
More explicit congruences such as
\begin{align*}
\eta_2(11^3n+479) &\equiv 0 \pmod{11}, \\
\eta_8(13^2n+162)&\equiv 0 \pmod{13}
\end{align*}
were discovered in \cite{BGM}.

In view of those congruences satisfied by $N_{2k}(n)$ and $\eta_{2k}(n)$, it is natural to ask whether the rank and crank moments satisfy congruences modulo arbitrary powers of primes like those of $p(n)$. The main goal of this paper is to answer this question in an affirmative way.

In order to study the arithmetic properties of $N_{2k}(n)$ and $M_{2k}(n)$, we first represent them using Fourier coefficients of some special weakly holomorphic modular forms of half-integral weights.  Throughout this paper, we let $q=e^{2\pi i\tau}$ with $\IM \tau>0$. Recall the Dedekind eta function
\begin{align}
\eta(\tau):=q^{1/24}(q;q)_\infty. \label{Dedekind}
\end{align}
Let $\sigma_{m}(n):=\sum_{d|n}d^{m}$ and $B_{n}$ be the $n$-th Bernoulli number defined by
\begin{align}\label{Bernoulli}
\frac{x}{e^x-1}=\sum_{n=0}^\infty \frac{B_nx^n}{n!}.
\end{align}
The classical weight $2r$ Eisenstein series on $\SL (2,\mathbb{Z})$ can be written as
\begin{align}
E_{2r}=E_{2r}(\tau):=1-\frac{4r}{B_{2r}}\sum_{n=1}^\infty \sigma_{2r-1}(n)q^n. \label{Eisenstein}
\end{align}
Note that for $r=1$, $E_2(\tau)$ is not a modular form but a
quasi-modular form. We define the sequence $e_{2r}(n)$ by
\begin{align}
\sum_{n=0}^\infty e_{2r}(n)q^n:=\frac{q^{1/24}E_{2r}(\tau)}{\eta(\tau)}=\frac{E_{2r}(\tau)}{(q;q)_\infty}. \label{e-defn}
\end{align}
For convenience, we let $E_0=E_0(\tau)=1$ so that $p(n)=e_0(n)$.

Following the notations in \cite{Atkin-Garvan}, we define
\begin{align}
&C_{2k}(q):=\sum_{n=0}^\infty M_{2k}(n)q^n, \label{C-defn} \\
&R_{2k}(q):=\sum_{n=0}^\infty N_{2k}(n)q^n. \label{R-defn}
\end{align}
By a result of Atkin and Garvan \cite{Atkin-Garvan}, we know that $(q;q)_\infty C_{2k}(q)$  is a quasi-modular form of weight $\leq 2k$. For $k\in \{2,3,4,5\}$, this fact enables us to express $M_{2k}(n)$ using $e_{2r}(n)$ and $p(n)$. For instance, for $k=2$ we find that
\begin{align}
M_4(n)&=\frac{1}{20}e_4(n)-\frac{1}{20}p(n)+2np(n)-12n^2p(n). \label{M4-repn}
\end{align}
In a similar way, for $j\in \{2,3, 4, 5\}$, we can express $N_{2j}(n)$ using $e_{2r}(n)$, $p(n)$ and $N_2(n)$. An example is that
\begin{align}
N_4(n)&=\frac{2}{15}e_4(n)-\frac{2}{15}p(n)+4np(n)-36n^2p(n)-12nN_2(n).  \label{N4-repn}
\end{align}
For $j\geq 6$, we need extra sequences to represent the moments of ranks and cranks. For instance, for $j=6$, besides the sequences $p(n)$, $e_{4}(n)$ and $e_6(n)$, we also need the coefficients of $(q;q)_{\infty}^{23}$ to represent $M_{6}(n)$ (see \eqref{M12-exp}).

Representations like \eqref{M4-repn} and \eqref{N4-repn} turn the study of  rank and crank moments to the study of $e_{2r}(n)$. Therefore, the major part of this paper will be devoted to investigating $e_{2r}(n)$.

By studying the action of half-integral weight Hecke operators on $E_{2r}(\tau)/\eta(\tau)$, we establish some explicit congruences modulo arbitrary powers of primes.  We show that if $r\in \{2,3,4,5,7\}$, $m\ge 1$ and $\ell \geq 5$ is a prime, then for any integer $n$ (see Theorem \ref{thm-general}),
\begin{align}
e_{2r}\left(\frac{\ell^{2m}n+1}{24} \right)\equiv 0 \pmod{\ell^{2(r-1)m}}. \label{e2r-gen-cong-intro}
\end{align}

Moreover, for $\ell \in \{5,7, 11, 13\}$, we illustrate unified
approaches for finding representations for
$\sum_{n=0}^{\infty}e_{2r}\left(\ell^k n+\delta_{\ell,k}
\right)q^n$. We also show how to derive congruences modulo arbitrary
powers of $\ell$ for $e_{2r}(n)$ using these representations. These
approaches work for all $r\geq 2$. For our purpose, we only illustrate
the methods using $r=2$ and $3$ as examples. We shall prove
congruences such as for $k\geq 1$ and $n\geq 0$ (see Theorems
\ref{thm-e4-5-cong}, \ref{thm-e4-7-cong}, \ref{thm-e4-13-cong} and
\ref{thm-e4-11-cong}),
\begin{align}
e_4(\ell^ kn +\delta_{\ell,k}) \equiv 0 \pmod{\ell^k}, \quad \ell \in \{5, 7, 11\}
\end{align}
and
\begin{align}
e_4({13}^{2k}n+\delta_{13,2k})\equiv 0 \pmod{{13}^{2k}}.
\end{align}
To the best of our knowledge,  the function $e_4(n)$ seems to be the first example that simple congruences modulo arbitrary powers of 5, 7, 11 and 13 are simultaneously satisfied.

Congruences for $e_{2r}(n)$ imply congruences for the moments of ranks and cranks. For example, for any integers $k\geq 1$ and $n\geq 0$, using \eqref{M4-repn}, \eqref{N4-repn} and congruences of $e_4(n)$, we prove that (see Theorems \ref{thm-M4-cong} and \ref{thm-N4-cong})
\begin{align}
&M_4(5^k n+\delta_{5,k})\equiv 0   \pmod{5^{k-1}}, \label{M4-5power-intro} \\
& M_4(7^k n+\delta_{7,k}) \equiv 0 \pmod{7^{\left\lfloor \frac{k+1}{2}\right\rfloor}}, \label{M4-7power-intro} \\
&M_4(11^k n+\delta_{11,k}) \equiv 0 \pmod{{11}^k}, \label{M4-11power-intro} \\
&N_4(5^k n+\delta_{5,k})\equiv 0 \pmod{5^{\left\lfloor \frac{k+1}{2} \right\rfloor}}, \label{N4-5power-intro} \\
&N_4(7^kn+\delta_{7,k})\equiv 0 \pmod{7^{\left\lfloor \frac{k+1}{2} \right\rfloor}}. \label{N4-7power-intro}
\end{align}
We also get similar congruences for the symmetrized moments and  higher order $\spt$-functions. For instance, we prove that for $k\geq 1$ and $n\geq 0$ (see Theorem \ref{thm-eta-cong}),
\begin{align}
&\eta_4(5^kn+\delta_{5,k})\equiv 0 \pmod{5^{\left\lfloor \frac{k+1}{2} \right\rfloor-\gamma_{k,1}}}, \label{eta4-5power-intro} \\
&\eta_4(7^kn+\delta_{7,k})\equiv 0 \pmod{7^{\left\lfloor \frac{k+1}{2} \right\rfloor}}, \label{eta4-7power-intro} \\
&\eta_6(5^kn+\delta_{5,k})\equiv 0 \pmod{5^{\left\lfloor \frac{k+1}{2}\right\rfloor-\gamma_{k,1}}}, \label{eta6-5power-intro}\\
&\eta_6(7^kn+\delta_{7,k})\equiv 0\pmod{7^{\left\lfloor\frac{k+1}{2}\right\rfloor-\gamma_{k,1}}}, \label{eta6-7power-intro}
\end{align}
where $\gamma_{k,1}$ equals 1 if $k=1$ and 0 otherwise. Note that \eqref{eta4-5power-intro} and \eqref{eta6-7power-intro} generalize Andrews' congruence \eqref{eta4-mod5} and the case $r=47$ of Garvan's congruence \eqref{eta6-mod7}, respectively.

The paper is organized as follows.  In Section \ref{sec-main-thm} we shall first study the action of half-integral weight Hecke operator on ${E_{2r}(\tau)}/{\eta(\tau)}$. Then we establish some congruences for $e_{2r}(n)$ including \eqref{e2r-gen-cong-intro}.  We then focus on the cases when $\ell \in \{5,7, 11, 13\}$. In Section \ref{sec-explicit} we treat the cases $\ell=5, 7$ and 13 uniformly. We first express ${E_{2r}(\ell \tau)}/{E_{2r}(\tau)}$ using hauptmoduls on $\Gamma_0(\ell)$. Then we show that  $\sum_{n=0}^\infty e_{2r}(\ell^kn+\delta_{\ell,k})q^n$ can be represented using Eisenstein series and some hauptmoduls on $\Gamma_0(\ell)$.

In Sections \ref{sec-5}--\ref{sec-13} we will use modular equations of order $\ell$ to present explicit representations for $\sum_{n=0}^\infty e_{2r}(\ell^kn+\delta_{\ell,k})q^n$ with $\ell$ being 5, 7 and 13, respectively. In Section \ref{sec-11}, we will use Atkin's basis for weakly holomorphic modular functions on $\Gamma_0(11)$ to treat the case $\ell=11$. Meanwhile, in Sections \ref{sec-5}--\ref{sec-11}, we will examine the $\ell$-adic properties of $e_{2r}(\ell^kn+\delta_{\ell,k})$ via the representations of their generating functions.

Finally, in Section \ref{sec-moment} we return to the moments of ranks and cranks. We first use the theory of quasi-modular forms to express $M_{2k}(n)$ and $N_{2k}(n)$ using sequences like $e_{2r}(n)$. Relations such as \eqref{M4-repn} and \eqref{N4-repn} will be proved. Then using these relations and the known congruences for $e_{2r}(n)$, we establish congruences for the moments and symmetrized moments of ranks and cranks as well as the higher order $\spt$-functions.

\section{$\ell$-adic properties of $e_{2r}(n)$}\label{sec-main-thm}

\subsection{Modular forms and Hecke operators}
Throughout this paper, we let $\mathcal{H}$ denote the upper half
complex plane. We let $X_{0}(N)=\Gamma_0(N) \backslash
(\mathcal{H}\cup \mathbb{Q}\cup \{\infty\})$ be the modular curve
associated to $\Gamma_{0}(N)$.

For $k\in \frac{1}{2}\mathbb{Z}$, we denote by $\mathcal{M}_k(\Gamma_0(N),\chi)$ (resp.\ $\mathcal{M}_k^{!}(\Gamma_0(N),\chi)$) the space of holomorphic (resp.\ weakly holomorphic) modular forms of weight $k$ with Nebentypus $\chi$ on $\Gamma_0(N)$. In other words, $\mathcal{M}_k^{!}(\Gamma_0(N),\chi)$ contains exactly those meromorphic forms whose poles are supported at the cusps of $\Gamma_0(N)$. If $\chi$ is trivial, then we drop it from the notation. Moreover, if $\chi$ is trivial and $N=1$, we simply write $\mathcal{M}_k(\Gamma_0(N),\chi)$ as $\mathcal{M}_k$.

The key tool in this section will be half-integral weight Hecke operators, and we now define them.
\begin{defn}\label{defn-Hecke}
Let $\lambda$ be an integer. Suppose that
\begin{align}
f(\tau)=\sum_{n\in \mathbb{Z}}a(n)q^n \in \mathcal{M}_{\lambda +\frac{1}{2}}^{!} (\Gamma_0(4N), \chi).
\end{align}
For primes $\ell$, the Hecke operator $T_{\ell^2, \lambda, \chi}$ is defined by
\begin{align}
&f(\tau)\mid T_{\ell^2, \lambda, \chi}\nonumber \\
=&~ \sum_{n\in \mathbb{Z}} \left(a(\ell^2n)+ \left(\frac{(-1)^{\lambda}n}{\ell}\right)\chi(\ell)\ell^{\lambda -1} a(n)+\chi(\ell^2)\ell^{2\lambda -1}a(n/\ell^2)  \right)q^n,
\end{align}
where  $a(n/\ell^2)=0$ if $\ell^2 \nmid n$.
\end{defn}

We need the following relation: if $\ell \nmid N$, then
\begin{align}
T_{\ell^{2m+2},\lambda, \chi}=T_{\ell^2, \lambda, \chi}T_{\ell^{2m},\lambda, \chi}-\chi(\ell^2)\ell^{2\lambda -1}T_{\ell^{2m-2},\lambda, \chi}. \label{Hecek-rec}
\end{align}

From now on we simply write $T_{\ell^{2m}, \lambda, \chi}$ as $T_{\ell^{2m}}$. We assume that $\chi$ is a quadratic character with conductor coprime to $\ell$. Then
\begin{align}
f(\tau)\mid T_{\ell^2}=\sum_{n\in \mathbb{Z}}\left(a(\ell^2 n)+\left(\frac{(-1)^\lambda n}{\ell} \right)\chi(\ell)\ell^{\lambda-1} a(n)+\ell^{2\lambda-1}a(n/\ell^2) \right)q^n.
\end{align}

We now derive some formulas for the action $T_{\ell^{2m}}$. These formulas are of independent interests themselves, and will also play important roles in deducing congruences for $e_{2r}(n)$ in the next subsection.

We denote
\begin{align}
f(\tau) \mid T_{\ell^{2m}}=\sum_{n\in \mathbb{Z}} a_m(n)q^n.
\end{align}
In particular, we have $a_0(n)=a(n)$.  For any prime $\ell$, let $\pi_{\ell}(n)$ be the exponent of $\ell$ in the unique prime factorization of $n$ and we agree that $\pi_{\ell}(0)=\infty$.
\begin{lemma}\label{lem-am}
For $m\geq 1$, there exist integers $c(m,k;n,\ell)$ such that
\begin{align}
a_m(n)=\sum_{k=-m}^{m} c(m,k;n,\ell) a_0(\ell^{2k}n) \label{am-exp}
\end{align}
and
\begin{align}
\pi_{\ell}\left(c(m,k;n,\ell)\right)\geq (m-k)(\lambda-1). \label{c-ord}
\end{align}
Moreover,
\begin{align}
c(m,m, n; \ell)=1, \quad c(m,-m; n, \ell)=\ell^{(2\lambda-1)m}. \label{F-leading}
\end{align}
\end{lemma}
\begin{proof}
By definition we have
\begin{align}
a_1(n)=a_0(\ell^2 n)+ \left(\frac{(-1)^\lambda n}{\ell}\right)\chi(\ell)\ell^{\lambda -1} a_0(n)+\ell^{2\lambda -1}a_0(n/\ell^2). \label{T1-coeff}
\end{align}
Thus the lemma holds for $m=1$.

Suppose that the lemma holds for $m\leq s$. By \eqref{Hecek-rec} we have
\begin{align}
f(\tau) \mid T_{\ell^{2s+2}}&=f(\tau) \mid T_{\ell^{2s}} \mid T_{\ell^2}-\ell^{2\lambda-1} f(\tau)\mid T_{\ell^{2s-2}}.
\end{align}
Therefore,
\begin{align}
a_{s+1}(n)=&~ a_s(\ell^2 n)+ \left(\frac{(-1)^\lambda n}{\ell}\right)\chi(\ell)\ell^{\lambda -1} a_s(n)+\ell^{2\lambda -1}a_s(n/\ell^2)\nonumber \\
&~ -\ell^{2\lambda-1}a_{s-1}(n). \label{asn-rec-add}
\end{align}
Let $c(s+1,-s-1;n,\ell)=\ell^{(2\lambda-1)(s+1)}$. For $-s\leq k\leq s+1$ we let
\begin{align}
c(s+1,k;n,\ell)=&~ c(s,k-1;\ell^2n, \ell)+\left(\frac{(-1)^{\lambda}n}{\ell}\right)\chi(\ell)\ell^{\lambda-1}c(s,k;n,\ell)\nonumber \\
&+\ell^{2\lambda-1}c(s,k+1;n/\ell^2, \ell)-\ell^{2\lambda-1}c(s-1,k;n,\ell), \label{c-rec}
\end{align}
where $c(s,k+1;n/\ell^2,\ell)=0$ if $\ell^2 \nmid n$.
From \eqref{asn-rec-add} and \eqref{c-rec} we know that when expressing $a_{s+1}(n)$ in terms of $a_0(\ell^{2k}n)$ ($-s-1\leq k\leq s+1$), the integer $c(s+1,k;n,\ell)$  defined above gives the correct coefficient of $a_0(\ell^{2k}n)$ for $-s\leq k\leq s+1$.  As for $k=-s-1$, we argue as follows. If $\ell^{2s+2}|n$, then we have
\begin{align*}
c(s+1,-s-1;n,\ell)=\ell^{2\lambda-1}c(s,-s;n/\ell^2, \ell),
\end{align*}
which means \eqref{c-rec} also holds for $k=-s-1$. If $\ell^{2s+2}\nmid n$, then the value of the coefficient of $a_0(\ell^{-2s-2}n)$ does not matter since $a_0(\ell^{-2s-2}n)=0$. Thus \eqref{am-exp} holds for $s+1$.

Furthermore, from \eqref{c-rec} we have $c(s+1,s+1;n,\ell)=1$ and
\begin{align*}
&\pi_{\ell}\left(c(s+1,k;n,\ell) \right)\nonumber \\
\geq & \min \left\{(\lambda-1)(s-k+1), (\lambda-1)(s-k+1), (s-k-1)(\lambda-1) +2\lambda-1  \right\}\nonumber \\
=& (s-k+1)(\lambda-1).
\end{align*}
This shows that the lemma holds for $s+1$. By induction we complete the proof.
\end{proof}

We define $B_0(\tau)=f(\tau)$, and for $m\geq 1$,
\begin{align}
B_m(\tau)=\sum_{n}b_m(n)q^n:=f\mid T_{\ell^{2m}}-\chi(\ell)\ell^{\lambda-1}f\mid T_{\ell^{2m-2}}. \label{fm-defn}
\end{align}
We establish the following proposition, which generalizes \cite[Proposition 2.1]{ABL}.
\begin{prop}\label{prop-b}
$\mathrm{(1)}$  For any $n$, and for all $m\geq 1$,
\begin{align}
b_m(\ell^2 n)-\ell^{2\lambda-1} b_{m-1}(n)=a_0(\ell^{2m+2}n)-\chi(\ell)\ell^{\lambda-1} a_0(\ell^{2m}n).  \label{prop-b-1}
\end{align}
$\mathrm{(2)}$ If $\ell \nmid n$, then for all $m\geq 1$,
\begin{align}
b_m(n)=a_0(\ell^{2m}n)+\left(1-\left(\frac{(-1)^{\lambda}n}{\ell}\right)\right)\sum_{k=1}^{m}(-1)^k\chi(\ell)^{k}\ell^{(\lambda-1)k}a_0(\ell^{2m-2k}n). \label{prop-b-2}
\end{align}
$\mathrm{(3)}$ If $\ell \parallel n$, then for all $m\geq 1$,
\begin{align}
b_m(n)=a_0(\ell^{2m}n)-\chi(\ell)\ell^{\lambda-1}a_0(\ell^{2m-2}n). \label{prop-b-3}
\end{align}
\end{prop}
\begin{proof}
By \eqref{fm-defn} we have
\begin{align}
b_m(n)&=a_m(n)-\chi(\ell)\ell^{\lambda-1}a_{m-1}(n) \nonumber \\
&=a_{m-1}(\ell^2 n)+\chi(\ell)\ell^{\lambda-1}\left(\left(\frac{(-1)^{\lambda}n}{\ell}\right)-1 \right)a_{m-1}(n)+\ell^{2\lambda-1}a_{m-1}(n/\ell^2). \label{b-rec-proof-1}
\end{align}
In particular, for $m=1$ we obtain
\begin{align}
b_1(\ell^{2k}n)=a_0(\ell^{2k+2}n)-\chi(\ell)\ell^{\lambda-1}a_0(\ell^{2k}n)+\ell^{2\lambda-1}a_0(\ell^{2k-2}n). \label{b-rec-proof-2}
\end{align}
From \eqref{Hecek-rec} we have
\begin{align}
B_{m-1}\mid T_{\ell^2}&=\left(f\mid T_{\ell^{2m-2}}-\chi(\ell)\ell^{\lambda-1}f\mid T_{\ell^{2m-4}} \right)\mid T_{\ell^2} \nonumber \\
&=f\mid T_{\ell^{2m}}+\ell^{2\lambda-1}f\mid T_{\ell^{2m-4}}-\chi(\ell)\ell^{\lambda-1}f\mid T_{\ell^{2m-4}}\mid T_{\ell^2}. \label{b-rec-proof-3}
\end{align}
By definition, we have
\begin{align}
B_{m-2} =f\mid T_{\ell^{2m-4}}-\chi(\ell)\ell^{\lambda-1}f\mid T_{\ell^{2m-6}}. \label{b-rec-proof-4}
\end{align}
Multiplying both sides of \eqref{b-rec-proof-4} by $\ell^{2\lambda-1}$ and then subtracting from \eqref{b-rec-proof-3}, we deduce that
\begin{align}
&\quad B_{m-1}\mid T_{\ell^2}-\ell^{2\lambda-1}B_{m-2} \nonumber \\
&= f\mid T_{\ell^{2m}}-\chi(\ell)\ell^{\lambda-1}\left(f\mid T_{\ell^{2m-4}}\mid T_{\ell^2}-\ell^{2\lambda-1}f\mid T_{\ell^{2m-6}}  \right) \nonumber \\
&=f\mid T_{\ell^{2m}}-\chi(\ell)\ell^{\lambda-1}f\mid T_{\ell^{2m-2}},
\end{align}
which implies
\begin{align}
B_m=B_{m-1}\mid T_{\ell^2}-\ell^{2\lambda-1}B_{m-2}. \label{b-rec-proof-5}
\end{align}
From \eqref{b-rec-proof-5} we deduce that
\begin{align}
b_m(n)=b_{m-1}(\ell^2 n)+\left(\frac{(-1)^\lambda n}{\ell} \right)\chi(\ell)\ell^{\lambda-1}b_{m-1}(n)+\ell^{2\lambda-1}b_{m-1}(n/\ell^2)-\ell^{2\lambda-1}b_{m-2}(n). \label{bm-rec}
\end{align}
This implies
\begin{align*}
b_m(\ell^2n)-\ell^{2\lambda-1}b_{m-1}(n)=b_{m-1}(\ell^4n)-\ell^{2\lambda-1}b_{m-2}(\ell^2 n).
\end{align*}
Applying this identity repeatedly, we obtain
\begin{align}
b_m(\ell^2n)-\ell^{2\lambda-1}b_{m-1}(n)=b_{1}(\ell^{2m}n)-\ell^{2\lambda-1}b_{0}(\ell^{2m-2} n). \label{b-rec-proof-6}
\end{align}
Substituting \eqref{b-rec-proof-2} into \eqref{b-rec-proof-6}, we arrive at \eqref{prop-b-1}.

Now we assume that $\ell^2 \nmid n$. From \eqref{bm-rec} we deduce that
\begin{align}
b_m(n)=b_{m-1}(\ell^2n)+\left(\frac{(-1)^\lambda n}{\ell} \right)\chi(\ell)\ell^{\lambda-1}b_{m-1}(n)-\ell^{2\lambda-1}b_{m-2}(n). \label{b-rec-pf-7}
\end{align}
Substituting \eqref{prop-b-1} (with $m$ replaced by $m-1$) into \eqref{b-rec-pf-7}, we get
\begin{align}
b_m(n)=a_0(\ell^{2m}n)-\chi(\ell)\ell^{\lambda-1}a_0(\ell^{2m-2}n)+\left(\frac{(-1)^\lambda n}{\ell}\right)\chi(\ell)\ell^{\lambda-1}b_{m-1}(n). \label{b-rec-pf-8}
\end{align}

If $\ell \parallel n$, then \eqref{b-rec-pf-8} reduces to \eqref{prop-b-3}.

If $\ell \nmid n$, then \eqref{b-rec-pf-8} gives
\begin{align}
b_1(n)=a_0(\ell^2 n)-\left(1-\left(\frac{(-1)^\lambda n}{\ell}\right) \right)\chi(\ell)\ell^{\lambda-1}a_0(n).
\end{align}
This proves \eqref{prop-b-2} for $m=1$. By induction on $m$ and using \eqref{b-rec-pf-8}, we can prove \eqref{prop-b-2}.
\end{proof}

\subsection{Congruences modulo prime powers for $e_{2r}(n)$}
We will establish the following congruences for $e_{2r}(n)$.
\begin{theorem}\label{thm-general}
Suppose $r\in \{2,3,4,5,7\}$, $m\ge 1$ and $\ell \geq 5$ is a prime. For any integer $n$ we have
\begin{align}
e_{2r}\left(\frac{\ell^{2m}n+1}{24} \right)\equiv 0 \pmod{\ell^{2(r-1)m}}.
\end{align}
\end{theorem}
To prove it, we will follow the idea of Ahlgren, Bringmann and Lovejoy \cite{ABL}, where they used weight $\frac{3}{2}$ Hecke operators to study the $\ell$-adic properties of some smallest parts functions.

For $r\geq 2$ we set $\lambda=2r-1$ and $\chi(n) =(\frac{12}{n})$. We choose
\begin{align}
f(\tau)=\frac{E_{2r}(24\tau)}{\eta(24\tau)}=q^{-1}+ O(q^{23}) \in \mathcal{M}_{\lambda+\frac{1}{2}}^{!}\left(\Gamma_0(576), \chi  \right).
\end{align}
We  assume that $r\in \{2,3,4,5,7\}$. Note that for these $r$, all the coefficients of $E_{2r}$ are integers and hence $f(\tau)\in \mathbb{Z}[[q]]$.

To understand the action of $T_{\ell^{2m}}$ on $f(\tau)$, let us do some specific calculations. We have
\begin{align}
f(\tau)\mid T_{\ell^2}=\ell^{2\lambda-1}q^{-\ell^2}+\chi(\ell)\ell^{\lambda-1} q^{-1}+O(q^{23}). \label{f-T2}
\end{align}
Using \eqref{Hecek-rec} we have
\begin{align}
f(\tau)\mid T_{\ell^4}&=f(\tau)\mid T_{\ell^2}\mid T_{\ell^2} -\ell^{2\lambda-1} f(\tau) \nonumber \\
&=\ell^{4\lambda -2}q^{-\ell^4} +\ell^{3\lambda-2} \chi(\ell) q^{-\ell^2} +\ell^{2\lambda-2} q^{-1}+O(q^{23}). \label{f-T4}
\end{align}

These calculations lead us to the following observation.
\begin{lemma}\label{lem-f-vanish}
$f(\tau)\mid T_{\ell^{2m}}$ vanishes  modulo $\ell^{m(\lambda-1)}$ to order at least $\ell^{2m}+23$.
\end{lemma}
\begin{proof}
By Lemma \ref{lem-am} we can write
\begin{align}
f(\tau)\mid T_{\ell^{2m}}=\ell^{m(2\lambda-1)} q^{-\ell^{2m}}+\cdots . \label{f-T2m-expan}
\end{align}
First we note that $a_0(-1)=1$ and $a_0(n)\neq 0$ only if $n\equiv 23 \pmod{24}$. Since $\ell^{2k}\equiv 1$ (mod 24), we see that $a_0(\ell^{2k}n)\neq 0$ only if $n\equiv 23$ (mod 24). Therefore,  $a_m(n)\neq 0$ only if $n\equiv 23$ (mod 24). By Lemma \ref{lem-am} we see that if $n<0$, then
\begin{align}
a_m(n)=\sum_{k=-m}^{0} c(m,k;n,\ell) a_0(\ell^{2k}n).
\end{align}
That is, the index $k$ must be non-positive since $a_0(\ell^{2k}n)=0$ for $k\geq 1$ and $n<0$. Note that $k\leq 0$ implies
\begin{align}
\pi_{\ell}\left(c(m,k;n,\ell)\right)\geq (m-k)(\lambda-1)\geq m(\lambda-1).
\end{align}

Thus modulo $\ell^{m(\lambda-1)}$, $f(\tau)\mid T_{\ell^{2m}}$ vanishes to order at least $\ell^{2m}+23$.
\end{proof}

Now we are going to reduce the level and get back to modular forms on $\SL(2,\mathbb{Z})$.

Let
\begin{align}
F(\tau):=f(\tau)\mid T_{\ell^{2m}}, \quad G(\tau):= \eta(24\tau)^{\ell^{2m}}F(\tau),  \quad H(\tau):=G\left(\frac{\tau}{24}\right). \label{GFH-defn}
\end{align}
\begin{lemma}\label{lem-H}
We have $H(\tau)\in \mathcal{M}_{2r+\frac{\ell^{2m}-1}{2}}$.
\end{lemma}
\begin{proof}
It is clear that $G(\tau)\in \mathcal{M}_{2r+\frac{\ell^{2m}-1}{2}}^{!}(\Gamma_0(576))$. The coefficients of $G$ are supported on exponents divisible by 24. Therefore we have
\begin{align}
H(\tau) \in \mathcal{M}_{2r+\frac{\ell^{2m}-1}{2}}^{!}(\Gamma_0(24)).
\end{align}
From \eqref{f-T2m-expan} we know that $H(\tau)$ is holomorphic on $\mathcal{H}\cup \{\infty\}$. Thus it suffices to show that
\begin{align}
H\left(-\frac{1}{24\tau}\right)=(24\tau)^{2r+\frac{\ell^{2m}-1}{2}}H(24\tau). \label{H-transform}
\end{align}
Recall that the Fricke involution $W_N$ in weight $k\in \frac{1}{2}+\mathbb{Z}$ is defined by
\begin{align}
g(\tau)\mid_k W_N:=(-iN^{1/2}\tau)^{-k}g\left(-\frac{1}{N\tau}\right).
\end{align}
Since $\eta(-1/\tau)=\sqrt{-i\tau}\eta(\tau)$, we deduce that
\begin{align}
H\left(-\frac{1}{24\tau}\right)=&\eta\left(-\frac{1}{24\tau}\right)^{\ell^{2m}}F\left(-\frac{1}{576\tau} \right) \nonumber \\
=& (-24i\tau)^{\ell^{2m}/2} \eta(24\tau)^{\ell^{2m}} F\left(-\frac{1}{576\tau} \right) \nonumber \\
=& (-1)^r (24\tau)^{2r+(\ell^{2m}-1)/2}\eta(24\tau)^{\ell^{2m}} F(\tau)\mid_{2r-\frac{1}{2}}W_{576}, \label{H-proof}
\end{align}
where in the last line we used
\begin{align}
F(\tau)\mid_{2r-\frac{1}{2}}W_{576}=(-i\sqrt{576}\tau)^{\frac{1}{2}-2r}F\left(-\frac{1}{576\tau}\right).
\end{align}
Recall that $f(\tau)=\frac{E_{2r}(24\tau)}{\eta(24\tau)}$. We have
\begin{align}
f(\tau)\mid_{2r-\frac{1}{2}}W_{576}&=(-i\sqrt{576}\tau)^{\frac{1}{2}-2r}\frac{E_{2r}(-\frac{1}{24\tau})}{\eta(-\frac{1}{24\tau})} \nonumber \\
&=(-i\sqrt{576}\tau)^{\frac{1}{2}-2r}\frac{(24\tau)^{2r} E_{2r}(24\tau)}{\sqrt{-24i\tau} \eta(24\tau)} \nonumber \\
&=(-1)^r f(\tau).
\end{align}
Since $F(\tau)=f(\tau)\mid T_{\ell^{2m}}$ and $W_{576}$ commutes with $T_{\ell^{2m}}$, we conclude that
\begin{align}
F(\tau)\mid_{2r-\frac{1}{2}}W_{576}=(-1)^r F(\tau). \label{G-transform}
\end{align}
Substituting \eqref{G-transform} into \eqref{H-proof}, we arrive at \eqref{H-transform} and thereby complete the proof.
\end{proof}

\begin{lemma}
If $r\in \{2, 3, 4, 5, 7\}$, we have $F(\tau)\equiv 0$ \text{\rm{(mod $\ell^{m(\lambda-1)}$)}}, i.e., $a_m(n)\equiv 0$  \text{\rm{(mod $\ell^{m(\lambda-1)}$)}} for all $n\in \mathbb{Z}$.
\end{lemma}
\begin{proof}
Recall the dimension formula
\begin{align}
\dim \mathcal{M}_k=\left\{\begin{array}{ll}
\left\lfloor \frac{k}{12} \right\rfloor & k\equiv 2 \pmod{12}, \\
\left\lfloor \frac{k}{12} \right\rfloor +1 & \textrm{else}.
\end{array}
\right.
\end{align}
For $r\in \{2,3,4,5,7\}$ we have
\begin{align}
\dim \mathcal{M}_{2r+(\ell^{2m}-1)/2}=\frac{\ell^{2m}+23}{24}.
\end{align}
By Lemma \ref{lem-f-vanish} we know that $H(\tau)$ vanishes modulo $l^{m(\lambda-1)}$ to order at least $\frac{\ell^{2m}+23}{24}$. By Sturm's criterion, we know $H(\tau)$ is identically zero modulo $l^{m(\lambda-1)}$. This proves the lemma.
\end{proof}

\begin{prop}\label{prop-a0}
For $m,n \geq 0$ we have
\begin{align}
a_0(\ell^{2m}n)\equiv 0 \pmod{\ell^{m(\lambda-1)}}.
\end{align}
\end{prop}
\begin{proof}
Since $a_m(n)\equiv 0$  (mod $\ell^{m(\lambda-1)}$), in particular, for $m=1$ by \eqref{T1-coeff} this implies
\begin{align}
a_0(\ell^2 n)\equiv 0 \pmod{\ell^{\lambda-1}}.
\end{align}
Thus the proposition holds for $m=1$. Suppose it holds for all $m< s$ with $s> 1$.  By \eqref{am-exp} we have
\begin{align}
\sum_{k=-s}^{s} c(s,k;n,\ell) a_0(\ell^{2k}n)\equiv 0 \pmod{\ell^{s(\lambda-1)}}. \label{proof-sum-0}
\end{align}
Note that by induction and \eqref{c-ord}, we have for $0\leq k <s$,
\begin{align}
\pi_{\ell}\left(c(s,k;n,\ell) a_0(\ell^{2k}n)\right)\geq (s-k)(\lambda-1)+k(\lambda-1)=s(\lambda-1). \label{estimate}
\end{align}
For $k<0$, clearly \eqref{estimate} also holds. Since $c(s,s;n,\ell)=1$, we conclude from \eqref{proof-sum-0} and \eqref{estimate} that
\begin{align*}
a_0(\ell^{2s}n)\equiv 0  \pmod{\ell^{s(\lambda-1)}}.
\end{align*}
By induction we complete the proof.
\end{proof}

\begin{proof}[Proof of Theorem \ref{thm-general}]
Note that
\begin{align}
f(\tau)=\sum_{n=0}^\infty e_{2r}(n)q^{24n-1}.
\end{align}
Hence we have
\begin{align}
a_0(n)=e_{2r}\left(\frac{n+1}{24}\right).
\end{align}
Theorem \ref{thm-general} follows from this relation and Proposition \ref{prop-a0} by noting that $\lambda=2r-1$.
\end{proof}

For some small values of $m$, we have the following supplementary congruences.
\begin{theorem}\label{thm-2nd}
Suppose $r\in \{2,3,4,5,7\}$, $m\ge 1$ and $\ell \geq 5$ is a prime. \\
$\mathrm{(1)}$  If $\left(\frac{-n}{\ell} \right)=1$, we have
\begin{align}
&e_{2r}\left(\frac{\ell^{2}n+1}{24} \right)\equiv 0\pmod{\ell^{4r-3}}, \\
&e_{2r}\left(\frac{\ell^{4}n+1}{24} \right)\equiv 0\pmod{\ell^{8r-6}}.
\end{align}
$\mathrm{(2)}$ If $\ell \nmid n$, we have
\begin{align}
&e_{2r}\left(\frac{\ell^{2}n+1}{24} \right)\equiv \left(\frac{3}{\ell} \right)\ell^{2r-2}e_{2r}\left(\frac{n+1}{24} \right)   \pmod{\ell^{4r-3}}, \\
&e_{2r}\left(\frac{\ell^{4}n+1}{24} \right)\equiv \left(\frac{3}{\ell} \right)\ell^{2r-2}e_{2r}\left(\frac{\ell^2 n+1}{24} \right)   \pmod{\ell^{8r-6}}.
\end{align}
\end{theorem}

\begin{proof}
Recall the definition of $B_m$ in \eqref{fm-defn}.
From \eqref{f-T2} and \eqref{f-T4} we obtain
\begin{align}
B_1=\ell^{2\lambda-1}q^{-\ell^2}+O(q^{23}), \\
B_2=\ell^{4\lambda-2}q^{-\ell^4}+O(q^{23}).
\end{align}
Replacing $f$ by $B_1$ and $B_2$ in \eqref{GFH-defn} and arguing similarly as in the proof of Theorem \ref{thm-general}, we can show that
\begin{align}
b_1(n)\equiv 0 \pmod{\ell^{2\lambda-1}} \label{b1-cong}
\end{align}
and
\begin{align}
b_2(n)\equiv 0 \pmod{\ell^{4\lambda-2}}. \label{b2-cong}
\end{align}
Note that $\lambda=2r-1$ is odd. If $\left(\frac{-n}{\ell}\right)=1$, then \eqref{prop-b-2} implies $b_m(n)=a_0(\ell^{2m}n)$. Using this fact, the congruences \eqref{b1-cong} and \eqref{b2-cong} are equivalent to
\begin{align}
a_0(\ell^2n)\equiv 0 \pmod{\ell^{4r-3}},  \quad \textrm{and}  \quad a_0(\ell^4n)\equiv 0  \pmod{\ell^{8r-6}}.
\end{align}
This proves part (1).

Replacing $n$ by $\ell n$ in \eqref{prop-b-3} and combining with \eqref{b1-cong} and \eqref{b2-cong}, we get part (2).
\end{proof}

\section{Generating functions of $e_{2r}\left(\ell^k n+\delta_{\ell,k} \right)$ with $\ell \in \{5, 7, 13\}$}\label{sec-explicit}
In this section, we shall illustrate how to represent the generating functions of $e_{2r}\left(\ell^k n+\delta_{\ell,k} \right)$ for $\ell \in \{5,7, 13\}$. In particular, our procedure applies well to the case $r=0$, which gives representations for $\sum_{n=0}^\infty p\left(\ell^k n+\delta_{\ell,k} \right)q^n$.

Let
\begin{align}
Z_{\ell}=Z_{\ell}(\tau):=\frac{\eta(\ell^2 \tau)}{\eta(\tau)}.
\end{align}
For $\ell \in \{5, 7, 13\}$, modular functions on $\Gamma_0(\ell)$ can be expressed as rational functions of the hauptmodul
\begin{align}
Y_{\ell}=Y_{\ell}(\tau):=\left(\frac{\eta(\ell\tau)}{\eta(\tau)}\right)^{\frac{24}{\ell-1}}.
\end{align}

Atkin's $U$-operator is defined as
\begin{align}
\left(\sum_{n\in \mathbb{Z}}c(n)q^n \right)\mid U(d)=\sum_{n \in \mathbb{Z}} c(dn)q^n.
\end{align}
Recall that in the introduction we defined $E_0(\tau)=1$. We define for $r\geq 0$ that
\begin{align}
L_{2r,\ell,0}=L_{2r,\ell,0}(\tau):=E_{2r}(\tau).
\end{align}
For $i\geq 1$ we recursively define
\begin{align}
&L_{2r,\ell,2i-1}=L_{2r,\ell,2i-1}(\tau):=\left(Z_{\ell}(\tau)L_{2r,\ell,2i-2}(\tau) \right)\mid U_{\ell}, \label{L-odd-defn} \\
&L_{2r,\ell,2i}=L_{2r,\ell,2i}(\tau):=L_{2r,\ell,2i-1}(\tau) \mid U_{\ell}. \label{L-even-defn}
\end{align}
By induction on $i$ it is not difficult to see that
\begin{align}
&L_{2r, \ell, 2i-1}=(q^{\ell};q^{\ell})_\infty \sum_{n=0}^\infty e_{2r}\left(\ell^{2i-1}n+\delta_{\ell, 2i-1} \right)q^{n+1}, \label{L-odd}\\
&L_{2r, \ell, 2i}=(q;q)_\infty \sum_{n=0}^\infty e_{2r}\left(\ell^{2i}n+\delta_{\ell, 2i} \right)q^{n+1}. \label{L-even}
\end{align}
The main goal of this section is to show that $L_{2r, \ell, k}$ can be represented using $E_{2r}(\tau)$, $E_{2r}(\ell \tau)$ and $Y_{\ell}(\tau)$.

We first establish the following result, which is of independent interest itself.

\begin{prop}\label{prop-PQ}
For $\ell \in \{5,7, 13\}$ and $r\geq 2$, there exist polynomials $P_{2r, \ell}(x)$ and $Q_{2r,\ell}(x)$ such that $\gcd(P_{2r,\ell},Q_{2r,\ell})=1$ and
\begin{align}
\frac{E_{2r}(\ell \tau)}{E_{2r}(\tau)}=\frac{P_{2r,\ell}(Y_\ell)}{Q_{2r,\ell}(Y_\ell)}.
\end{align}
Moreover, we have $\deg P_{2r,\ell}=\deg Q_{2r,\ell}=d_{\ell}(r)$, where
\begin{align}
d_{5}(r)=\left\{\begin{array}{ll}
r & r\equiv 0 \pmod{2}, \\
r-1 & r\equiv 1 \pmod{2};
\end{array}
\right.
\end{align}
\begin{align}
d_{7}(r)=\left\{\begin{array}{ll}
\lfloor \frac{4r}{3} \rfloor & r \equiv 0, 2\pmod{3}, \\
\lfloor \frac{4r}{3} \rfloor -1 &r\equiv 1 \pmod{3}.
\end{array}\right.
\end{align}
and
\begin{align}
d_{13}(r)=\left\{\begin{array}{ll}
2r+\lfloor \frac{r}{3}\rfloor, &r\equiv 0,2 \pmod{6}, \\
2r-2+\lfloor \frac{r}{3}\rfloor, & r\equiv 1 \pmod{6}, \\
2r-1+\lfloor \frac{r}{3}\rfloor, & r\equiv 3, 4, 5 \pmod{6}.
\end{array}\right.
\end{align}
\end{prop}
\begin{rem}\label{rem-prop-PQ}
Let $d_{\ell}(0)=0$ and $P_{0,\ell}=Q_{0,\ell}=1$. This proposition
also holds trivially for $r=0$.
 We remark that
$$
\dim\mathcal M_{2r}(\Gamma_0(\ell))=d_\ell(r)+1.
$$
\end{rem}

Before we give a proof, we present some examples. By direct computations we find
\begin{align}
&\frac{E_{4}(5\tau)}{E_{4}(\tau)}=\frac{1+2\cdot 5Y_5+5Y_5^2}{1+2\cdot 5^3Y_{5}+5^5Y_5^2}, \label{E4-mod5-eq} \\
&\frac{E_4(7\tau)}{E_{4}(\tau)}= \frac{1+5Y_{7}+Y_{7}^2}{1+ 5\cdot 7^2Y_{7}+7^4Y_{7}^2}, \label{E4-mod7-eq} \\
&\frac{ E_4(13\tau)}{E_4(\tau)}=\frac{1+7Y_{13}+20Y_{13}^2+19Y_{13}^3+Y_{13}^4}{1+19\cdot 13 Y_{13}+20\cdot 13^2 Y_{13}^2+7\cdot {13}^3Y_{13}^3+{13}^4 Y_{13}^4}, \label{E4-mod13-eq} \\
&\frac{E_6(5\tau)}{E_{6}(\tau)}=\frac{1+4Y_5-Y_5^2}{1-4\cdot 5^3Y_5-5^6Y_{5}^2}, \label{E6-mod5-eq} \\
&\frac{E_6(7\tau)}{E_{6}(\tau)}=\frac{1+2\cdot 7Y_7+9\cdot 7Y_7^2+10\cdot 7Y_7^3-7Y_7^4}{1-10\cdot 7^2Y_7-9\cdot 7^4Y_7^2-2\cdot 7^6Y_7^3-7^7Y_7^4}, \label{E6-mod7-eq}  \\
&\frac{E_6(13\tau)}{E_{6}(\tau)}\nonumber \\
&=\frac{1+10Y_{13}+46Y_{13}^2+108Y_{13}^3+122Y_{13}^4+38Y_{13}^5-Y_{13}^6}{1-38\cdot 13 Y_{13}-122\cdot 13^2 Y_{13}^2-108\cdot {13}^3Y_{13}^3-46\cdot {13}^4 Y_{13}^4-10\cdot {13}^5Y_{13}^5-{13}^6Y_{13}^6}. \label{E6-mod13-eq}
\end{align}

To prove  Proposition \ref{prop-PQ}, we need some properties of the zeros of Eisenstein series. For any function $g(\tau)$, we let $\ord_p g$ be the order of $g(\tau)$ at the point $p$.   Let $\rho=e^{\frac{2\pi i}{3}}$. Let
\begin{align*}
\mathcal{F}:=\left\{\tau\in \mathcal{H}: |\tau|>1, -\frac{1}{2}\leq \RE \tau \leq 0\right\}\bigcup \left\{\tau \in \mathcal{H}: |\tau|>1, 0<\RE \tau <\frac{1}{2}\right\}
\end{align*}
be the standard fundamental domain for $\SL(2,\mathbb{Z})$.
The valence formula in the case of Eisenstein series $E_{2r}(\tau)$ says that
\begin{align}
\frac{1}{2}\ord_{i}E_{2r}+\frac{1}{3}\ord_{\rho}E_{2r}+\sum_{p \in \mathcal{F} \backslash \{i,\rho\}} \ord_{p} E_{2r} =\frac{r}{6}. \label{valence}
\end{align}
Therefore, the total number of zeros (counting multiplicities) of $E_{2r}(\tau)$ in the fundamental domain is $\frac{r}{6}$.
Rankin and Swinnerton-Dyer \cite{RS-D} found the locations of these zeros.
\begin{lemma}\label{lem-zero}
(Cf.\ \cite{RS-D}.) All the zeros of $E_{2r}$ ($r\geq 2$) in the fundamental domain of $\SL(2,\mathbb{Z})$ lie on the arc $\{\tau \in \mathbb{H}: |\tau|=1, -\frac{1}{2}\leq \RE \tau \leq 0\}$ and are all simple.
\end{lemma}
Moreover, the arguments in \cite{RS-D} also show that $E_{2r}(i)=0$ if and only if $r\equiv 1$ (mod 2) and $E_{2r}(\rho)=0$ if and only if $r\not\equiv 0$ (mod 3).

Kohnen \cite{Kohnen} studied the transcendence of the zeros.
\begin{lemma}\label{lem-trans}
(Cf.\ \cite{Kohnen}.) Let $\tau_0$ be a zero of $E_{2r}(\tau)$. If $\tau_0$ is not equivalent to $i$ or $\rho$, then $\tau_0$ is transcendental.
\end{lemma}
We also need the following fact.
\begin{lemma}\label{lem-common}
For any integer $m> 1$, $E_{2r}(m\tau)$ and $E_{2r}(\tau)$ do not share the same zeros, unless that they both vanish at some elliptic points.
\end{lemma}

\begin{proof}
 We prove by contradiction. Suppose there exists some $w$ which is not
 equivalent to $i$ or $\rho$ but $E_{2r}(w)=E_{2r}(m w)=0$. Without
 loss of generality, we assume that $w$ is in the fundamental domain
 of $\SL(2,\mathbb{Z})$. By Lemma \ref{lem-zero} we see that
 $|w|=1$. Since $m w$ is also a zero of $E_{2r}(\tau)$, there  exists
 some matrix $\begin{pmatrix} a& b \\ c& d \end{pmatrix} \in
 \SL(2,\mathbb{Z})$ and some zero $w'$ of $E_{2r}(\tau)$ in the
 fundamental domain such that
\begin{align}
\frac{amw+b}{cmw+d}=w'. \label{w-eq}
\end{align}
Again by Lemma \ref{lem-zero}, we have $|w'|=1$. Taking complex norm on both sides of \eqref{w-eq}, we see that
\begin{align}
\frac{(amw+b)(am\bar{w}+b)}{(cmw+d)(cm\bar{w}+d)}=w' \bar{w'}=1.
\end{align}
This implies
\begin{align}
2(ab-cd)m \RE w=(c^2-a^2)m^2+(d^2-b^2). \label{Imw}
\end{align}

If $ab\neq cd$, then \eqref{Imw} implies $\RE w \in \mathbb{Q}$. This together with $\IM w=\pm(1-\RE w)^{1/2}$ implies that $w$ is algebraic. By Lemma \ref{lem-trans} we see that $w$ must be equivalent to $i$ or $\rho$, which contradicts with the assumption.

If $ab=cd$, then \eqref{Imw} implies
\begin{align}
(am)^2+b^2=(cm)^2+d^2. \label{abcd}
\end{align}
Therefore,
\begin{align}
(am)^2+b^2\pm 2abm=(cm)^2+d^2 \pm 2cdm.
\end{align}
This implies $|am\pm b|=|cm\pm d|$ and therefore $a^2m^2-b^2=\pm (c^2m^2-d^2)$.

If $a^2m^2-b^2=c^2m^2-d^2$, then this together with \eqref{abcd} implies $a^2=c^2$ and $b^2=d^2$. We then deduce that $(ad-bc)(ad+bc)=a^2d^2-b^2c^2=0$. Since $ad-bc=1$, we must have $ad+bc=0$. But this implies $ad=\frac{1}{2}$, which is impossible.

If $a^2m^2-b^2=d^2-c^2m^2$, then this together with \eqref{abcd} implies $b^2=c^2m^2$ and $d^2=a^2m^2$. Therefore $m|b$ and $m|d$. But this implies $ad-bc=1\equiv 0$ (mod $m$), which is impossible.
\end{proof}

To prove Proposition \ref{prop-PQ}, we shall regard $E_{2r}(\ell \tau)$ and $E_{2r}(\tau)$ as modular forms in $\mathcal{M}_{2r}(\Gamma_0(\ell))$. We now discuss the number of zeros a modular form in $\mathcal{M}_{2r}(\Gamma_0(N))$ has at elliptic points. Let $\epsilon_2(N)$ and  $\epsilon_3(N)$ denote the number of elliptic points of orders 2 and 3 on $X_0(N)$, respectively. From \cite[p.\ 107]{Diamond-book} we find
\begin{align}
\epsilon_2(N)=\left\{\begin{array}{ll}
0 & \textrm{if $4|N$}, \\
\prod\limits_{p|N}\left(1+\left(\frac{-4}{p}\right) \right) & \textrm{otherwise},
\end{array} \right. \label{elliptic-2}\\
\epsilon_3(N)=\left\{\begin{array}{ll}
0 & \textrm{if $9|N$}, \\
\prod\limits_{p|N}\left(1+\left(\frac{-3}{p}\right) \right) & \textrm{otherwise}.
\end{array} \right. \label{elliptic-3}
\end{align}
Let $\alpha_2(N,k)$ and $\alpha_3(N,k)$ count the number of forced complex zeros of a form in $\mathcal{M}_k(\Gamma_0(N))$ at the elliptic points of order 2 and order 3, respectively. We have (see \cite[Eq.\ (3.3)]{Ahlgren}, for example)
\begin{align}\label{zero-formula}
(\alpha_2(N,k),\alpha_3(N,k))=\left\{\begin{array}{ll}
(\epsilon_2(N),2\epsilon_3(N)) & \textrm{if $k\equiv 2$ (mod 12)}, \\
(0, \epsilon_3(N)) & \textrm{if $k\equiv 4$ (mod 12)}, \\
(\epsilon_2(N),0) & \textrm{if $k\equiv 6$ (mod 12)}, \\
(0, 2\epsilon_3(N)) &\textrm{if $k\equiv 8$ (mod 12)}, \\
(\epsilon_2(N),\epsilon_3(N)) &\textrm{if $k\equiv 10$ (mod 12)}, \\
(0,0) & \textrm{if $k\equiv 0$ (mod 12)}.
\end{array}\right.
\end{align}
Let $\alpha(N,k)$ be the number of zeros (counting multiplicities) of $g$ at elliptic points. We have
\begin{align}
\alpha(N,k)=\frac{1}{2}\alpha_2(N,k)+\frac{1}{3}\alpha_3(N,k).
\end{align}

Now we are ready to prove Proposition \ref{prop-PQ}.
\begin{proof}[Proof of Proposition \ref{prop-PQ}]
Assume that the zeros of $E_{2r}(\ell \tau)$ are $z_1, z_2, \dots, z_n$ and the zeros of $E_{2r}(\tau)$ are $w_1, w_2, \dots, w_n$. Since $Y_\ell(\tau)$ is a hauptmodul on $\Gamma_0(\ell)$, the function
\begin{align}\label{Y-construct}
Y(\tau):=\prod\limits_{j=1}^n\frac{Y_{\ell}(\tau)-Y_{\ell}(z_j)}{Y_{\ell}(\tau)-Y_\ell(w_j)}
\end{align}
has the same zeros and poles as $\frac{E_{2r}(\ell \tau)}{E_{2r}(\tau)}$. Therefore, $Y(\tau)$ must equal to  $\frac{E_{2r}(\ell \tau)}{E_{2r}(\tau)}$ up to a scalar factor. This proves the existence of the polynomials $P_{2r,\ell}(x)$ and $Q_{2r,\ell}(x)$.

 Now we determine the degrees of $P_{2r,\ell}$ and $Q_{2r,\ell}$. By the valence formula \eqref{valence} and the fact that
\begin{align}
[\SL(2,\mathbb{Z}):\Gamma_0(\ell)]=1+\ell,
\end{align}
we see that the total number of zeros (counting multiplicities) of $E_{2r}(\tau)$ in  $X_0(\ell)$ is $\frac{r}{6}(\ell+1)$. Note that it is possible that one of the $z_j$'s is equal to one of the $w_k$'s. In this case, the corresponding factors in \eqref{Y-construct} cancel out and the degrees of the polynomials are reduced. According to Lemma \ref{lem-common}, the only possible common zeros of $E_{2r}(\ell \tau)$ and $E_{2r}(\tau)$ are at the elliptic points of $X_0(\ell)$. Thus we have
\begin{align}
\deg P_{2r,\ell}=\deg Q_{2r,\ell}=\frac{r}{6}(\ell+1)-\alpha(\ell,2r). \label{deg-formula}
\end{align}

By  \eqref{elliptic-2} and \eqref{elliptic-3} we find that
\begin{align*}
\epsilon_2(5)=2, \quad \epsilon_3(5)=0, \quad
\epsilon_2(7)=0, \quad \epsilon_3(7)=2, \quad
\epsilon_2(13)=2, \quad \epsilon_3(13)=2.
\end{align*}
From \eqref{zero-formula} we deduce that
\begin{align}
\alpha(5,2r)=\left\{\begin{array}{ll}
0 & \textrm{if $r\equiv 0$ (mod 2)} \\
1 &\textrm{if $r\equiv 1$ (mod 2)},
\end{array}\right.\\
\alpha(7,2r)=\left\{\begin{array}{ll}
0 & \textrm{if $r\equiv 0$ (mod 3)}, \\
\frac{4}{3} &\textrm{if $r\equiv 1$ (mod 3)}, \\
\frac{2}{3} &\textrm{if $r\equiv 2$ (mod 3)},
\end{array}\right. \\
\alpha(13,2r)=\left\{\begin{array}{ll}
0 & \textrm{if $r\equiv 0$ (mod 6)}, \\
\frac{7}{3} & \textrm{if $r\equiv 1$ (mod 6)}, \\
\frac{2}{3} & \textrm{if $r\equiv 2$ (mod 6)}, \\
1 & \textrm{if $r\equiv 3$ (mod 6)}, \\
\frac{4}{3} & \textrm{if $r\equiv 4$ (mod 6)}, \\
\frac{5}{3} & \textrm{if $r\equiv 5$ (mod 6)}.
\end{array}\right.
\end{align}
Substituting these results into \eqref{deg-formula}, we get the desired formulas for $d_{\ell}(r)$ with $\ell \in \{5,7,13\}$.
\end{proof}

Let $\mathbb{F}[x]$ denote the polynomial ring over a field $\mathbb{F}$. For $n\geq 0$, let $\mathbb{F}_n[x]$ denote the set of all polynomials in $x$ with degree no more than $n$ together with the zero polynomial. For convention, when $n<0$, we let $\mathbb{F}_{n}[x]=\{0\}$. We have the following  useful consequence of Proposition \ref{prop-PQ}.
\begin{corollary}\label{coro-sum}
For each $r\geq 2$ and $\ell \in \{5, 7, 13\}$, we have
\begin{align}
E_{2r}(\tau)\mathbb{C}[Y_{\ell}(\tau)]+ E_{2r}(\ell \tau)\mathbb{C}[Y_{\ell}(\tau)]=E_{2r}(\tau)\mathbb{C}[Y_{\ell}(\tau)]\oplus E_{2r}(\ell \tau)\mathbb{C}_{d_{\ell}(r)-1}[Y_{\ell}(\tau)].
\end{align}
\end{corollary}

Let $\lceil x\rceil$ denote the least integer greater than or equal to $x$. We can use $E_{2r}(\tau)$, $E_{2r}(\ell \tau)$ and $Y_{\ell}(\tau)$ to represent $L_{2r,\ell, k}$. For this we first prove the following fact.
\begin{lemma}\label{lem-UEZ-repn}
Assume that $\ell\in\{5,7,13\}$, $r\ge 0$ and $r\neq 1$. For an integer $n$, let
\begin{align*}
  a=a_{\ell,n}=\left\lceil\frac{(\ell^2-1)n}{24\ell}\right\rceil,
  \qquad
  b=b_{\ell,n}=\frac{(\ell^2-1)n}{24}-a.
\end{align*}
Assume that $g(\tau)\in \mathcal{M}_{2r}(\Gamma_0(\ell))\cap\Q[[q]]$. \\
$(1)$   If $n$ and $r$ satisfies $b\ge d_\ell(r)-1$, then
    $$
    (g(\tau)Z_\ell(\tau)^n)\big| U_\ell\in
    Y_\ell(\tau)^aE_{2r}(\tau)\Q_b[Y_\ell(\tau)]\oplus
    Y_\ell(\tau)^aE_{2r}(\ell\tau)\Q_{d_\ell(r)-1}[Y_\ell(\tau)].
    $$
$(2)$ If $n\ge0 $ and $b<d_\ell(r)-1$, then
    $$
    (g(\tau)Z_\ell(\tau)^n)\big| U_\ell\in
    Y_\ell(\tau)^aE_{2r}(\tau)\Q_{d_\ell(r)-1}[Y_\ell(\tau)]\oplus
    Y_\ell(\tau)^aE_{2r}(\ell\tau)\Q_{d_\ell(r)-1}[Y_\ell(\tau)].
    $$
$(3)$ If $n<0$ and $|a|\ge d_\ell(r)-1$, then
    $$
    (g(\tau)Z_\ell(\tau)^n)\big| U_\ell\in
    E_{2r}(\tau)\Q_{|a|}[Y_\ell(\tau)^{-1}]\oplus
    E_{2r}(\ell\tau)\Q_{d_\ell(r)-1}[Y_\ell(\tau)^{-1}].
    $$
$(4)$ If $n<0$ and $|a|<d_\ell(r)-1$, then
    $$
    (g(\tau)Z_\ell(\tau)^n)\big| U_\ell\in
    E_{2r}(\tau)\Q_{d_\ell(r)-1}[Y_\ell(\tau)^{-1}]\oplus
    E_{2r}(\ell\tau)\Q_{d_\ell(r)-1}[Y_\ell(\tau)^{-1}].
    $$
\end{lemma}
\begin{proof}
Let $f(\tau)=(g(\tau)Z_\ell(\tau)^n)|U_\ell$. We claim that in the case $n\ge 0$, we have
  $$
  \ord_\infty f\ge a, \qquad
  \ord_0f\ge -\frac{(\ell^2-1)n}{24}=-(a+b).
  $$

Indeed, since
$Z_\ell(\tau)=\eta(\ell^2\tau)/\eta(\tau)=q^{(\ell^2-1)/24}+\cdots$
 and $g(\tau)$ is holomorphic at $\infty$, the order of
  $f$ at infinity is at least $\lceil(\ell^2-1)n/(24\ell)\rceil=a$.
  To find the order of $f$ at $0$, we
  compute the Fourier expansion of
  $f|\left(\begin{smallmatrix}0&-1\\\ell&0\end{smallmatrix}
  \right)$. We have
  $$
  f\Big|\begin{pmatrix}0&-1\\\ell&0\end{pmatrix}
  =\frac1\ell\sum_{k=0}^{\ell-1}g(\tau)Z_\ell(\tau)^n\Big|
  \begin{pmatrix}1&k\\0&\ell\end{pmatrix}
  \begin{pmatrix}0&-1\\\ell&0\end{pmatrix}
  =\frac1\ell\sum_{k=0}^{\ell-1}g(\tau)Z_\ell(\tau)^n\Big|
  \begin{pmatrix}k\ell&-1\\\ell^2&0\end{pmatrix}.
  $$
  When $k=0$, we have
  \begin{equation} \label{equation: k=0}
  Z_\ell\left(-\frac{1}{\ell^2\tau}\right)=\frac{\eta(-1/\tau)}{\eta(-1/(\ell^2\tau))}
  =\frac{\eta(\tau)}{\ell\eta(\ell^2\tau)}
  =\frac1\ell q^{-(\ell^2-1)/24}+\cdots
  \end{equation}
  When $1\leq k\leq \ell-1$, we have
  $$
  \eta\left(\ell^2\frac{k\ell\tau-1}{\ell^2\tau}\right)
  =\eta(k\ell-1/\tau)=e^{2\pi ik\ell/24}
  \sqrt{-i\tau}\eta(\tau),
  $$
  and
  $$
  \eta\left(\frac{k\ell\tau-1}{\ell^2\tau}\right)
  =\eta\left(\begin{pmatrix}k&u\\\ell&v\end{pmatrix}
    (\tau-v/\ell)\right)=\epsilon_k\sqrt{-i\ell \tau}
  \eta(\tau-v/\ell)
  $$
  for some $24$th root of unity $\epsilon_k$, where $u$ and $v$ are
  integers satisfying $kv-\ell u=1$ and $0<v<\ell$. Therefore,
  we have
  \begin{equation} \label{equation: k<>0}
  Z_\ell\left(\begin{pmatrix}k\ell&-1\\\ell^2&0\end{pmatrix}
  \tau\right)=c_{k,0}+c_{k,1}q+\cdots
  \end{equation}
  for some nonzero number $c_{k,0}$. Together with \eqref{equation: k=0},
  this shows that the order of pole of $f$ at $0$ is at most
  $(\ell^2-1)n/24=a+b$. That is, $\ord_0f\ge-(a+b)$. This proves the
  claim.

  Now we consider the function $f(\tau)Y_\ell(\tau)^{-a}$. This is a
  weakly holomorphic modular form of weight $2r$ on $\Gamma_0(\ell)$
  that is holomorphic at $\infty$ and has a pole of
  order at most $(a+b)-a=b$ at $0$. Let
  $\mathcal{M}_{2r}^{(b)}(\Gamma_0(\ell))$ be the space of all such weakly
  holomorphic forms. Using the
  Riemann-Roch theorem or using the fact $\dim
  \mathcal{M}_{2r}(\Gamma_0(\ell))=d_\ell(r)+1$, we find that
  $$
  \dim \mathcal{M}_{2r}^{(b)}(\Gamma_0(\ell))=d_\ell(r)+b+1.
  $$

  Assume first that $b\ge d_\ell(r)-1$. It is clear that the $d_\ell(r)+b+1$
  functions
  $$
  E_{2r}(\tau)Y_\ell(\tau)^k, \quad k=0,\ldots, b
  $$
  and
  $$
  E_{2r}(\ell\tau)Y_\ell(\tau)^k, \quad k=0,\ldots,d_{\ell}(r)-1
  $$
  all belong to the space $\mathcal{M}_{2r}^{(b)}(\Gamma_0(\ell))$. Moreover,
  by Proposition \ref{prop-PQ}, they are linearly independent over $\mathbb C$.
  Therefore, they form a basis for the space
  $\mathcal{M}_{2r}^{(b)}(\Gamma_0(\ell))$. This proves the lemma for the case
  $b\ge d_{\ell}(r)-1$.

  When $n\ge 0$ and $b<d_{\ell}(r)-1$, we have
  $$
  f(\tau)Y_\ell(\tau)^{-a}\in \mathcal{M}_{2r}^{(b)}(\Gamma_0(\ell))
  \subset \mathcal{M}_{2r}^{(d_{\ell}(r)-1)}(\Gamma_0(\ell)).
  $$
  We can similarly show that
  $E_{2r}(\tau)Y_\ell(\tau)^k,E_{2r}(\ell\tau)Y_\ell(\tau)^k$,
  $k=0,\ldots,d_{\ell}(r)-1$, form a basis for
  $\mathcal{M}_{2r}^{(d_{\ell}(r)-1)}(\Gamma_0(\ell))$ and the assertion for the case
  $n\ge 0$ and $b<d_{\ell}(r)-1$ follows. We next consider the case $n<0$.

  In the case $n<0$, \eqref{equation: k=0} and \eqref{equation: k<>0}
  show that
  $$
  \ord_\infty f\ge a, \qquad \ord_0 f\ge 0,
  $$
  i.e., $f\in \mathcal{M}^{\{|a|\}}_{2r}(\Gamma_0(\ell))$, where
  for a nonnegative integer $m$, we let
  $\mathcal{M}^{\{m\}}_{2r}(\Gamma_0(\ell))$ denote the space of weakly holomorphic modular forms of weight $2r$ on $\Gamma_0(\ell)$ which is holomorphic at 0 and has a pole of order at most $m$ at $\infty$ . As in the previous cases, we have
  $$
  \dim \mathcal{M}^{\{m\}}_{2r}(\Gamma_0(\ell))=m+d_{\ell}(r)+1.
  $$
  When $m\ge d_{\ell}(r)-1$,
  $$
  E_{2r}(\tau)Y_\ell(\tau)^{-k}, \quad k=0,\ldots,m,
  $$
  and
  $$
  E_{2r}(\ell\tau)Y_\ell(\tau)^{-k},\quad k=0,\ldots,d_{\ell}(r)-1,
  $$
  form a basis for $\mathcal{M}^{\{m\}}_{2r}(\Gamma_0(\ell))$. Then Parts (3) and
  (4) of the lemma follow.
\end{proof}

\begin{theorem}\label{thm-Li-E-repn}
Let $r\geq 0$ and $r\neq 1$. For $\ell \in \{5,7, 13\}$ and $k\geq 1$, we have
\begin{align}\label{L-rli}
L_{2r,\ell, k}(\tau) \in E_{2r}(\tau)\mathbb{Q}[Y_{\ell}(\tau)] \oplus E_{2r}(\ell \tau)\mathbb{Q}_{d_{\ell}(r)-1}[Y_{\ell}(\tau)].
\end{align}
\end{theorem}
\begin{proof}
We first show that
\begin{align}\label{thm-Li-proof-0}
L_{2r,\ell, k}(\tau) \in E_{2r}(\tau)\mathbb{Q}[Y_{\ell}(\tau)] + E_{2r}(\ell \tau)\mathbb{Q}[Y_{\ell}(\tau)].
\end{align}
We proceed by induction on $k$. For $k=1$, since
\[L_{2r,\ell,1}=\left(E_{2r}(\tau) Z_{\ell}(\tau) \right)\mid U_{\ell},\]
by Parts (1) and (2) of Lemma \ref{lem-UEZ-repn} we know \eqref{thm-Li-proof-0} is true. Suppose \eqref{thm-Li-proof-0} is true for $k=2i-1$ where $i$ is some positive integer.
We first show that for each $n\geq 0$,
\begin{align}
\left(E_{2r}(\tau)Y_{\ell}(\tau)^n\right)\mid U_{\ell} \in E_{2r}(\tau)\mathbb{Q}[Y_{\ell}(\tau)]+E_{2r}(\ell \tau)\mathbb{Q}[Y_{\ell}(\tau)]. \label{thm-Li-proof-2}
\end{align}
In fact, we have
\begin{align}
\left(E_{2r}(\tau)Y_{\ell}(\tau)^n\right)\mid U_{\ell}&=\left(E_{2r}(\tau)\left(\frac{\eta(\ell^2 \tau)}{\eta(\tau)} \right)^{\frac{24n}{\ell-1}} \left(\frac{\eta(\ell \tau)}{\eta(\ell^2 \tau)} \right)^{\frac{24n}{\ell-1}} \right)\mid U_{\ell} \nonumber\\
&=Y_{\ell}(\tau)^{-n}\left(E_{2r}(\tau)Z_{\ell}(\tau)^{\frac{24n}{\ell-1}} \right)\mid U_{\ell}. \label{thm-Li-E-repn-proof-1}
\end{align}
From \eqref{thm-Li-E-repn-proof-1} and Parts (1)--(2) of Lemma \ref{lem-UEZ-repn} we see that \eqref{thm-Li-proof-2} holds.
Therefore, we have
\begin{align}\label{thm-Li-proof-4}
L_{2r,\ell, 2i}(\tau) =L_{2r,\ell, 2i-1}(\tau)\mid U_{\ell} \in E_{2r}(\tau)\mathbb{Q}[Y_{\ell}(\tau)] + E_{2r}(\ell \tau)\mathbb{Q}[Y_{\ell}(\tau)].
\end{align}
Next, since $L_{2r,\ell, 2i+1}(\tau)=\left(Z_{\ell}(\tau)L_{2r,\ell, 2i} \right)\mid U_{\ell}$, arguing similarly as before, we can show that
\eqref{thm-Li-proof-0} holds for $k=2i+1$. By induction we know that \eqref{thm-Li-proof-0} holds for all $k\geq 1$. The theorem then follows from Corollary \ref{coro-sum}.
\end{proof}
We now illustrate Theorem \ref{thm-Li-E-repn} with the case $r=2$. By direct computations we find that
\begin{align}
&L_{4,5,1}(\tau)=E_4(\tau)\left(\frac{3\cdot 5^4}{11}+\frac{6927\cdot 5}{11}Y_5 \right)-E_4(5\tau)\left(\frac{3\cdot 5^4}{11}+\frac{52\cdot 5^5}{11}Y_5 \right), \\
&L_{4,7,1}(\tau)=E_4(\tau)\left(\frac{23\cdot 7^3}{17}+\frac{40877\cdot 7}{17}Y_7 +7^6 Y_7^2 \right)-E_4(7\tau)\left(\frac{23\cdot 7^3}{17}+\frac{60\cdot 7^5}{17}Y_7 \right), \\
&L_{4,13,1}(\tau)=E_4(\tau)\Big(\frac{1146834\cdot {13}^3}{253}+\frac{87005903}{253}Y_{13}+\frac{3389888554\cdot 13}{253}Y_{13}^2  \nonumber \\
&\qquad \qquad \quad +\frac{46405373859\cdot {13}^2}{253}Y_{13}^3+ 1254970340\cdot {13}^7Y_{13}^4+4894384326\cdot {13}^8Y_{13}^5 \nonumber \\
&\qquad \qquad  \quad +10604499373\cdot {13}^9Y_{13}^6+10604499373\cdot {13}^9Y_{13}^7 \Big) \nonumber \\
&\qquad \qquad \quad -\frac{1}{253}E_4(13\tau)\left(1146834 \cdot {13}^3 + 322168080 \cdot {13}^4Y_{13}+4971613270\cdot {13}^6Y_{13}^2 \right. \nonumber \\
&\qquad \qquad \quad \left. + 12342150613\cdot {13}^6Y_{13}^3 \right).
\end{align}

If we use the above representations to analyze the $\ell$-adic orders of $e_{4}(\ell^{k}n+\delta_{\ell,k})$, then we need to examine two group of coefficients. Namely, one group of coefficients of $E_{4}(\tau)$ and one group of coefficients of $E_4(\ell \tau)$. This is not convenient and so we seek for alternative ways.

Recall the polynomials $P_{2r,\ell}$ and $Q_{2r,\ell}$ in Proposition \ref{prop-PQ}. We now define
\begin{align}
\widetilde{E}_{2r,\ell}=\widetilde{E}_{2r,\ell}(\tau):=\frac{E_{2r}(\tau)}{Q_{2r,\ell}(Y_{\ell})}=\frac{E_{2r}(\ell \tau)}{P_{2r,\ell}(Y_{\ell})}.
\end{align}
In particular, we let $\widetilde{E}_{0,\ell}=\widetilde{E}_{0,\ell}(\tau):=1$ (see Remark \ref{rem-prop-PQ}).
From the proof of Proposition \ref{prop-PQ} (in particular the construction of $Y(\tau)$ in \eqref{Y-construct}), we see that the zeros of $Q_{2r,\ell}$ are also zeros of $E_{2r}(\tau)$. Hence $\widetilde{E}_{2r,\ell}(\tau)$ is holomorphic on $\mathcal{H}\cup \{\infty\}$. This implies
\begin{align}
\widetilde{E}_{2r,\ell}(\tau) \in \mathcal{M}_{2r}(\Gamma_0(\ell)). \label{wE-modular}
\end{align}

  \begin{rem} \label{remark: E tilde}
    In fact, since $Y_\ell$ has a simple pole at the cusp
  $0$, the modular form $\widetilde E_{2r,\ell}(\tau)$ vanishes to the
  order $d_\ell(r)$ at $0$. Since $d_\ell(r)$ is the maximal possible
  order of vanishing of a modular form in $\mathcal
  M_{2r}(\Gamma_0(\ell))$ at $0$, up to scalars,
  $\widetilde E_{2r,\ell}(\tau)$ is the unique modular form with this
  vanishing order at $0$.
  \end{rem}

By Lemma \ref{lem-UEZ-repn} we know that there exists $m(i,j)\in \mathbb{Q}$ such that
\begin{align}
\left(\widetilde{E}_{2r,\ell} Z_{\ell}^i \right)\mid U_{\ell}=\widetilde{E}_{2r,\ell}\sum_{j=\left\lceil \frac{(\ell^2-1)i}{24\ell}\right\rceil}^{\infty} m(i,j)Y_{\ell}^j. \label{wEU-m}
\end{align}
Note that the sum on the right side is indeed finite. That is, for each $i$, there are only finitely many $j$ such that $m(i,j)\neq 0$.
Formula \eqref{wEU-m} leads to the following result.
\begin{theorem}\label{thm-L-wE-repn}
For $\ell \in \{5,7,13\}$, $k\geq 1$, $r\geq 0$ and $r\neq 1$, there exist integers $a(k,j)$ such that
\begin{align}
L_{2r,\ell,k}=\widetilde{E}_{2r,\ell}\sum_{j=1}^\infty a(k,j)Y_{\ell}^j. \label{general-L-wE-repn}
\end{align}
Moreover, for $i\geq 1$, we have the recurrence relations:
\begin{align}
&a(2i,j)=\sum_{k=1}^\infty a(2i-1,k)m\left(\frac{24k}{\ell-1},k+j \right), \label{general-a-even-rec} \\
&a(2i+1,j)=\sum_{k=1}^\infty a(2i,k)m\left(\frac{24k}{\ell-1}+1,k+j   \right). \label{general-a-odd-rec}
\end{align}
\end{theorem}
\begin{proof}
Theorem \ref{thm-Li-E-repn} implies that for $\ell \in \{5,7, 13\}$, $r\geq 0$ and $r\neq 1$,
\begin{align}\label{L-EQY}
L_{2r,\ell,k}\in \widetilde{E}_{2r,\ell}\mathbb{Q}[Y_{\ell}].
\end{align}
The representation \eqref{general-L-wE-repn} follows.

By \eqref{L-even-defn} and \eqref{wEU-m} we deduce that
\begin{align*}
L_{2r,\ell,2i}&=\sum_{k=1}^\infty a(2i-1,k)\left(\widetilde{E}_{2r,\ell} \left(\frac{\eta(\ell^2\tau)}{\eta(\tau)}\right)^{\frac{24k}{\ell-1}}\left(\frac{\eta(\ell\tau)}{\eta(\ell^2 \tau)}\right)^{\frac{24k}{\ell-1}} \right)\mid U_{\ell} \nonumber \\
&=\widetilde{E}_{2r,\ell} \sum_{k=1}^\infty \sum_{s=\left\lceil \frac{k(\ell+1)}{\ell}\right\rceil}^{\infty} a(2i-1,k)m\left(\frac{24k}{\ell-1},s\right)Y_\ell^{s-k} \quad (\textrm{replace $s-k$ by $j$}) \nonumber \\
&=\widetilde{E}_{2r,\ell}  \sum_{j=1}^\infty \left(\sum_{k=1}^\infty a(2i-1,k)m\left(\frac{24k}{\ell-1},k+j\right) \right)Y_\ell^j.
\end{align*}
This proves \eqref{general-a-even-rec}.

Similarly, by \eqref{L-odd-defn} and \eqref{wEU-m} we deduce that
\begin{align*}
L_{2r,\ell,2i+1}&=\sum_{k=1}^\infty a(2i,k)\left(Z_{\ell}\widetilde{E}_{2r,\ell} \left(\frac{\eta(\ell^2\tau)}{\eta(\tau)}\right)^{\frac{24k}{\ell-1}}\left(\frac{\eta(\ell\tau)}{\eta(\ell^2 \tau)}\right)^{\frac{24k}{\ell-1}} \right)\mid U_{\ell} \nonumber \\
&=\widetilde{E}_{2r,\ell} \sum_{k=1}^\infty \sum_{s=\left\lceil \frac{k(\ell+1)}{\ell}+\frac{\ell^2-1}{24\ell}\right\rceil}^{\infty} a(2i,k)m\left(\frac{24k}{\ell-1}+1,s\right)Y_\ell^{s-k} \quad (\textrm{replace $s-k$ by $j$}) \nonumber \\
&=\widetilde{E}_{2r,\ell}  \sum_{j=1}^\infty \left(\sum_{k=1}^\infty a(2i,k)m\left(\frac{24k}{\ell-1}+1,k+j\right) \right)Y_\ell^j.
\end{align*}
This proves \eqref{general-a-odd-rec}.
\end{proof}
In the representation \eqref{general-L-wE-repn}, we only have one group of coefficients. Moreover, these coefficients satisfy certain recurrence relations, which help us to analyze their $\ell$-adic properties.

In Sections \ref{sec-5}-\ref{sec-13}, we will use Theorem \ref{thm-L-wE-repn} to give explicit representations for $L_{2r,\ell,k}$ for $\ell=5, 7$ and 13, respectively. We shall illustrate the process for $r=2$ and 3 in details. For other $r\geq 2$, the process are similar.

\section{Generating functions and congruences for $e_{2r}(5^kn+\delta_{5,k})$}\label{sec-5}

In this section we treat the case $\ell=5$. We need the 5-th order modular equation:
\begin{align}
Z_{5}^5=\left(25Z_{5}^4+25Z_{5}^3+15Z_{5}^2+5Z_{5}+1\right)Y_{5}(5\tau). \label{modeq-5}
\end{align}
This equation has been used by Watson \cite{Watson} to prove Ramanujan's congruence \eqref{pn-mod5}.

\subsection{Congruences for $e_4(n)$ modulo powers of 5}
From \eqref{wE-modular} we know that $\widetilde{\mathcal{E}}_{4,5}\in \mathcal{M}_4(\Gamma_0(5))$. It is then not difficult to find that
\begin{align}
\widetilde{E}_{4,5}(\tau)=\frac{\eta(\tau)^{10}}{\eta(5\tau)^2}. \label{wE4-5-repn}
\end{align}
(We can check that $\eta(\tau)^{10}/\eta(5\tau)^2$ has a
  zero of order $2=d_5(2)$ at the cusp $0$ and hence must be a multiple of
  $\widetilde E_{4,5}(\tau)$. See Remark \ref{remark: E tilde}.)

While we can use the representation in Theorem \ref{thm-L-wE-repn}, in view of \eqref{wE4-5-repn}, we prefer to give equivalent representations in slightly different forms. As a consequence, the sequences $m(i,j)$ and $a(i,j)$ in this section will be different from those in Theorem \ref{thm-L-wE-repn}.

\begin{lemma}\label{lem-e4-U5-Z}
For any integer $i$, we have
\begin{align}
Z_{5}^i\mid U_5=\sum_{j=\left\lceil \frac{i}{5}\right\rceil}^\infty m(i,j)Y_{5}^j,
\end{align}
where $m(i,j)$ ($1\leq i \leq 5$, $j\geq 1$) are given by
\begin{align}
\begin{pmatrix}
 5 & 0 & 0 & 0 & 0 & 0 & 0 & \cdots\\
  2\cdot 5 & 5^3 & 0 & 0 & 0 & 0 & 0 & \cdots \\
 3^2 & 3\cdot 5^3 & 5^5 & 0 & 0 & 0 & 0 & \cdots \\
  2^2 & 2\cdot 5^2 \cdot 11 & 2^2 \cdot 5^5 & 5^7 &  0 & 0 & 0 & \cdots\\
  1 & 2^2\cdot 5^3 & 2^3\cdot 5^5 & 5^8 & 5^9 & 0 & 0 & \cdots
\end{pmatrix}.
\end{align}
For $i\leq 0$ or $i\geq 6$, $m(i,j)$ satisfies
\begin{align}
m(i,j)=&~25m(i-1,j-1)+25m(i-2,j-1)+15m(i-3,j-1) \nonumber \\
&~ +5m(i-4,j-1)+m(i-5,j-1). \label{m-5-rec}
\end{align}
\end{lemma}
For a proof of this lemma, see \cite{Hirschhorn-Hunt}.

\begin{lemma}\label{lem-m-5-ord}
(Cf.\ \cite[Lemma 4.1]{Hirschhorn-Hunt}.) For any $i,j$, we have
\begin{align}
\pi_5(m(i,j))\geq \left\lfloor \frac{5j-i-1}{2}\right\rfloor.
\end{align}
\end{lemma}

\begin{theorem}\label{thm-L45}
For $k\geq 1$, we have
\begin{align}
L_{4,5,k}(\tau)=\eta(\tau)^4\eta(5\tau)^4\sum_{j=0}^\infty a(k,j)Y_{5}^j, \label{L45-repn}
\end{align}
where
\begin{align}\label{e4-5-a-initial}
a(1,0)=5857\cdot 5, \quad a(1,1)=1874\cdot 5^4, \quad a(1,2)=5^{10}, \quad a(1,j)=0, \quad \forall j\geq 3,
\end{align}
and for $i\geq 1$,
\begin{align}
a(2i,j)=\sum_{k=0}^\infty a(2i-1,k)m(6k-4,k+j), \label{L25-a-even} \\
a(2i+1,j)=\sum_{k=0}^\infty a(2i,k)m(6k-3,k+j). \label{L25-a-odd}
\end{align}
\end{theorem}
\begin{proof}
For $k=1$, \eqref{L45-repn} can be verified directly. Now we suppose that \eqref{L45-repn} holds for $k=2i-1$ for some $i\geq 1$. That is,
\begin{align}
L_{4,5,2i-1}(\tau)=\eta(\tau)^4\eta(5\tau)^4\sum_{j=0}^\infty a(2i-1,j)Y_{5}^j. \label{L25-repn-pf-1}
\end{align}
By \eqref{L-even-defn} and Lemma \ref{lem-e4-U5-Z}, we have
\begin{align}
L_{4,5,2i}(\tau)&=\sum_{j=0}^\infty a(2i-1,j)\left(\left(\frac{\eta(25\tau)}{\eta(\tau)}\right)^{6j-4}\frac{\eta(5\tau)^{6j+4}}{\eta(25\tau)^{6j-4}} \right)\mid U_5 \nonumber \\
&=\eta(\tau)^4\eta(5\tau)^4\sum_{j=0}^\infty \sum_{s=\left\lceil \frac{6j-4}{5}\right\rceil}^{\infty} a(2i-1,j)m(6j-4,s)Y_5^{s-j} \quad (\textrm{replace $s-j$ by $k$}) \nonumber \\
&=\eta(\tau)^4\eta(5\tau)^4 \sum_{k=0}^\infty \left(\sum_{j=0}^\infty a(2i-1,j)m(6j-4,k+j) \right)Y_5^k.
\end{align}
Therefore, we have
\begin{align}
L_{4,5,2i}(\tau)=\eta(\tau)^4\eta(5\tau)^4 \sum_{j=0}^\infty a(2i,j)Y_5^j.
\end{align}
Next, by \eqref{L-odd-defn} and Lemma \ref{lem-e4-U5-Z} we have
\begin{align}
L_{4,5,2i+1}(\tau)&=\sum_{j=0}^\infty a(2i,j)\left(Z_{5}\left(\frac{\eta(25\tau)}{\eta(\tau)}\right)^{6j-4}\frac{\eta(5\tau)^{6j+4}}{\eta(25\tau)^{6j-4}} \right)\mid U_5 \nonumber \\
&=\eta(\tau)^4\eta(5\tau)^4\sum_{j=0}^\infty \sum_{s=\left\lceil\frac{6j-3}{5}\right\rceil} a(2i,j)m(6j-3,s)Y_5^{s-j} \quad (\textrm{replace $s-j$ by $k$}) \nonumber \\
&=\eta(\tau)^4\eta(5\tau)^4 \sum_{k=0}^\infty \left(\sum_{j=0}^\infty a(2i,j)m(6j-3,k+j) \right)Y_5^k.
\end{align}
Thus \eqref{L45-repn} holds for $k=2i+1$. By induction we know it holds for any $k\geq 1$.
\end{proof}

To establish congruences modulo powers of 5 for $e_4(5^kn+\delta_{5,k})$, we need to examine the 5-adic orders of the coefficients $a(k,j)$ in \eqref{L45-repn}.
\begin{lemma}\label{lem-L45-a-ord}
For $i\geq 1$ and $j\geq 0$ we have
\begin{align}
\pi_5(a(2i-1,j))\geq 2i-1+\left\lfloor \frac{5j}{2}\right\rfloor, \label{L45-a-odd-order} \\
\pi_5(a(2i,j)) \geq 2i+\left\lfloor \frac{5j+1}{2} \right\rfloor. \label{L45-a-even-order}
\end{align}
\end{lemma}
\begin{proof}
From \eqref{e4-5-a-initial} we see that \eqref{L45-a-odd-order} holds for $i=1$. Now suppose \eqref{L45-a-odd-order} holds for some $i\geq 1$. By \eqref{L25-a-even} and Lemma \ref{lem-m-5-ord} we have
\begin{align}
\pi_5(a(2i,j)) &\geq \min\limits_{k\geq 0} \pi_5(a(2i-1,k))+\pi_5(m(6k-4,k+j)) \nonumber \\
&\geq \min\limits_{k\geq 0} 2i-1+\left\lfloor \frac{5k}{2}\right\rfloor +\left\lfloor \frac{5j-k+3}{2}\right\rfloor \nonumber \\
&\geq 2i-1+\left\lfloor \frac{5j+3}{2}\right\rfloor \nonumber \\
&=2i+\left\lfloor \frac{5j+1}{2}\right\rfloor.
\end{align}
Next, by \eqref{L25-a-odd} and Lemma \ref{lem-m-5-ord} we deduce that
\begin{align}
\pi_5(a(2i+1,j)) &\geq \min\limits_{k\geq 0} \pi_5(a(2i,k))+\pi_5(m(6k-3,k+j)) \nonumber \\
&\geq \min\limits_{k\geq 0} 2i+\left\lfloor \frac{5k+1}{2}\right\rfloor +\left\lfloor \frac{5j-k+2}{2}\right\rfloor \nonumber \\
&\geq 2i+\left\lfloor \frac{5j+2}{2}\right\rfloor \nonumber \\
&=2i+1+\left\lfloor \frac{5j}{2}\right\rfloor.
\end{align}
This shows that \eqref{L45-a-odd-order} holds for $i+1$. By induction we know that both \eqref{L45-a-odd-order} and \eqref{L45-a-even-order} hold for any $i\geq 1$.
\end{proof}

\begin{theorem}\label{thm-e4-5-cong}
For any $k\geq 1$ and $n\geq 0$, we have
\begin{align}
e_4(5^kn+\delta_{5,k})\equiv 0 \pmod{5^k}.
\end{align}
\end{theorem}
\begin{proof}
This follows from Theorem \ref{thm-L45} and Lemma \ref{lem-L45-a-ord}.
\end{proof}
\begin{rem}
Setting $(\ell,r)=(5,2)$ in Theorem \ref{thm-general}, we get
\begin{align}
e_4(5^{2m}n+\delta_{5,2m})\equiv 0 \pmod{5^{2m}}.
\end{align}
We see that Theorem \ref{thm-e4-5-cong} extends this congruence.
\end{rem}
\subsection{Congruences for $e_6(n)$ modulo powers of 5}
For any prime $\ell$ we define
\begin{align}
\mathcal{E}_{2,\ell}=\mathcal{E}_{2,\ell}(\tau):=\frac{1}{\ell-1}(\ell E_2(\ell \tau)-E_2(\tau)).
\end{align}
It is known that $\mathcal{E}_{2,\ell}(\tau) \in \mathcal{M}_2(\Gamma_0(\ell))$. Since $\widetilde{\mathcal{E}}_{6,5} \in \mathcal{M}_6(\Gamma_0(5))$. It is not difficult to find that
\begin{align}
\widetilde{E}_{6,5}=\mathcal{E}_{2,5}\frac{\eta(\tau)^{10}}{\eta(5\tau)^{2}}. \label{wE6-5-repn}
\end{align}
(Again, because $\mathcal
  E_{2,5}\eta(\tau)^{10}/\eta(5\tau)^2$ has a zero of order
  $2$ at $0$, it must be a multiple of $\widetilde E_{6,5}$. See
  Remark \ref{remark: E tilde}.)

We will reuse the notation $m(i,j)$ and $a(i,j)$. The readers should note that they have different meanings in each subsection in Sections \ref{sec-5}--\ref{sec-11}.

\begin{lemma}\label{lem-UE2Z}
(Cf.\  \cite[Lemma 3.4]{Garvan-TAMS}.) For any integer $i$, we have
\begin{align}
\left(\mathcal{E}_{2,5}Z_{5}^i \right)\mid U_5=\mathcal{E}_{2,5}\sum_{j=\left\lceil \frac{i}{5}\right\rceil}^\infty m(i,j)Y_{5}^j,
\end{align}
where $m(i,j)$ satisfies the recurrence relation \eqref{m-5-rec}, and for $0\leq i \leq 4$, $j\geq 1$ it is given by
\begin{align}
\begin{pmatrix}
1 & 0 & 0 & 0 & 0 & 0 & \cdots \\
0 & 5^3 & 0 & 0 & 0 & 0 & \cdots \\
0 & 4\cdot 5^2 & 5^5 & 0 & 0 & 0 & \cdots \\
0 & 9\cdot 5 & 9\cdot 5^4 & 5^7 & 0 & 0 & \cdots \\
0 & 2 \dot 5 & 44 \cdot 5^3 & 14\cdot 5^6 & 5^9 & 0 \cdots
\end{pmatrix}.
\end{align}
\end{lemma}
\begin{rem}
 This lemma was only proved for $i\geq 0$ in \cite{Garvan-TAMS}. But the proof there also works for $i<0$ after extending the definition of $m(i,j)$ using the recurrence \eqref{m-5-rec}.
\end{rem}

\begin{lemma}\label{lem-E6-m-ord}
(Cf.\ \cite[Lemma 3.7]{Garvan-TAMS}.) For any integers $i$ and $j$,
\begin{align}
\pi_5(m(i,j))\geq \left\lfloor \frac{5j-i+1}{2}\right\rfloor.
\end{align}
\end{lemma}
Although in \cite{Garvan-TAMS} this lemma is only stated for $i\geq 0$, it is easy to see that using \eqref{m-5-rec}, the inequality also holds for $i<0$.

\begin{theorem}\label{thm-L65-repn}
For $k\geq 1$, we have
\begin{align}
L_{6,5,k}(\tau)=-\mathcal{E}_{2,5}(\tau)\eta(\tau)^4\eta(5\tau)^4\sum_{j=0}^\infty a(k,j)Y_{5}^j, \label{L65-repn}
\end{align}
where
\begin{align}
a(1,0)=27619\cdot 5^2, \quad a(1,1)=28124\cdot 5^5, \quad a(1,2)=5^{13}, \quad a(1,j)=0, \quad \forall j\geq 3,
\end{align}
and for $i\geq 1$,
\begin{align}
a(2i,j)=\sum_{k=0}^\infty a(2i-1,k)m(6k-4,k+j), \label{L65-a-even} \\
a(2i+1,j)=\sum_{k=0}^\infty a(2i,k)m(6k-3,k+j). \label{L65-a-odd}
\end{align}
\end{theorem}
\begin{proof}
We check \eqref{L65-repn} for $i=1$ directly. Then we can proceed by induction on $k$ as in the proof of Theorem \ref{thm-L45}. We omit the details.
\end{proof}
\begin{lemma}\label{lem-L65-a-order}
For $i\geq 1$ and $j\geq 0$,
\begin{align}
\pi_5(a(2i-1,j))\geq 4i-2+\left\lfloor \frac{5j}{2}\right\rfloor, \label{L65-a-odd-order}\\
\pi_5(a(2i,j))\geq 4i+\left\lfloor \frac{5j+1}{2}\right\rfloor. \label{L65-a-even-order}
\end{align}
\end{lemma}
This lemma can be proved in a way similar to the proof of Lemma \ref{lem-L45-a-ord}. We omit the details.

\begin{theorem}\label{thm-e6-5-cong}
For $k\geq 1$ and $n\geq 0$, we have
\begin{align}
e_6(5^kn+\delta_{5,k})\equiv 0 \pmod{5^{2k}}.
\end{align}
\begin{proof}
This follows from Theorem \ref{thm-L65-repn} and Lemma \ref{lem-L65-a-order}.
\end{proof}
\end{theorem}

\section{Generating functions and congruences for $e_{2r}(7^kn+\delta_{7,k})$}\label{sec-7}
To establish the generating function for $e_{2r}(7^kn+\delta_{7,k})$, we need the 7-th order modular equation that was used by Watson to prove Ramanujan's congruence \eqref{pn-mod7}:
\begin{align}
Z_{7}^7=&~ (1+7Z_7+21Z_{7}^2+49Z_{7}^3+147Z_{7}^4+343Z_{7}^5+343Z_{7}^6)Y_{7}(7\tau)^2 \nonumber \\
&~+(7Z_{7}^4+35Z_{7}^5+49Z_{7}^6)Y_7(7\tau). \label{modeq-7}
\end{align}

\subsection{Congruences for $e_4(n)$ modulo powers of 7}
Since $\widetilde{\mathcal{E}}_{4,7} \in \mathcal{M}_4(\Gamma_0(7))$, it is not difficult to find that
\begin{align}
\widetilde{\mathcal{E}}_{4,7}=\mathcal{E}_1 \frac{\eta(\tau)^7}{\eta(7\tau)}, \label{wE4-repn}
\end{align}
where $\mathcal{E}_1$ is the weight 1 Eisenstein series on $\Gamma_0(7)$ given by
\begin{align}
\mathcal{E}_1=\mathcal{E}_1(\tau):=1+2\sum_{n=1}^\infty \left(\frac{n}{7} \right)\frac{q^n}{1-q^n}.
\end{align}

\begin{lemma}\label{lem-e4-7-UEZ}
For any integer $i$, we have
\begin{align}
\mathcal{E}_1Z_7^i\mid U_{7} =\mathcal{E}_1 \sum_{j=\left\lceil \frac{2i}{7}\right\rceil}^\infty m(i,j)Y_{7}^j, \label{E1Z7-U-eq}
\end{align}
where $m(i,j)$ ($-6\leq i \leq 0$) are given in Group \uppercase\expandafter{\romannumeral1} of the Appendix, and for other $i$, $m(i,j)$ can be determined by the recurrence
\begin{align}\label{m-7-rec}
m(i,j)=&~ 7^2m(i-1,j-1)+5\cdot 7m(i-2,j-1)+7m(i-3,j-1)+7^3m(i-1,j-2) \nonumber \\
&~ +7^3m(i-2,j-2)+3\cdot 7^2m(i-3,j-2)+7^2m(i-4,j-2) \nonumber \\
&~ +3\cdot 7m(i-5,j-2)+7m(i-6,j-2)+m(i-7,j-2).
\end{align}
\end{lemma}
\begin{proof}
For $-6\leq i\leq 0$, \eqref{E1Z7-U-eq} can be verified directly. The lemma can then be proved using \eqref{modeq-7}.
\end{proof}

\begin{lemma}\label{lem-e4-7-m-order}
For any integers $i$ and $j$,
\begin{align}
\pi_7(m(i,j))\geq \left\lfloor \frac{7j-2i+1}{4} \right\rfloor.
\end{align}
\end{lemma}
\begin{proof}
For $-6\leq i\leq 0$, the desired inequality can be verified directly. For $i<-6$ or $i>0$, it can be proved by induction upon using \eqref{m-7-rec}.
\end{proof}

\begin{theorem}\label{thm-e4-7-gen}
For $k\geq 1$, we have
\begin{align}
L_{4, 7, k}(\tau)=\mathcal{E}_1 \eta(\tau)^3\eta(7\tau)^3 \sum_{j=0}^\infty a(k,j)Y_7^j, \label{e4-7-gen}
\end{align}
where
\begin{align}
&a(1,0)=9841\cdot 7, \quad a(1,1)=14748\cdot 7^3, \quad a(1,2)=12\cdot 7^8, \quad a(1,3)=7^{10}, \nonumber  \\
&a(1,j)=0, \quad  \forall j\geq 4, \label{a-e4-7-initial}
\end{align}
and for $i\geq 1$,
\begin{align}
&a(2i,j)=\sum_{k=0}^\infty a(2i-1,k)m(4k-3,k+j), \label{e4-7-a-even-rec} \\
&a(2i+1,j)=\sum_{k=0}^\infty a(2i,k)m(4k-2,k+j). \label{e4-7-a-odd-rec}
\end{align}
\end{theorem}
\begin{proof}
The case $k=1$ of \eqref{e4-7-gen} can be verified by direct computation. The theorem can then be proved by induction on $k$ upon using Lemma \ref{lem-e4-7-UEZ} and the following facts:
\begin{align*}
\left(\mathcal{E}_1\eta(\tau)^3\eta(7\tau)^3 Y_7^j \right)\mid U_7 &=\left(\mathcal{E}_1\left(\frac{\eta(49\tau)}{\eta(\tau)}\right)^{4j-3}\cdot \frac{\eta(7\tau)^{4j+3}}{\eta(49\tau)^{4j-3}} \right)\mid U_7 \nonumber \\
&=\eta(\tau)^3\eta(7\tau)^3 Y_{7}^{-j} \left(\mathcal{E}_1 Z_7^{4j-3} \mid U_7\right),
\end{align*}
\begin{align*}
\left(Z_{7}\mathcal{E}_1\eta(\tau)^3\eta(7\tau)^3 Y_7^j \right)\mid U_7 =\eta(\tau)^3\eta(7\tau)^3 Y_{7}^{-j} \left(\mathcal{E}_1 Z_7^{4j-2} \mid U_7\right).
\end{align*}
Since the process is similar to the proof of Theorem \ref{thm-L45}, we omit the details.
\end{proof}

\begin{lemma}\label{lem-e4-7-a-ord}
For $i\geq 1$ and $j\geq 0$, we have
\begin{align}
&\pi_7(a(2i-1,j))\geq 2i-1+\left\lfloor \frac{7j+1}{4}\right\rfloor, \label{e4-7-a-odd-ord} \\
&\pi_7(a(2i,j))\geq 2i+\left\lfloor \frac{7j+3}{4}\right\rfloor. \label{e4-7-a-even-ord}
\end{align}
\end{lemma}
\begin{proof}
From \eqref{a-e4-7-initial}
 we see that  \eqref{e4-7-a-odd-ord} holds for $i=1$. Now suppose it holds for some $i\geq 1$. By \eqref{e4-7-a-even-rec} and Lemma \ref{lem-e4-7-m-order}, we have
\begin{align*}
\pi_7(a(2i,j))&\geq \min\limits_{k\geq 0}\pi_7(a(2i-1,k))+\pi_7(m(4k-3,k+j)) \\
&\geq \min\limits_{k\geq 0} 2i-1+\left\lfloor \frac{7k+1}{4}\right\rfloor +\left\lfloor \frac{7j-k+7}{4}\right\rfloor \\
&\geq 2i+\left\lfloor \frac{7j+3}{4}\right\rfloor.
\end{align*}
This shows that \eqref{e4-7-a-even-ord} holds for $i$. Next, by \eqref{e4-7-a-odd-rec} and Lemma \ref{lem-e4-7-m-order}, we have
\begin{align*}
\pi_7(a(2i+1,j))&\geq \min\limits_{k\geq 0}\pi_7(a(2i,k))+\pi_7(m(4k-2,k+j)) \\
&\geq \min\limits_{k\geq 0} 2i+\left\lfloor \frac{7k+3}{4}\right\rfloor +\left\lfloor \frac{7j-k+5}{4}\right\rfloor \\
&\geq 2i+1+\left\lfloor \frac{7j+1}{4}\right\rfloor.
\end{align*}
Thus \eqref{e4-7-a-odd-ord} holds for $i+1$. By induction we complete the proof.
\end{proof}
\begin{theorem}\label{thm-e4-7-cong}
For $k\geq 1$ and $n\geq 0$, we have
\begin{align}
e_4(7^kn+\delta_{7,k})\equiv 0 \pmod{7^k}. \label{e4-7-cong}
\end{align}
\end{theorem}

\subsection{Congruences for $e_6(n)$ modulo powers of 7}
Since $\widetilde{\mathcal{E}}_{6,7} \in \mathcal{M}_6(\Gamma_0(7))$, it is not difficult to find that
\begin{align}
\widetilde{\mathcal{E}}_{6,7}=\frac{\eta(\tau)^{14}}{\eta(7\tau)^2}.
\end{align}
\begin{lemma}\label{lem-e6-7-UZ}
(Cf.\ \cite[Lemma 3.1]{Garvan-Aust}.) For any integers $i$,  we have
\begin{align}
Z_7^i\mid U_7=\sum_{j=\left\lceil \frac{2i}{7}\right\rceil}^\infty m(i,j)Y_7^j,
\end{align}
where $m(i,j)$ for $-6\leq i \leq 0$ are given by Group \uppercase\expandafter{\romannumeral2} of the Appendix, and for other $i$, $m(i,j)$ is determined by the recurrence \eqref{m-7-rec}.
\end{lemma}

\begin{lemma}\label{lem-e6-7-m-ord}
(Cf.\ \cite[Lemma 5.1]{Garvan-Aust}.) For any integers $i$ and $j$, we have
\begin{align}
\pi_7(m(i,j))\geq \left\lfloor \frac{7j-2i-1}{4}\right\rfloor.
\end{align}
\end{lemma}

\begin{theorem}\label{thm-e6-7-gen}
For $k\geq 1$, we have
\begin{align}
L_{6,7,k}(\tau)=-\frac{\eta(\tau)^{14}}{\eta(7\tau)^2}\sum_{j=1}^\infty a(k,j)Y_{7}^j, \label{e6-7-gen}
\end{align}
where
\begin{align}
&a(1,1)=343799\cdot 7, \quad a(1,2)=13424541\cdot 7^2, \quad a(1,3)=9999649\cdot 7^4, \nonumber \\
&a(1,4)=2823575\cdot 7^6, \quad a(1,5)=3\cdot 7^{14}, \quad a(1,6)=7^{15}, \quad a(1,j)=0, \quad \forall j\geq 7, \label{a-e6-7-initial}
\end{align}
and for $i\geq 1$,
\begin{align}
&a(2i,j)=\sum_{k=1}^\infty a(2i-1,k)m(4k-14,k+j-4), \label{e6-7-a-even-rec}\\
&a(2i+1,j)=\sum_{k=1}^\infty a(2i,k)m(4k-13,k+j-4). \label{e6-7-a-odd-rec}
\end{align}
\end{theorem}
\begin{proof}
We check that \eqref{e6-7-gen} holds for $k=1$ by direct computation. Using Lemma \ref{lem-e6-7-UZ} and the following facts:
\begin{align}
&\left(\frac{\eta(\tau)^{14}}{\eta(7\tau)^2}\cdot Y_7^j\right)\mid U_7=\frac{\eta(\tau)^{14}}{\eta(7\tau)^2}Y_7^{4-j}\left(Z_7^{4j-14} \mid U_7\right), \\
&\left(Z_7\frac{\eta(\tau)^{14}}{\eta(7\tau)^2}\cdot Y_7^j\right)\mid U_7=\frac{\eta(\tau)^{14}}{\eta(7\tau)^2}Y_7^{4-j}\left(Z_7^{4j-13} \mid U_7\right),
\end{align}
we can prove the theorem by induction on $k$.
\end{proof}
\begin{corollary}
We have
\begin{align}
e_6(7n+5)\equiv 0 \pmod{7} \label{e6-mod7-add-eq}
\end{align}
and
\begin{align}
e_6(49n+r)\equiv 0 \pmod{49}, \quad r\in \{19, 33, 40,47\}. \label{e6-mod49-add-eq}
\end{align}
\end{corollary}
\begin{proof}
Congruence \eqref{e6-mod7-add-eq} follows from \eqref{e6-7-gen} with $k=1$.

Using \eqref{e6-7-gen} with $k=1$ and the binomial theorem, we deduce that
\begin{align}
\sum_{n=0}^\infty e_6(7n+5)q^{n}&\equiv -7(q;q)_\infty^{10}(q^7;q^7)_\infty \pmod{49} \\
&\equiv -7(q;q)_\infty^3(q^7;q^7)_\infty^2 \pmod{49}. \label{e6-mod49-pf}
\end{align}
By Jacobi's identity, we have
$$(q;q)_\infty^3=\sum_{i=0}^\infty (-1)^i(2i+1)q^{i(i+1)/2}.$$
Note that the residue of $\frac{i(i+1)}{2}$ modulo 7 can only be 0, 1, 3 or 6. Moreover, $\frac{i(i+1)}{2}\equiv 6$ (mod 7) if and only if $i\equiv 3$ (mod 7), in which case $2i+1\equiv 0 $ (mod 7). Congruence \eqref{e6-mod49-add-eq} follows by comparing the coefficients of $q^{7n+r}$ ($r\in \{2,4,5,6\}$) on both sides of \eqref{e6-mod49-pf}.
\end{proof}

\begin{lemma}\label{lem-e6-7-a-ord}
For $i\geq 1$ and $j\geq 1$, we have
\begin{align}
&\pi_7(a(2i-1,j))\geq 4i-4+\left\lfloor \frac{7j-5}{4}\right\rfloor, \label{e6-7-a-odd-ord}\\
&\pi_7(a(2i,j))\geq 4i+\left\lfloor \frac{7j-4}{4}\right\rfloor -2\gamma_{j,2}-4\gamma_{j,3}, \label{e6-7-a-even-ord}
\end{align}
where $\gamma_{j,n}=1$ if $j=n$ and $\gamma_{j,n}=0$ otherwise.
\end{lemma}
\begin{proof}
From \eqref{a-e6-7-initial} we see that \eqref{e6-7-a-odd-ord} holds for $i=1$. Now suppose it holds for some $i\geq 1$. Then by \eqref{e6-7-a-even-rec}  we have
\begin{align}
\pi_7(a(2i,j))&\geq \min\limits_{k\geq 1} \pi_7(a(2i-1,k))+\pi_7(m(4k-14,k+j-4)). \label{e6-7-a-ord-pf-1}
\end{align}
When $k=1$, noting that $Z_7^{-10}\mid U_7=7^4$, we see that
\begin{align}\label{e6-a-ord-k=1}
\pi_7(a(2i-1,1))+\pi_7(m(-10,j-3))\geq \left\{\begin{array}{ll}
\infty & j\neq 3, \\
4i & j=3.
\end{array}\right.
\end{align}
When $k=2$, noting that $Z_7^{-6}\mid U_7=7^2$, we see that
\begin{align}\label{e6-a-ord-k=2}
\pi_7(a(2i-1,2))+\pi_7(m(-6,j-2))\geq \left\{\begin{array}{ll}
\infty & j\neq  2, \\
4i & j=2.
\end{array}\right.
\end{align}
When $k\geq 3$, by Lemma \ref{lem-e6-7-m-ord} we have
\begin{align}\label{e6-a-ord-other}
& \pi_7(a(2i-1,k))+\pi_7(m(4k-14,k+j-4)) \nonumber \\
\geq &~ 4i-4+\left\lfloor \frac{7k-5}{4}\right\rfloor +\left\lfloor \frac{7j-k-1}{4}\right\rfloor \nonumber  \\
\geq &~ 4i+\left\lfloor \frac{7j-4}{4}\right\rfloor.
\end{align}
Combining \eqref{e6-a-ord-k=1}, \eqref{e6-a-ord-k=2} and \eqref{e6-a-ord-other}, by \eqref{e6-7-a-ord-pf-1} we deduce that
\begin{align}
\pi_7(a(2i,j))\geq 4i+\left\lfloor \frac{7j-4}{4}\right\rfloor -2\gamma_{j,2}-4\gamma_{j,3}.
\end{align}
This shows that \eqref{e6-7-a-even-ord} holds for $i$.

Next, by \eqref{e6-7-a-odd-rec} and Lemma \ref{lem-e6-7-m-ord} we have
\begin{align}
\pi_7(a(2i+1,j))&\geq \min\limits_{k\geq 1} \pi_7(a(2i,k))+\pi_7(m(4k-13,k+j-4)).  \label{e6-7-add-eq-1}
\end{align}
If $k=3$, then
\begin{align*}
 \pi_7(a(2i,3))+\pi_7(m(-1,j-1))\geq \left\{\begin{array}{ll}
 \infty & j\neq 1, \\
 4i &j=1.
\end{array}\right.
\end{align*}
If $k\geq 1$ and $k\neq 3$, then
\begin{align*}
 &\pi_7(a(2i,k))+\pi_7(m(4k-13,k+j-4))\nonumber \\
 \geq  &~   4i+\left\lfloor \frac{7k-4}{4}\right\rfloor -2\gamma_{k,2}-4\gamma_{k,3}+\left\lfloor \frac{7j-k-3}{4}\right\rfloor  \\
\geq &~ 4i+\left\lfloor \frac{7j-5}{4}\right\rfloor.
\end{align*}
Combining these cases, from \eqref{e6-7-add-eq-1} we see that \eqref{e6-7-a-odd-ord} holds for $i+1$. By induction we complete the proof.
\end{proof}

\begin{theorem}\label{thm-e6-mod7-cong}
For $k\geq 1$ and $n\geq 0$, we have
\begin{align}
e_6(7^{2k}n+\delta_{7,2k})\equiv 0 \pmod{7^{2k}}.
\end{align}
\end{theorem}
\begin{proof}
This follows from Theorem \ref{thm-e6-7-gen} and Lemma \ref{lem-e6-7-a-ord}.
\end{proof}

\section{Generating functions and congruences for $e_{2r}({13}^kn+\delta_{13,k})$}\label{sec-13}
To establish congruences modulo powers of 13, we need to use the 13th-order modular equation
\begin{align}\label{modeq-13}
Z_{13}^{13}(\tau)=\sum_{r=1}^{13}\sum_{s=\left\lfloor \frac{1}{2}(r+2)\right\rfloor}^{7} \psi_{r,s}Y^s(13\tau)Z_{13}^{13-r}(\tau),
\end{align}
where the matrix $\Psi=(\psi(r,s))_{1\leq r\leq 13, 1\leq s\leq 7}$ is given in \cite[Eq.\ (3.52)]{Garvan-TAMS}.
This identity was derived by Lehner \cite{Lehner} and also used by Atkin and O'Brien \cite{Atkin-OBrien} to study properties of $p(n)$ modulo powers of 13.

\subsection{Congruences for $e_4(n)$ modulo powers of 13}

We follow Theorem \ref{thm-L-wE-repn} and  use the function
\begin{align}
\widetilde{E}_{4,13}(\tau)=\frac{E_4(\tau)}{1+247Y_{13}+3380Y_{13}^2+15379Y_{13}^3+28561Y_{13}^3}.
\end{align}

\begin{lemma}\label{lem-e4-13-UEZ}
For any integer $i$, we have
\begin{align}
\left(\widetilde{E}_{4,13}Z_{13}^i \right)\mid U_{13}=\widetilde{E}_{4,13}\sum_{j=\left\lfloor \frac{7i}{13}\right\rfloor} m(i,j)Y_{13}^{j},
\end{align}
where the values of $m(i,j)$ for $-4\leq i \leq 8$ are given in Group \uppercase\expandafter{\romannumeral2} of the Appendix, and for other $i$, $m(i,j)$ is determined by the recurrence relation
\begin{align}\label{e4-13-m-rec}
m(i,j)=\sum_{r=1}^{13}\sum_{s=\left\lfloor \frac{1}{2}(r+2)\right\rfloor}^7 \psi(r,s)m(i-r,j-s).
\end{align}
\end{lemma}

\begin{lemma}\label{lem-e4-13-m-ord}
For any integers $i$ and $j$, we have
\begin{align}
\pi_{13}(m(i,j))\geq \left\lfloor \frac{13j-7i+3}{14}\right\rfloor. \label{e4-m-13-ord}
\end{align}
\end{lemma}
\begin{proof}
From \cite{Atkin-OBrien} we know
\begin{align}\label{psi-13-ord}
\pi_{13}(\psi(r,s))\geq \left\lfloor \frac{13s-7r+13}{14}\right\rfloor.
\end{align}
For $-4\leq i \leq 8$, \eqref{e4-m-13-ord} can be verified directly. We use \eqref{psi-13-ord} and the recurrence \eqref{e4-13-m-rec} to prove the general result by induction.
\end{proof}

\begin{theorem}\label{thm-e4-13-gen}
For $k\geq 1$, we have
\begin{align}
L_{4,13,k}(\tau)=\widetilde{E}_{4,13}\sum_{j=1}^\infty a(k,j)Y_{13}^j, \label{e4-13-gen}
\end{align}
where
\begin{align}
&a(1,1)=158411, \quad a(1,2)=6539045\cdot 13, \quad a(1,3)=2054214\cdot {13}^3,  \nonumber \\
& a(1,4)=43926819\cdot {13}^3, \quad a(1,5)=1409\cdot {13}^8, \quad a(1,6)=817\cdot {13}^9, \nonumber \\
& a(1,7)=4122\cdot {13}^9, \quad a(1,8)=1085\cdot {13}^{10}, \quad a(1,9)=189\cdot {13}^{11}, \nonumber\\
& a(1,10)={13}^{13}, \quad a(1,j)=0, \quad \forall j\geq 11, \label{a-e4-13-initial}
\end{align}
and for $i\geq 1$,
\begin{align}
&a(2i,j)=\sum_{k=1}^\infty a(2i-1,k)m(2k,k+j), \label{e4-13-a-even-rec} \\
&a(2i+1,j)=\sum_{k=1}^\infty a(2i,k)m(2k+1,k+j). \label{e4-13-a-odd-rec}
\end{align}
\end{theorem}
\begin{proof}
We verify \eqref{e4-13-gen} by direct computation. The theorem then follows from Theorem \ref{thm-L-wE-repn}.
\end{proof}

\begin{lemma}\label{lem-e4-13-a-ord}
For $i\geq 1$ and $j\geq 1$, we have
\begin{align}
&\pi_{13}(a(2i-1,j))\geq 2i-2+\left\lfloor \frac{13j-7}{14}\right\rfloor, \label{e4-13-a-odd} \\
&\pi_{13}(a(2i,j))\geq 2i+\left\lfloor \frac{13j-13}{14} \right\rfloor-\gamma_{j,3}. \label{e4-13-a-even}
\end{align}
\end{lemma}
\begin{proof}
From \eqref{a-e4-13-initial} we see that \eqref{e4-13-a-odd} holds for $i=1$. Now suppose it holds for some $i\geq 1$. By \eqref{e4-13-a-even-rec}  we deduce that
\begin{align}
\pi_{13}(a(2i,j))&\geq \min\limits_{k\geq 1} \pi_{13}(a(2i-1,k))+\pi_{13}(m(2k,k+j)) \nonumber \\
&\geq \min\limits_{k\geq 1} 2i-2+\left\lfloor \frac{13k-7}{14}\right\rfloor + \pi_{13}(m(2k,k+j)). \label{e4-13-a-ord-pf-1}
\end{align}
When $k\geq 2$,  Lemma \ref{lem-e4-13-m-ord} and \eqref{e4-13-a-ord-pf-1} imply
\begin{align}
\pi_{13}(a(2i,j))\geq &~ 2i-2+\left\lfloor \frac{13k-7}{14}\right\rfloor +\left\lfloor \frac{13j-k+3}{14}\right\rfloor \nonumber \\
\geq &~ 2i+\left\lfloor \frac{13j-13}{14}\right\rfloor. \label{e4-13-a-ord-pf-2}
\end{align}
When $k=1$, note that $m(2,4)=-13^2$ and $m(2,1+j)=0$ when $j\neq 3$. We have
\begin{align}
2i-2+\left\lfloor \frac{13-7}{14}\right\rfloor + \pi_{13}(m(2,1+j))\geq
\left\{\begin{array}{ll}
2i & j=3, \\
\infty & j\neq 3.
\end{array}\right. \label{e4-13-a-ord-pf-3}
\end{align}
Combining \eqref{e4-13-a-ord-pf-1}, \eqref{e4-13-a-ord-pf-2} and \eqref{e4-13-a-ord-pf-3}, we see that
\begin{align}
\pi_{13}(a(2i,j))\geq 2i+\left\lfloor \frac{13j-13}{14} \right\rfloor-\gamma_{j,3}.
\end{align}
This shows that \eqref{e4-13-a-even} holds for $i$.

Next, by \eqref{e4-13-a-odd-rec} and Lemma \ref{lem-e4-13-m-ord} we deduce that
\begin{align}
\pi_{13}(a(2i+1,j))&\geq \min\limits_{k\geq 1} \pi_{13}(a(2i,k))+\pi_{13}(m(2k+1,k+j)) \nonumber \\
&\geq \min\limits_{k\geq 1} 2i+\left\lfloor \frac{13k-13}{14}\right\rfloor-\gamma_{k,3} + \left\lfloor \frac{13j-k-4}{14}\right\rfloor \nonumber \\
&\geq 2i+\left\lfloor \frac{13j-7}{14}\right\rfloor.
\end{align}
Therefore, \eqref{e4-13-a-odd} holds for $i+1$. By induction we complete our proof.
\end{proof}

\begin{theorem}\label{thm-e4-13-cong}
For $k\geq 1$ and $n\geq 0$, we have
\begin{align}
e_4(13^{2k}n+\delta_{13,2k})\equiv 0 \pmod{13^{2k}}.
\end{align}
\end{theorem}
\begin{proof}
This follows from Theorem \ref{thm-e4-13-gen} and Lemma \ref{lem-e4-13-a-ord}.
\end{proof}

\subsection{Congruences for $e_6(n)$ modulo powers of 13}
This time we need to use the function $\widetilde{E}_{6,13}(\tau)=\frac{E_6(\tau)}{P_{6,13}(Y_{13})}$ where as in the identity \eqref{E6-mod13-eq},
\begin{align*}
P_{6,13}(Y_{13})=&~ 1-38\cdot 13 Y_{13}-122\cdot 13^2 Y_{13}^2-108\cdot {13}^3Y_{13}^3-46\cdot {13}^4 Y_{13}^4\nonumber \\
&-10\cdot {13}^5Y_{13}^5-{13}^6Y_{13}^6.
\end{align*}

\begin{lemma}\label{lem-e6-13-UEZ}
For any integer $i$, we have
\begin{align}
\left(\widetilde{E}_{6,13}Z_{13}^i\right)\mid U_{13}=\widetilde{E}_{6,13}\sum_{j=\left\lceil \frac{7i}{13}\right\rceil}m(i,j)Y_{13}^j,
\end{align}
where for $0\leq i \leq 12$, $m(i,j)$ is given in Group \uppercase\expandafter{\romannumeral4} of the Appendix, and for other $i$, it is determined by the recurrence \eqref{e4-13-m-rec}.
\end{lemma}

\begin{lemma}\label{lem-e6-13-m-ord}
For any integers $i$ and $j$, we have
\begin{align}
\pi_{13}(m(i,j))\geq \left\lfloor \frac{13j-7i+5}{14}\right\rfloor.
\end{align}
\end{lemma}
This lemma can be proved in a way similar to the proof of Lemma \ref{lem-e4-13-m-ord}. We omit the details.
\begin{theorem}\label{thm-e6-13-gen}
For $k\geq 1$, we have
\begin{align}
L_{6,13,k}(\tau)=-\widetilde{E}_{6,13}\sum_{j=1}^\infty a(k,j)Y_{13}^j, \label{e6-13-gen}
\end{align}
where
\begin{align}
&a(1,1)=7154773, \quad a(1,2)=1554139318\cdot 13, \quad a(1,3)=12501650600\cdot {13}^2, \nonumber \\
&a(1,4)=2584942916\cdot {13}^4, \quad a(1,5)=3629017744\cdot {13}^5, \quad a(1,6)=41201641625 \cdot {13}^5, \nonumber \\
&a(1,7)=65636\cdot {13}^{11}, \quad a(1,8)=27416\cdot {13}^{12}, \quad a(1,9)=8208\cdot {13}^{13},  \nonumber \\
&a(1,10)=1746\cdot {13}^{14}, \quad a(1,11)=254\cdot {13}^{15}, \quad a(1,12)=23\cdot {13}^{16}, \nonumber \\
&a(1,13)={13}^{17}, \quad a(1,j)=0, \quad \forall j\geq 14, \label{a-e6-13-initial}
\end{align}
and $a(i,j)$ satisfies the recurrence relations in \eqref{e4-13-a-even-rec} and \eqref{e4-13-a-odd-rec}.
\end{theorem}

\begin{lemma}\label{lem-e6-13-a-ord}
For $i\geq 1$ and $j\geq 1$, we have
\begin{align}
&\pi_{13}(a(2i-1,j))\geq 4i-4+\left\lfloor \frac{13j-7}{14}\right\rfloor, \label{e6-13-a-odd-ord} \\
&\pi_{13}(a(2i,j))\geq 4i+\left\lfloor \frac{13j-12}{14}\right\rfloor-\gamma_{j,3}-2\gamma_{j,4}-3\gamma_{j,5}. \label{e6-13-a-even-ord}
\end{align}
\end{lemma}
\begin{proof}
From \eqref{a-e6-13-initial} we see that \eqref{e6-13-a-odd-ord} holds for $i=1$. Now suppose it holds for some $i\geq 1$. By \eqref{e4-13-a-even-rec}   we deduce that
\begin{align}
\pi_{13}(a(2i,j))&\geq \min\limits_{k\geq 1} \pi_{13}(a(2i-1,k))+\pi_{13}(m(2k,k+j)) \nonumber \\
&\geq \min\limits_{k\geq 1}  4i-4+\left\lfloor \frac{13k-7}{14}\right\rfloor  +\pi_{13}(m(2k,k+j)). \label{e6-13-pf-1}
\end{align}
When $k=1$, note that $m(2,6)={13}^4$ and $m(2,1+j)=0$ if $j\neq 5$. Thus
\begin{align}
4i-4+\left\lfloor \frac{13-7}{14}\right\rfloor  +\pi_{13}(m(2,1+j))\geq \left\{\begin{array}{ll}
4i & j=5, \\
\infty & j\neq 5.
\end{array}\right.
\end{align}
When $k=2$, note that $m(4,6)={13}^3$ and $m(4,2+j)=0$ if $j \neq 4$. Thus
\begin{align}
4i-4+\left\lfloor \frac{26-7}{14}\right\rfloor  +\pi_{13}(m(4,2+j))\geq \left\{\begin{array}{ll}
4i & j=4, \\
\infty & j\neq 4.
\end{array}\right.
\end{align}
When $k=3$, note that $m(6,6)={13}^2$ and $m(6,3+j)=0$ if $j\neq 3$. Thus
\begin{align}
4i-4+\left\lfloor \frac{39-7}{14}\right\rfloor  +\pi_{13}(m(6,3+j))\geq \left\{\begin{array}{ll}
4i & j=3, \\
\infty & j\neq 3.
\end{array}\right.
\end{align}
When $k\geq 4$, by Lemma \ref{lem-e6-13-m-ord}  we see that
\begin{align}
&4i-4+\left\lfloor \frac{13k-7}{14}\right\rfloor  +\pi_{13}(m(2k,k+j)) \nonumber \\
\geq &~ 4i-4+\left\lfloor \frac{13k-7}{14}\right\rfloor + \left\lfloor \frac{13j-k+5}{14}\right\rfloor \nonumber \\
=&~ 4i+\left\lfloor \frac{13j-13}{14}\right\rfloor.
\end{align}
Combining the above cases, we see that
\begin{align}
\pi_{13}(a(2i,j))\geq 4i+\left\lfloor \frac{13j-13}{14}\right\rfloor-\gamma_{j,3}-2\gamma_{j,4}-3\gamma_{j,5}.
\end{align}
This shows that \eqref{e6-13-a-even-ord} holds for $i$.

Next, by \eqref{e4-13-a-odd-rec} and Lemma \ref{lem-e6-13-m-ord} we deduce that
\begin{align}
&\pi_{13}(a(2i+1,j)) \nonumber \\
\geq &~ \min\limits_{k\geq 1} \pi_{13}(a(2i,k))+\pi_{13}(m(2k+1,k+j)) \nonumber \\
\geq &~ \min\limits_{k\geq 1}  4i+\left\lfloor \frac{13k-13}{14}\right\rfloor +\left\lfloor \frac{13j-k-2}{14} \right\rfloor -\gamma_{k,3}-2\gamma_{k,4}-3\gamma_{k,5} \nonumber \\
=&~ 4i+\left\lfloor \frac{13j-7}{14} \right\rfloor.
\end{align}
Therefore, \eqref{e6-13-a-odd-ord} holds for $i+1$. By induction we complete the proof.
\end{proof}

\begin{theorem}\label{thm-e6-13-cong}
For $k\geq 1$ and $n\geq 0$, we have
\begin{align}
e_6(13^{2k}n+\delta_{13,2k})\equiv 0 \pmod{{13}^{2k}}.
\end{align}
\end{theorem}
\begin{proof}
This follows from Theorem \ref{thm-e6-13-gen} and Lemma \ref{lem-e6-13-a-ord}.
\end{proof}

\section{Generating functions and congruences for $e_{2r}({11}^kn+\delta_{11,k})$}\label{sec-11}
By definition, it is not difficult to see that for any $k\geq 1$,
\begin{align}
\frac{L_{2r,11,k}(\tau)}{\eta(\tau)^{2r}\eta(11\tau)^{2r}} \in \mathcal{M}_0^{!}(\Gamma_0(11)). \label{L-11-V}
\end{align}
Thus to give representations for $L_{2r,11,k}$, it suffices to find suitable basis for $\mathcal{M}_0^{!}(\Gamma_0(11))$.  Atkin \cite{Atkin} has constructed a nice basis for $\mathcal{M}_0^{!}(\Gamma_0(11))$, and Gordon \cite{Gordon}  rephrased them as follows.
\begin{defn}\label{defn-J}
Let $g_k(\tau)$ and $G_k(\tau)$ be defined as in \cite{Atkin}. We define $J_k=J_k(\tau)$ as follows: \\
$\mathrm{(1)}$ $J_{0}(\tau)=1$; \\
$\mathrm{(2)}$ For $k\ge 1$,
\begin{equation}\label{J-g-exchange}
J_{k}(\tau)=\left\{\begin{array}{ll}
g_{k}(\tau) & \textrm{if $k\equiv 0$ (mod 5)},\\
g_{k+2}(\tau) & \textrm{if $k\equiv 4$ (mod 5)}, \\
g_{k+1}(\tau) & \textrm{otherwise};
\end{array}\right.
\end{equation}
$\mathrm{(3)}$ For $k\le -2$, $J_{k}(\tau)=G_{-k}(\tau)$ ; \\
$\mathrm{(4)}$  $J_{-1}(\tau)=J_{-6}(\tau)J_{5}(\tau)$.
\end{defn}
For explicit expressions of $J_{k}(\tau)$ ($-6\le k \le 5$), see \cite[Appendix A]{Atkin}.

We need the following lemma on the properties of the functions $J_k$.
\begin{lemma}\label{J-prop}
(Cf. \cite[Lemma 3]{Gordon}.) For all $k\in \mathbb{Z}$, we have\\
$\mathrm{(1)}$ $J_{k+5}(\tau)=J_{k}(\tau)J_{5}(\tau)$,\\
$\mathrm{(2)}$ $\{J_{k}(\tau): k \in \mathbb{Z}\}$ is a basis of $V$, \\
$\mathrm{(3)}$ $\mathrm{ord}_{\infty}J_{k}(\tau)=k$,\\
$\mathrm{(4)}$ $\mathrm{ord}_{0}J_{k}(\tau)=\left\{\begin{array}{ll}
-k & \textrm{if $k\equiv 0 \pmod{5}$},\\
-k-1 & \textrm{if $k\equiv 1,2$ or $3 \pmod{5}$}, \\
-k-2 &\textrm{if $k \equiv 4 \pmod{5}$},
\end{array} \right. $\\
$\mathrm{(5)}$ The Fourier series of $J_{k}(\tau)$ has integer coefficients, and is of the form $J_{k}(\tau)=q^{k}+$ higher degree terms.
\end{lemma}
For any integer $\lambda$,  from \cite{Gordon} we know that $\mathcal{M}_0^{!}(\Gamma_0(11))$ is mapped into itself by the linear transformation
\[\widetilde{T}_{\lambda}: g(\tau) \rightarrow \left(Z_{11}(\tau)^{\lambda}g(\tau)\right) \mid U_{11}. \]
Following Atkin, we write the elements of $\mathcal{M}_0^{!}(\Gamma_0(11))$ as row vectors and let matrices act on the right. Let $C^{(\lambda)}=(c_{\mu,\nu}^{(\lambda)})$ be the matrix of $\widetilde{T}_{\lambda}$ with respect to the basis $\{J_{k}\}$ of $\mathcal{M}_0^{!}(\Gamma_0(11))$. We have
\begin{equation}\label{T-map}
\left(Z_{11}(\tau)^{\lambda}J_{\mu}(\tau))\right)\mid U_{11}=\sum\limits_{\nu \in \mathbb{Z}}c_{\mu,\nu}^{(\lambda)}J_{\nu}(\tau).
\end{equation}

\begin{theorem}\label{thm-L-wE-repn-11}
There exist integers $a(k,j)$ such that
\begin{align}
L_{2r, 11, k}(\tau) = \eta(\tau)^{2r}\eta(11\tau)^{2r} \sum_{j=1-r}^\infty a(k,j)J_{j}. \label{L-wE-repn-11}
\end{align}
Moreover, for $i\geq 1$,
\begin{align}
a(2i,\nu)=\sum_{j=1-r}^\infty a(2i-1,j)c_{j,\nu}^{(-2r)}, \label{e2r-11-a-even-rec} \\
a(2i+1,\nu)=\sum_{j=1-r}^\infty a(2i,j)c_{j,\nu}^{(1-2r)}. \label{e2r-11-a-odd-rec}
\end{align}
\end{theorem}
\begin{proof}
From \eqref{L-11-V} we know \eqref{L-wE-repn-11} holds.

By \eqref{L-even-defn} and \eqref{T-map} we have
\begin{align*}
L_{2r,11,2i}(\tau)&=\sum_{j=1-r}^\infty a(2i-1,j)\left(\eta(\tau)^{2r}\eta(11\tau)^{2r}J_{j}(\tau) \right)\mid U_{11} \nonumber \\
&=\eta(\tau)^{2r}\eta(11\tau)^{2r}\sum_{j=1-r}^\infty a(2i-1,j)\left(Z_{11}^{-2r}J_{j}(\tau) \right)\mid U_{11} \nonumber \\
&=\eta(\tau)^{2r}\eta(11\tau)^{2r}\sum_{j=1-r}^\infty a(2i-1,j)\sum_{\nu \in \mathbb{Z}}c_{j,\nu}^{(-2r)}J_{\nu}(\tau).
\end{align*}
This proves \eqref{e2r-11-a-even-rec}.

Similarly, by \eqref{L-odd-defn} and \eqref{T-map} we deduce that
\begin{align*}
L_{2r,11,2i+1}(\tau)&=\sum_{j=1-r}^\infty a(2i,j)\left(\eta(\tau)^{2r}\eta(11\tau)^{2r}Z_{11}J_{j}(\tau) \right)\mid U_{11} \nonumber \\
&=\eta(\tau)^{2r}\eta(11\tau)^{2r}\sum_{j=1-r}^\infty a(2i,j)\left(Z_{11}^{1-2r}J_{j}(\tau) \right)\mid U_{11} \nonumber \\
&=\eta(\tau)^{2r}\eta(11\tau)^{2r}\sum_{j=1-r}^\infty a(2i,j)\sum_{\nu \in \mathbb{Z}}c_{j,\nu}^{(1-2r)}J_{\nu}(\tau).
\end{align*}
This proves \eqref{e2r-11-a-odd-rec}.
\end{proof}

To examine the 11-adic orders of $a(k,j)$, we need the 11-adic properties of $c_{\mu,\nu}^{(\lambda)}$, which has been studied by Gordon \cite{Gordon}.
\begin{lemma}\label{lem-11-c-ord}
(Cf.\ \cite[Eq.\ (17)]{Gordon}.) We have
\begin{equation}\label{c-order}
\pi_{11}(c_{\mu,\nu}^{(\lambda)}) \geq \left\lfloor \frac{11\nu-\mu-5\lambda+s}{10}\right\rfloor,
\end{equation}
where $s=s(\mu,\nu)$ depends on the residues of $\mu$ and $\nu$ \text{\rm{(mod 5)}} according to Table \ref{Table-delta}.
\end{lemma}
\begin{table}[h]
\caption{}\label{Table-delta}
\begin{tabular}{|c|ccccc|}
\hline
\diagbox{$\mu$}{$\nu$} & 0 & 1 & 2 & 3 & 4 \\
\hline
0 & -1 & 8 & 7 &6 &15\\
1  &0 & 9 & 8 & 2 &11 \\
2 & 1 &10 &4 & 3 &12 \\
3& 2 &6 &5 &4 &13\\
4&3 &7 &6 &5 &9\\
\hline
\end{tabular}
\end{table}
In the following subsections, we will use Theorem \ref{thm-L-wE-repn-11} to give explicit representations for $L_{2r,11,k}(\tau)$ and prove congruences for $e_{2r}(n)$ modulo powers of 11. We only take $r=2$ and 3 as examples.

\subsection{Congruences for $e_4(n)$ modulo powers of 11}
\begin{theorem}\label{thm-e4-11-gen}
For any $k\geq 1$, we have
\begin{align}
L_{4,11,k}(\tau)=\eta(\tau)^4\eta(11\tau)^4\sum_{j=-1}^\infty a(k,j)J_{j}, \label{e4-11-gen}
\end{align}
where
\begin{align}
&a(1,-1)=14401 \cdot 11, \quad a(1,0)=3629544\cdot 11, \quad a(1,1)=1200563\cdot {11}^3,  \nonumber \\
&a(1,2)=104\cdot 11^8, \quad a(1,3)=694\cdot {11}^8, \quad a(1,4)=13\cdot {11}^{10},  \nonumber \\
&a(1,5)=56\cdot {11}^{10}, \quad a(1,6)={11}^{12}, \quad a(1,j)=0, \quad \forall j\geq 7, \label{a-e4-11-initial}
\end{align}
and $a(k,j)$ satisfies \eqref{e2r-11-a-even-rec} and \eqref{e2r-11-a-odd-rec} with $r=2$.
\end{theorem}
For $k=1$, \eqref{e4-11-gen} can be verified by direct computations. The  theorem then follows from Theorem \ref{thm-L-wE-repn-11} with $r=2$.

\begin{lemma}\label{lem-e4-11-a-ord}
For $i\geq 1$ we have
\begin{align}
&\pi_{11}(a(2i-1,j))\geq 2i-1+\gamma_{j,-1}+\left\lfloor \frac{11j+4}{10}\right\rfloor, \label{e4-11-a-odd-ord} \\
&\pi_{11}(a(2i,j))\geq 2i+\gamma_{j,-1}+\left\lfloor \frac{11j+9}{10}\right\rfloor. \label{e4-11-a-even-ord}
\end{align}
\end{lemma}
\begin{proof}
From  \eqref{a-e4-11-initial} we see that \eqref{e4-11-a-odd-ord} holds for $i=1$.  Now suppose it holds for some $i\geq 1$. By \eqref{e2r-11-a-even-rec} with $r=2$ and Lemma \ref{lem-11-c-ord} we deduce that
\begin{align}
\pi_{11}(a(2i,\nu))&\geq \min\limits_{j\geq -1} \pi_{11}(a(2i-1,j))+\pi_{11}(c_{j,\nu}^{(-4)})  \nonumber  \\
&\geq \min\limits_{j\geq -1} 2i-1+\gamma_{j,-1}+\left\lfloor \frac{11j+4}{10}\right\rfloor +\left\lfloor \frac{11\nu-j+20+s(j,\nu)}{10} \right\rfloor. \label{e4-11-a-ord-pf-1}
\end{align}
From Table \ref{Table-delta} we see that $s(j,\nu)\geq -1$. Hence
\begin{align}
\pi_{11}(a(2i,\nu))&\geq \min\limits_{j\geq -1} 2i-1+\gamma_{j,-1}+\left\lfloor \frac{11j+4}{10}\right\rfloor +\left\lfloor \frac{11\nu-j+19}{10} \right\rfloor\nonumber \\
 &\geq \min\limits_{j\geq -1} \left\{2i-1+\left\lfloor \frac{11\nu +20}{10}\right\rfloor, 2i-1+\left\lfloor \frac{11\nu+19}{10}\right\rfloor \right\} \nonumber \\
&=2i+\left\lfloor \frac{11\nu+9}{10}\right\rfloor. \label{e4-11-a-ord-pf-2}
\end{align}
When $\nu=-1$, from Table \ref{Table-delta} we see that $s(j,-1)\geq 9$. Therefore, from \eqref{e4-11-a-ord-pf-1} we have
\begin{align}
\pi_{11}(a(2i,-1))&\geq \min\limits_{j\geq -1} 2i-1+\gamma_{j,-1}+\left\lfloor \frac{11j+4}{10}\right\rfloor +\left\lfloor \frac{11(-1)-j+29}{10}\right\rfloor \nonumber \\
&=2i. \label{e4-11-a-ord-pf-3}
\end{align}
Combining \eqref{e4-11-a-ord-pf-2} with \eqref{e4-11-a-ord-pf-3}, we deduce that
\begin{align}
\pi_{11}(a(2i,j))\geq 2i+\gamma_{j,-1}+\left\lfloor \frac{11j+9}{10}\right\rfloor.
\end{align}
Thus \eqref{e4-11-a-even-ord} holds for $i$.

Next, by \eqref{e2r-11-a-odd-rec} with $r=2$, Lemma \ref{lem-11-c-ord} and the fact that $s(j,\nu)\geq -1$, we have
\begin{align}
\pi_{11}(a(2i+1,\nu))&\geq \min\limits_{j\geq -1}  \pi_{11}(a(2i,j))+\pi_{11}(c_{j,\nu}^{(-3)}) \nonumber \\
&\geq \min\limits_{j\geq -1} 2i+\gamma_{j,-1}+\left\lfloor \frac{11j+9}{10}\right\rfloor +\left\lfloor \frac{11\nu-j+15+s(j,\nu)}{10}\right\rfloor \label{e4-11-a-ord-pf-4} \\
&\geq \min\limits_{j\geq -1} 2i+\gamma_{j,-1}+\left\lfloor \frac{11j+9}{10}\right\rfloor +\left\lfloor \frac{11\nu-j+14}{10}\right\rfloor \nonumber \\
&=\min \left\{2i+\left\lfloor \frac{11\nu+15}{10}\right\rfloor, 2i+1+\left\lfloor \frac{11\nu+4}{10}\right\rfloor  \right\} \nonumber \\
&=2i+1+\left\lfloor \frac{11\nu+4}{10}\right\rfloor. \label{e4-11-a-ord-pf-5}
\end{align}
When $\nu=-1$, since $s(j,-1)\geq 9$,  \eqref{e4-11-a-ord-pf-4} implies
\begin{align}
\pi_{11}(a(2i+1,-1)&\geq \min\limits_{j\geq -1} 2i+\gamma_{j,-1}+\left\lfloor \frac{11j+9}{10}\right\rfloor +\left\lfloor \frac{-11-j+24}{10}\right\rfloor \nonumber \\
& =2i+1. \label{e4-11-a-ord-pf-6}
\end{align}
Combining \eqref{e4-11-a-ord-pf-5} with \eqref{e4-11-a-ord-pf-6}, we see that \eqref{e4-11-a-odd-ord} holds for $i+1$. By induction we complete the proof.
\end{proof}

\begin{theorem}\label{thm-e4-11-cong}
For $k\geq 1$ and $n\geq 0$, we have
\begin{align}
e_4(11^kn+\delta_{11,k})\equiv 0 \pmod{{11}^k}.
\end{align}
\end{theorem}
\begin{proof}
This follows from Theorem \ref{thm-e4-11-gen} and Lemma \ref{lem-e4-11-a-ord}.
\end{proof}
\subsection{Congruences for $e_6(n)$ modulo powers of 11}
\begin{theorem}\label{thm-e6-11-gen}
For $k\geq 1$ we have
\begin{align}
L_{6,11,k}(\tau)=-\eta(\tau)^6\eta(11\tau)^6\sum_{j=-2}^\infty a(k,j)J_{j},
\end{align}
where
\begin{align}
&a(1,-2)=7154773, \quad a(1,-1)=69109531\cdot {11}^2, \quad a(1,0)=4623886162\cdot {11}^2, \nonumber \\
&a(1,1)=1217062918\cdot {11}^4, \quad a(1,2)=1394057457\cdot {11}^5, \quad a(1,3)=5523\cdot {11}^{11}, \nonumber \\
&a(1,4)=135\cdot {11}^{13}, \quad a(1,5)=546\cdot {11}^{13}, \quad a(1,6)=20\cdot {11}^{15}, \nonumber \\
&a(1,7)={11}^{16}, \quad a(1,j)=0, \quad \forall j\geq 8, \label{a-e6-11-initial}
\end{align}
and $a(k,j)$ satisfies \eqref{e2r-11-a-even-rec} and \eqref{e2r-11-a-odd-rec} with $r=3$.
\end{theorem}

\begin{lemma}\label{lem-e6-11-a-ord}
For $i\geq 1$ we have
\begin{align}
&\pi_{11}(a(2i-1,j))\geq 4i-2+\gamma_{j,-1}+\gamma_{j,1}+\left\lfloor \frac{11j+3}{10} \right\rfloor, \label{e6-11-a-odd-ord}\\
&\pi_{11}(a(2i,j))\geq 4i+4\gamma_{j,-2}+2\gamma_{j,-1}+\gamma_{j,0}+\left\lfloor \frac{11j-9}{10} \right\rfloor. \label{e6-11-a-even-ord}
\end{align}
\end{lemma}
\begin{proof}
When $i=1$, \eqref{e6-11-a-odd-ord} can be verified directly. Now we suppose it holds for some $i\geq 1$. By \eqref{e2r-11-a-even-rec} with $r=3$, Lemma \ref{lem-11-c-ord} and the fact that $s(j,\nu)\geq -1$, we deduce that
\begin{align}
&\quad \pi_{11}(a(2i,\nu))\nonumber \\
&\geq \min\limits_{j\geq -2} \pi_{11}(a(2i-1,j))+\pi_{11}(c_{j,\nu}^{(-6)})  \label{e6-11-a-ord-pf-0}  \\
&\geq \min\limits_{j\geq -2} 4i-2+\gamma_{j,-1}+\gamma_{j,1}+\left\lfloor \frac{11j+3}{10}\right\rfloor +\left\lfloor \frac{11\nu-j+30+s(j,\nu)}{10} \right\rfloor \label{e6-11-a-ord-pf-1} \\
&\geq \min\limits_{j\geq -2} 4i-2+\gamma_{j,-1}+\gamma_{j,1}+\left\lfloor \frac{11j+3}{10}\right\rfloor +\left\lfloor \frac{11\nu-j+29}{10} \right\rfloor  \label{e6-11-a-ord-add} \\
&\geq 4i-2-2+\left\lfloor \frac{11\nu+31}{10}\right\rfloor \nonumber  \\
&=4i+\left\lfloor \frac{11\nu-9}{10}\right\rfloor. \label{e6-11-a-ord-pf-2}
\end{align}
Clearly, when $\nu\geq 1$, the left side is at least $4i$. For $\nu\leq 0$, we need more arguments.

We denote
\begin{align}
t_{i,\nu}(j):=4i-2+\gamma_{j,-1}+\gamma_{j,1}+\left\lfloor \frac{11j+3}{10}\right\rfloor + \pi_{11}(c_{j,\nu}^{(-6)}).
\end{align}
By direct computation, we find that
\begin{align}\label{e6-11-UZ-data}
\left(Z_{11}^{-6}J_{-2} \right)\mid U_{11}=11^4J_{-1}.
\end{align}

When $\nu=0$, for $j\geq -1$, from \eqref{e6-11-a-ord-add} it is easy to see that $t_{i,0}(j)\geq 4i$. When $j=-2$, by \eqref{e6-11-UZ-data} we also deduce that $t_{i,0}(j)\geq 4i$. Thus $\pi_{11}(a(2i,0))\geq 4i$.

When $\nu=-1$, since $s(j,-1)\geq 9$, for $j\geq -1$ we have
\begin{align}
t_{i,-1}(j)\geq 4i-2+\gamma_{j,1}+\gamma_{j,-1}+\left\lfloor \frac{11j+3}{10}\right\rfloor +\left\lfloor \frac{-11-j+39}{10} \right\rfloor \geq 4i.
\end{align}
For $j=-2$, by \eqref{e6-11-UZ-data} we also deduce that $t_{i,-1}(j)\geq 4i$. Thus $\pi_{11}(a(2i,-1))\geq 4i$.

In the same way, we can show that $\pi_{11}(a(2i,-2))\geq 4i$.

Combining the above cases, we conclude that
\begin{align}
\pi_{11}(a(2i,j))\geq 4i+4\gamma_{j,-2}+2\gamma_{j,-1}+\gamma_{j,0}+\left\lfloor \frac{11j-9}{10} \right\rfloor.
\end{align}
This shows that \eqref{e6-11-a-even-ord} holds for $i$.

Next, by \eqref{e2r-11-a-odd-rec} with $r=3$, Lemma \ref{lem-11-c-ord} and the fact that $s(j,\nu)\geq -1$, we deduce that
\begin{align}
&\pi_{11}(a(2i+1,\nu))\nonumber \\
\geq &~ \min\limits_{j\geq -2} \pi_{11}(a(2i,j))+\pi_{11}(c_{j,\nu}^{(-5)})  \label{e6-11-a-ord-odd-pf-1}  \\
\geq &~ \min\limits_{j\geq -2} 4i+4\gamma_{j,-2}+2\gamma_{j,-1}+\gamma_{j,0}+\left\lfloor \frac{11j-9}{10}\right\rfloor +\left\lfloor \frac{11\nu-j+25+s(j,\nu)}{10} \right\rfloor \label{e6-11-a-ord-odd-pf-2} \\
\geq &~ \min\limits_{j\geq -2} 4i+4\gamma_{j,-2}+2\gamma_{j,-1}+\gamma_{j,0}+\left\lfloor \frac{11j-9}{10}\right\rfloor +\left\lfloor \frac{11\nu-j+24}{10} \right\rfloor  \label{e6-11-a-ord-odd-pf-3} \nonumber \\
\geq &~ 4i+2+\left\lfloor \frac{11\nu+3}{10}\right\rfloor. \nonumber
\end{align}
We still need to treat the cases $\nu=-1$ and $\nu=1$ separately.
Let
\begin{align}
\rho_{i,\nu}(j)= 4i+4\gamma_{j,-2}+2\gamma_{j,-1}+\gamma_{j,0}+\left\lfloor \frac{11j-9}{10}\right\rfloor+\pi_{11}(c_{j,\nu}^{(-5)}).
\end{align}

When $\nu=-1$, from \eqref{e6-11-a-ord-odd-pf-2} and the fact that $s(j,-1)\geq 9$ we see that
\begin{align}
\rho_{i,-1}(j)\geq  4i+4\gamma_{j,-2}+2\gamma_{j,-1}+\gamma_{j,0}+\left\lfloor \frac{11j-9}{10}\right\rfloor +\left\lfloor \frac{-11-j+34}{10} \right\rfloor \geq 4i+2.
\end{align}
When $\nu=1$, from Table \ref{Table-delta} we see that $s(j,1)\geq 6$. Hence from \eqref{e6-11-a-ord-odd-pf-2} we deduce that
\begin{align}
\rho_{i,1}(j)\geq  4i+4\gamma_{j,-2}+2\gamma_{j,-1}+\gamma_{j,0}+\left\lfloor \frac{11j-9}{10}\right\rfloor +\left\lfloor \frac{11-j+31}{10} \right\rfloor \geq 4i+4.
\end{align}

Combining the above cases, we conclude that \eqref{e6-11-a-odd-ord} holds for $i+1$. By induction we complete our proof.
\end{proof}

From Theorem \ref{thm-e6-11-gen} and Lemma \ref{lem-e6-11-a-ord} we  get the following result.
\begin{theorem}\label{thm-e6-11-cong}
For $k\geq 1$ and $n\geq 0$, we have
\begin{align}
e_6(11^{2k}n+\delta_{11,2k})\equiv 0 \pmod{{11}^{2k}}.
\end{align}
\end{theorem}

\section{Moments of ranks and cranks}\label{sec-moment}
In this section, we return to the arithmetic properties of functions related to the moments of ranks and cranks. The results established for $e_{2r}(n)$ in previous sections will be employed.

\subsection{The second moments of ranks and cranks}
Recall that Dyson proved that $M_2(n)=2np(n)$ (see \eqref{M2-repn}). So the arithmetic properties of $M_2(n)$ is quite clear to us. Now we consider the second moment of ranks.

Recall in \eqref{spt-N2} we have
\begin{align}
N_2(n)=2np(n)-2\spt(n).
\end{align}
Therefore, the arithmetic properties of $\spt(n)$ will help us to understand $N_2(n)$.

Using modular equations of orders 5, 7 and 13, Garvan \cite{Garvan-TAMS} proved some congruences analogous to those of $p(n)$ for $\spt(n)$. Namely, he proved that for $j \geq 3$,
\begin{align}
\spt(5^j n+\delta_{5,j})+5\spt(5^{j-2}n+\delta_{5,j-2}) &\equiv 0 \pmod{5^{2j-3}}, \label{spt-5power}\\
\spt(7^j n+\delta_{7,j})+7\spt(7^{j-2}n+\delta_{7,j-2}) &\equiv 0 \pmod{7^{\lfloor (3j-2)/2\rfloor}}, \label{spt-7power} \\
\spt(13^jn+\delta_{13,j}) -13\spt(13^{j-2}n+\delta_{13,j-2}) &\equiv 0 \pmod{13^{j-1}}. \label{spt-13power}
\end{align}
These congruences together with \eqref{spt-mod5}--\eqref{spt-mod13} imply for $j\geq 1$ and $\ell \in \{5,7,13\}$,
\begin{align}
\spt(\ell^jn+\delta_{\ell,j}) \equiv 0 \pmod{\ell^{\lfloor (j+1)/2\rfloor}}. \label{spt-general}
\end{align}
If we combine this congruence with Ramanujan's congruences \eqref{pn-mod5} and \eqref{pn-mod7}, we  get the following congruences for $N_2(n)$.
\begin{theorem}\label{thm-N2-cong}
For $j\geq 1$ and $n\geq 0$,
\begin{align}
N_2(5^jn+\delta_{5,j}) &\equiv 0 \pmod{5^{\lfloor(j+1)/2\rfloor}}, \label{N2-5power} \\
N_2(7^jn+\delta_{7,j}) &\equiv 0 \pmod{7^{\lfloor(j+1)/2\rfloor}}. \label{N2-7power}
\end{align}
\end{theorem}

For the symmetrized  moments of ranks and cranks, by \eqref{eta-defn} and \eqref{mu-defn} we immediately see that
\begin{align}
\eta_2(n)=\frac{1}{2}N_2(n), \quad \mu_2(n)=\frac{1}{2}M_2(n).
\end{align}
So the congruences for them are inherited from those of $M_2(n)$ and $N_2(n)$.

\subsection{Quasi-modular forms and representations of $M_k(n)$ and $N_k(n)$}
To study $k$-th moments for $k\geq 4$, we need the theory of quasi-modular forms, which was first systematically investigated by Kaneko and Zagier \cite{K-Zagier}.

Let $\mathcal{M}$ denote the space of  modular forms on $\SL(2,\mathbb{Z})$. It is well known that
\begin{align}\label{M-E246}
\mathcal{M}=\sum_{k=0}^\infty \mathcal{M}_k =\mathbb{C}[E_4,E_6].
\end{align}
We say that $f$ is a quasi-modular form if it is in the algebra generated by $E_2$ and $\mathcal{M}$.  Let $\widetilde{\mathcal{M}}_{2n}$ be the space of quasi-modular forms of weight $\leq 2n$. Then
\begin{align}
\widetilde{\mathcal{M}}_{2n}=\left\{\sum_{j=0}^n f_j E_2^j: f_j\in \sum_{k=0}^{n-j}\mathcal{M}_{2k}  \right\}. \label{quasi}
\end{align}
Let $\widetilde{\mathcal{M}}$ be the space of quasi-modular forms. From \eqref{M-E246} we see that
\begin{align}
\widetilde{\mathcal{M}}=\mathbb{C}[E_2,E_4,E_6]. \label{quasi-algebra}
\end{align}
It is known that $E_2, E_4$ and $E_6$ are algebraically independent over $\mathbb{C}$ (see \cite[Lemma 117]{Martin-Royer}, for example).  Hence we can give a basis for $\widetilde{\mathcal{M}}_{2n}$ using $E_2, E_4$ and $E_6$.
\begin{lemma}\label{lem-quasi}
Let $n$ be a nonnegative integer. The set
\begin{align}
\{E_2^aE_4^bE_6^c: a+2b+3c\leq n \, \textrm{with a,b,c nonnegative integers}\}
\end{align}
is a basis for $\widetilde{\mathcal{M}}_{2n}$.
\end{lemma}
This lemma was also proved in \cite{Atkin-Garvan}. It gives a way to calculate the dimension of $\widetilde{\mathcal{M}}_{2n}$. Below in Table \ref{table-dim} we list the dimensions of $\widetilde{\mathcal{M}}_{2n}$ for $1\leq n \leq 10$, which were calculated by Atkin and Garvan. See the table on page 356 of \cite{Atkin-Garvan}.
\begin{table}[h]
\caption{}\label{table-dim}
\begin{tabular}{ccccccccccc}
  \hline
  $n$ & 1 & 2 & 3 & 4 & 5 & 6 & 7 & 8 & 9 & 10 \\
  \hline
  $\dim \widetilde{\mathcal{M}}_{2n}$ & 2 & 4 & 7 & 11 & 16 & 23 & 31 & 41 & 53 & 67 \\
  \hline
\end{tabular}
\end{table}

Let $\Theta=q\frac{d}{dq}$ and
\begin{align}
P=P(q):=\sum_{n=0}^\infty p(n)q^n=\frac{1}{(q;q)_\infty}.  \label{P-defn}
\end{align}
Ramanujan \cite[p.\ 165]{Ramanujan} found the following identities:
\begin{align}
&\Theta(E_2)=\frac{E_2^2-E_4}{12}, \label{Theta-E2} \\
&\Theta(E_4)=\frac{E_2E_4-E_6}{3}, \label{Theta-E4}\\
&\Theta(E_6)=\frac{E_2E_6-E_4^2}{2}. \label{Theta-E6}
\end{align}
\begin{lemma}\label{lem-Theta}
If $f\in \widetilde{\mathcal{M}}_{2n}$, then $\Theta(f)\in \widetilde{\mathcal{M}}_{2n+2}$ and $\Theta(Pf)\in P\widetilde{\mathcal{M}}_{2n+2}$.
\end{lemma}
\begin{proof}
The first assertion follows from \eqref{quasi-algebra} and Ramanujan's identities \eqref{Theta-E2}--\eqref{Theta-E6}.

Taking logarithmic differentiation on both sides of \eqref{P-defn}, we deduce that
\begin{align}
\Theta(P)=P\sum_{n=1}^\infty \frac{nq^n}{1-q^n}=\frac{1}{24}P(1-E_2).
\end{align}
Therefore, $\Theta(P)\in P\widetilde{\mathcal{M}}_2$.

Note that
\begin{align}
\Theta(Pf)=\Theta(P) f +P\Theta(f). \label{Theta-Pf}
\end{align}
Since $\Theta(f) \in \widetilde{\mathcal{M}}_{2n+2}$ and $\Theta(P)\in P\widetilde{\mathcal{M}}_2$. From \eqref{Theta-Pf} we deduce that $\Theta(Pf)\in P\widetilde{\mathcal{M}}_{2n+2}$.
\end{proof}

By letting $f=E_{2r}$ in the above lemma, we get the following fact.
\begin{corollary}\label{cor-E-quasi}
For $r\geq 0$ and $k\geq 0$, we have
\begin{align}
\Theta^{k}(PE_{2r})\in P\widetilde{\mathcal{M}}_{2r+2k}.
\end{align}
\end{corollary}

Recall that (see \eqref{C-defn} and \eqref{R-defn}) $R_{2k}$ and $C_{2k}$ denote the generating functions of the $2k$-th moments of ranks and cranks, respectively. From \cite[Corollary 11]{Andrews-Chan-Kim} we find that \footnote{The meanings of $C_k(q)$ and $R_k(q)$ in \cite{Andrews-Chan-Kim} are slightly different from ours. When $k$ is even, they differ by a multiplier 2.} for $k\geq 1$,
\begin{align}
&R_{2k}(q)=\frac{2}{(q;q)_\infty}\sum_{n=1}^\infty (-1)^{n+1}q^{n(3n-1)/2}(1-q^n)\sum_{m=0}^\infty m^{2k}q^{nm}, \label{Rk-exp} \\
&C_{2k}(q)=\frac{2}{(q;q)_\infty}\sum_{n=1}^\infty (-1)^{n+1}q^{n(n-1)/2}(1-q^n)\sum_{m=0}^\infty m^{2k}q^{nm}. \label{Ck-exp}
\end{align}
As noted by Atkin and Garvan \cite{Atkin-Garvan}, $R_{2k}$ and $C_{2k}$  are closely related to quasi-modular forms.

\begin{lemma}\label{lem-moment-quasi}
For any integer $n\geq 1$, we have
\begin{align}
C_{2n}(q)\in P\widetilde{\mathcal{M}}_{2n} \label{C2k}
\end{align}
and
\begin{align}
R_{2n}(q) \in P\widetilde{\mathcal{M}}_{2n} + \oplus_{i=0}^{n-1}\mathbb{Q} \Theta^{i}(R_2). \label{R2k}
\end{align}
\end{lemma}
\begin{proof}
The first assertion, i.e., \eqref{C2k} is proved by Atkin and Garvan. See \cite[Eq.\ (4.8)]{Atkin-Garvan}.

To prove \eqref{R2k}, we shall use induction on $n$.
For $n=1$, \eqref{R2k} holds trivially. Suppose it holds for all $n<k$ where $k$ is some integer greater than 1.

Following \cite[Eq.\ (5.11)]{Atkin-Garvan}, for $k\geq 1$ we define
\begin{align}
T_k=&~ (2k-1)(k-1)R_{2k}+6\sum_{i=1}^{k-1}\binom{2k}{2i}(2^{2i-1}-1)\Theta(R_{2k-2i}) \nonumber \\
&+\sum_{i=1}^{k-1} \left(\binom{2k}{2i+2}(2^{2i+1}-1)-2^{2i}\binom{2k}{2i+1}+\binom{2k}{2i}  \right)R_{2k-2i}. \label{T-defn}
\end{align}
It was shown in \cite[Eq.\ (5.12)]{Atkin-Garvan} that
\begin{align}\label{Tk-belong}
T_k\in P\widetilde{\mathcal{M}}_{2k}.
\end{align}
Then by \eqref{T-defn} and \eqref{Tk-belong} we see that \eqref{R2k} also holds for $n=k$. Therefore, \eqref{R2k} holds for any integer $n\geq 1$.
\end{proof}
\begin{rem}
A more detailed discussion about the modularity and $l$-adic properties for prime $\ell >3$ of $C_{2k}(q)$ and $R_{2k}(q)$ were given by Bringmann, Garvan and Malhburg \cite{BGM}.
\end{rem}

We now relate $R_{2k}(q)$ and $C_{2k}(q)$  to the function $E_{2r}(\tau)/ \eta(\tau)$.
For $1\leq n \leq 5$, from Table \ref{table-dim} we find that
\begin{align}
\dim \widetilde{\mathcal{M}}_{2n}=\frac{n(n+1)}{2}+1.
\end{align}
Note that the set $\mathcal{A}_n=\left\{\Theta^{j}(PE_{2k}): 0\leq j \leq n-k, \, \textrm{$k=0$ or  $2\leq k \leq n$} \right\}$ has $\frac{n(n+1)}{2}+1$ elements, and by Corollary \ref{cor-E-quasi} all these elements belong to $P\widetilde{\mathcal{M}}_{2n}$. Thus, there is a linear relation between $C_{2n}$ with these elements. Employing \eqref{Ck-exp}, we have used \emph{Maple} to find this relation, which in turn gives an expression for $M_{2n}(m)$ in terms of $e_{2r}(m)$ and $p(m)$. For example,
\begin{align}
M_4(n)=&\frac{1}{20}e_4(n)-\frac{1}{20}p(n)+2np(n)-12n^2p(n), \label{M4-expression} \\
M_6(n)=&-\frac{11}{378}e_6(n)+\frac{1}{14}e_4(n)-\frac{3}{14}ne_4(n)-\frac{8}{189}p(n)+\frac{11}{6}np(n)\nonumber \\
& -20n^2p(n)+40n^3p(n). \label{M6-expression}
\end{align}
The formulas for $M_8(n)$ and $M_{10}(n)$ are stated in  Group \uppercase\expandafter{\romannumeral5} in the Appendix.

For $n\geq 6$, besides those in $\mathcal{A}_n$, we need more functions to express $C_{2n}(q)$. Recall the modular discriminant
\begin{align}
\Delta(\tau):=\eta(\tau)^{24},
\end{align}
which is a modular form of weight 12. For any integer $k$ we define $p_k(n)$ by
\begin{align}
\sum_{n=0}^\infty p_k(n)q^n=(q;q)_\infty^k.
\end{align}
Then $p_{-1}(n)=p(n)$ and we have
\begin{align}
P\Delta=\sum_{n=1}^\infty p_{23}(n-1)q^n \in P\widetilde{\mathcal{M}}_{12}.
\end{align}
For $n=6$, note that from Table \ref{table-dim} we have $\dim \widetilde{\mathcal{M}}_{12}=23$. The set $\mathcal{A}_6\cup \{P\Delta\}$
has 23 functions, and thus there is a linear relation between its elements and $C_{12}(q)$. Using \emph{Maple} we find this relation, which gives an expression for $M_{12}(n)$.  In the same way, we can express $M_{14}(n)$ in terms of $e_{2r}(n)$ and $p_{23}(n-1)$. See \eqref{M12-exp} and \eqref{M14-exp} in the Appendix for the formulas for $M_{12}(n)$ and $M_{14}(n)$.

Similarly, using \eqref{R2k} and \eqref{Rk-exp}, we can find expression of $N_{2k}(n)$ in terms of $e_{2r}(n)$ and other suitable functions such as $N_2(n)$ and $p_{23}(n-1)$. For example,
\begin{align}
N_4(n)=&~ \frac{2}{15}e_4(n)-\frac{2}{15}p(n)+4np(n)-36n^2p(n)+N_2(n)-12nN_2(n), \label{N4-expression} \\
N_6(n)=&~ -\frac{1}{21}e_6(n)+\frac{5}{21}e_4(n)-\frac{12}{7}ne_4(n)-\frac{4}{21}p(n)+8np(n)-108n^2p(n)\nonumber \\
&~  +432n^3p(n)+N_2(n)-24nN_2(n)+108n^2N_2(n). \label{N6-expression}
\end{align}
Formulas for $N_8(n), N_{10}(n), N_{12}(n)$ and $N_{14}(n)$ are recorded in Group \uppercase\expandafter{\romannumeral6} of the Appendix.

\subsection{Congruences for the fourth and sixth moments of cranks}
The expressions of $M_k(n)$ and $N_k(n)$ in terms of $e_{2r}(n)$ allow us to establish congruences for them. Here we only take $k=4$ and 6 as examples.

\begin{theorem}\label{thm-M4-cong}
For $k\geq 1$ and $n\geq 0$, we have
\begin{align}
&M_4(5^k n+\delta_{5,k})\equiv 0   \pmod{5^{k-1}}, \label{M4-5power} \\
& M_4(7^k n+\delta_{7,k}) \equiv 0 \pmod{7^{\left\lfloor \frac{k}{2}\right\rfloor+1}}, \label{M4-7power} \\
&M_4(11^k n+\delta_{11,k}) \equiv 0 \pmod{{11}^k}. \label{M4-11power}
\end{align}
\end{theorem}
\begin{proof}
From \eqref{M4-expression}, we see that \eqref{M4-5power} follows from \eqref{pn-mod5} and  Theorem \ref{thm-e4-5-cong}, \eqref{M4-7power} follows from \eqref{pn-mod7} and Theorem \ref{thm-e4-7-cong}, and  \eqref{M4-11power} follows from \eqref{pn-mod11} and Theorem \ref{thm-e4-11-cong}.
\end{proof}

In the same way, we can use \eqref{M6-expression} to establish  congruences for $M_6(n)$.
\begin{theorem}\label{thm-M6-cong}
For $k\geq 1$ and $n\geq 0$, we have
\begin{align}
&M_6(5^kn+ \delta_{5,k}) \equiv 0 \pmod{5^k}, \label{M6-5power} \\
&M_6(7^kn+\delta_{7,k})\equiv 0 \pmod{7^{\left\lfloor \frac{k}{2}\right\rfloor}}, \label{M6-7power} \\
&M_6(11^kn+\delta_{11,k})\equiv 0 \pmod{11^k}. \label{M6-11power}
\end{align}
\end{theorem}

From \eqref{mu-defn} we know that $\mu_{2k}(n)$ can be expressed as linear sum of $M_{2i}(n)$ ($1\leq i \leq k$). For example, we have
\begin{align}
\mu_4(n)=&~ \frac{1}{24}M_4(n)-\frac{1}{24}M_2(n), \label{mu4-M4} \\
\mu_6(n)=&~ \frac{1}{720}M_6(n)-\frac{1}{144}M_4(n)+\frac{1}{180}M_2(n). \label{mu6-M6}
\end{align}
Together with \eqref{M4-expression} and \eqref{M6-expression}, we deduce that
\begin{align}
\mu_4(n)=&~ \frac{1}{480}e_4(n)-\frac{1}{480}p(n)-\frac{1}{2}n^2p(n), \label{mu4-e4} \\
\mu_6(n)=&~ -\frac{11}{272160}e_6(n)-\frac{1}{4032}e_4(n)-\frac{1}{3360}ne_4(n)+\frac{157}{544320}p(n)\nonumber \\
&~ -\frac{1}{4320}np(n)+\frac{1}{18}n^2p(n)+\frac{1}{18}n^3p(n). \label{mu6-e6}
\end{align}
These formulas lead to the following congruences for $\mu_4(n)$ and $\mu_6(n)$.
\begin{theorem}\label{thm-mu-cong}
For $k\geq 1$ and $n\geq 0$, we have
\begin{align}
&\mu_4(5^kn+\delta_{5,k})\equiv 0 \pmod{5^{k-1}}, \label{mu4-5power} \\
&\mu_4(7^kn+\delta_{7,k})\equiv 0 \pmod{7^{\left\lfloor\frac{k}{2}\right\rfloor+1} }, \label{mu4-7power}\\
&\mu_4(11^kn+\delta_{11,k})\equiv 0 \pmod{{11}^{k}}, \label{mu4-11power} \\
&\mu_6 (5^kn+\delta_{5,k})\equiv 0 \pmod{5^{k-1}}, \label{mu6-5power} \\
&\mu_6(7^kn+\delta_{7,k})\equiv 0 \pmod{7^{\left\lfloor\frac{k}{2}\right\rfloor} }, \label{mu6-7power}\\
&\mu_6(11^kn+\delta_{11,k})\equiv 0 \pmod{{11}^{k}}. \label{mu6-11power}
\end{align}
\end{theorem}

\subsection{Congruences for the fourth and sixth moments of ranks}

From \eqref{N4-expression}, \eqref{N6-expression} and congruences for $e_4(n), e_6(n)$ and $N_2(n)$, we can establish congruences for $N_4(n)$ and $N_6(n)$.
\begin{theorem}\label{thm-N4-cong}
For $k\geq 1$ and $n\geq 0$, we have
\begin{align}
&N_4(5^kn+\delta_{5,k})\equiv 0 \pmod{5^{\left\lfloor \frac{k+1}{2} \right\rfloor-\gamma_{k,1}}}, \label{N4-5power} \\
&N_4(7^kn+\delta_{7,k})\equiv 0 \pmod{7^{\left\lfloor \frac{k+1}{2} \right\rfloor}}. \label{N4-7power}
\end{align}
\end{theorem}

\begin{theorem}\label{thm-N6-cong}
For $k\geq 1$ and $n\geq 0$, we have
\begin{align}
&N_6(5^kn+ \delta_{5,k}) \equiv 0 \pmod{5^{\left\lfloor \frac{k+1}{2}\right\rfloor}}, \label{N6-5power} \\
&M_6(7^kn+\delta_{7,k})\equiv 0 \pmod{7^{\left\lfloor \frac{k+1}{2}\right\rfloor-\gamma_{k,1}}}. \label{N6-7power}
\end{align}
\end{theorem}

From \eqref{eta-defn} it is easy to see that $\eta_{2k}(n)$ can be expressed as linear sum of $N_{2i}(n)$ ($1\leq i \leq k$). For example, we have
\begin{align}
&\eta_4(n)=\frac{1}{24}N_4(n)-\frac{1}{24}N_2(n), \label{eta4-N4} \\
&\eta_6(n)=\frac{1}{720}N_6(n)-\frac{1}{144}N_4(n)+\frac{1}{180}N_2(n). \label{eta6-N6}
\end{align}
Therefore,  we can also express $\eta_{2k}(n)$ using $e_{2r}(n)$, $p(n)$ and $N_2(n)$ and other suitable functions. Alternatively, since $\eta_{2k}(n)$ is a linear sum of  $N_{2i}(n)$ ($1\leq i \leq k$), it follows from Lemma \ref{lem-moment-quasi} that
\begin{align}
\sum_{n=0}^\infty \eta_{2k}(n)q^n \in  P\widetilde{\mathcal{M}}_{2k} + \oplus_{i=0}^{k-1}\mathbb{Q} \Theta^{i}(R_2).
\end{align}
Thus as in the previous subsections, we can use \emph{Maple} to help us to find representations of $\eta_{2k}(n)$. For example, we find that
\begin{align}
\eta_4(n)=&~ \frac{1}{180}e_4(n)-\frac{1}{180}p(n)+\frac{1}{6}np(n)-\frac{3}{2}n^2p(n)-\frac{1}{2}nN_2(n), \label{eta4-exp}\\
\eta_6(n)=&~ -\frac{1}{15120}e_6(n)-\frac{1}{1680}e_4(n)-\frac{1}{420}ne_4(n)+\frac{1}{1512}p(n)-\frac{1}{60}np(n)\nonumber \\
&~ +\frac{1}{10}n^2p(n)+\frac{3}{5}n^3p(n)+\frac{1}{20}nN_2(n)+\frac{3}{20}n^2N_2(n). \label{eta6-exp}
\end{align}
These expressions together with the known congruences for $e_4(n)$, $e_6(n)$, $p(n)$ and $N_2(n)$ yield the following result.
\begin{theorem}\label{thm-eta-cong}
For $k\geq 1$ and $n\geq 0$, we have
\begin{align}
&\eta_4(5^kn+\delta_{5,k})\equiv 0 \pmod{5^{\left\lfloor \frac{k+1}{2} \right\rfloor-\gamma_{k,1}}}, \label{eta4-5power} \\
&\eta_4(7^kn+\delta_{7,k})\equiv 0 \pmod{7^{\left\lfloor \frac{k+1}{2} \right\rfloor}}, \label{eta4-7power} \\
&\eta_6(5^kn+\delta_{5,k})\equiv 0 \pmod{5^{\left\lfloor \frac{k+1}{2}\right\rfloor-\gamma_{k,1}}}, \label{eta6-5power}\\
&\eta_6(7^kn+\delta_{7,k})\equiv 0\pmod{7^{\left\lfloor\frac{k+1}{2}\right\rfloor-\gamma_{k,1}}}. \label{eta6-7power}
\end{align}
\end{theorem}

As a byproduct, we can also establish congruences for higher order $\spt$-functions.
Recall from  \eqref{sptk-defn} we have
\begin{align}
\spt_k(n)=\mu_{2k}(n)-\eta_{2k}(n).
\end{align}
From those expressions for $\mu_{2k}(n)$ and $\eta_{2k}(n)$, we can express $\spt_k(n)$ in terms of functions such as $e_{2r}(n)$, $p(n)$ and $N_2(n)$. For instance, we have
\begin{align}
&\spt_2(n)=-\frac{1}{288}e_4(n)+\frac{1}{2}nN_2(n)+\frac{1}{288}p(n)-\frac{1}{6}np(n)+n^2p(n), \label{spt2-exp}\\
&\spt_3(n)=\frac{1}{38880}e_6(n)+\frac{1}{2880}e_4(n)+\frac{1}{480}ne_4(n)-\frac{29}{77760}p(n)+\frac{71}{4320}np(n) \nonumber \\
&\quad \quad \quad \quad -\frac{2}{45}n^2p(n)-\frac{49}{90}n^3p(n)-\frac{1}{20}nN_2(n)-\frac{3}{20}n^2N_2(n). \label{spt3-exp}
\end{align}
These formulas lead to the following congruences for $\spt_k(n)$.
\begin{theorem}\label{thm-spt2}
For $k\geq 1$ and $n\geq 0$, we have
\begin{align}
\spt_2(5^kn+\delta_{5,k})\equiv 0 \pmod{5^{\left\lfloor \frac{k+1}{2}\right\rfloor}}, \label{spt2-5power} \\
\spt_2(7^kn+\delta_{7,k})\equiv 0 \pmod{7^{\left\lfloor\frac{k+1}{2}\right\rfloor} }, \label{spt2-7power}\\
\spt_3(5^kn+\delta_{5,k})\equiv 0 \pmod{5^{\left\lfloor \frac{k-1}{2}\right\rfloor}}, \label{spt3-5power} \\
\spt_3(7^kn+\delta_{7,k})\equiv 0 \pmod{7^{\left\lfloor\frac{k+1}{2}\right\rfloor} }. \label{spt3-7power}
\end{align}
\end{theorem}

\subsection*{Acknowledgements}
The first author was partially supported by the National Natural Science Foundation of China (11801424), the Fundamental Research Funds for the Central Universities (Project No.\ 2042018kf0027, Grant 1301--413000053) and a start-up research grant (1301--413100048) of the Wuhan University. The second author was partially supported by Grant 106-2115-M-002-009-MY3 of the Ministry
of Science and Technology, Taiwan (R.O.C.).

\section{Appendix}
\subsection{Formulas for the action of $U$-operators}
\noindent \\

Group \uppercase\expandafter{\romannumeral1}:
\begin{align*}
&\mathcal{E}_1\mid U_7=\mathcal{E}_1, \quad   \left(\mathcal{E}_1 Z_{7}^{-1} \right)\mid U_7=\mathcal{E}_1, \quad \left(\mathcal{E}_1 Z_{7}^{-2} \right)\mid U_7=7\mathcal{E}_1, \\
&\left(\mathcal{E}_1 Z_{7}^{-3} \right)\mid U_7=-7\mathcal{E}_1, \quad \left(\mathcal{E}_1 Z_{7}^{-4} \right)\mid U_7=\mathcal{E}_1(-2Y_{7}^{-1}-7^2), \quad   \left(\mathcal{E}_1 Z_{7}^{-5} \right)\mid U_7=49\mathcal{E}_1, \\
&\left(\mathcal{E}_1 Z_{7}^{-6} \right)\mid U_7=\mathcal{E}_1(-8\cdot 7Y_{7}^{-1}-7^3).
\end{align*}

Group \uppercase\expandafter{\romannumeral2}:
\begin{align*}
&1 \mid U_7=1, \quad Z_7^{-1}\mid U_7=-1, \quad Z_7^{-2}\mid U_7=1, \quad Z_7^{-3}\mid U_7=-7,\\
&Z_7^{-4}\mid U_7 =-4Y_7^{-1}-7, \quad Z_7^{-5}\mid U_7 =10Y_{7}^{-1}+7^2, \quad Z_7^{-6}\mid U_7=7^2.
\end{align*}

Group \uppercase\expandafter{\romannumeral3}
\begin{align*}
&\left(\widetilde{E}_{4,13}Z_{13}^{-4}\right)\mid U_{13}=\widetilde{E}_{4,13}\left(45Y_{13}^{-2}+885\cdot 13Y_{13}^{-1}+900\cdot {13}^2+315\cdot {13}^3Y_{13}\right.\\
&\qquad \qquad \qquad \qquad \left. +45\cdot {13}^4Y_{13}^2 +0\cdot Y_{13}^3-{13}^5Y_{13}^4 \right), \\
&\left(\widetilde{E}_{4,13}Z_{13}^{-3}\right)\mid U_{13}=\widetilde{E}_{4,13}\left(-138Y_{13}^{-1}+124\cdot 13+ {13}^5Y_{13}^4 \right), \\
&\left(\widetilde{E}_{4,13}Z_{13}^{-2}\right)\mid U_{13}=\widetilde{E}_{4,13}\left(-9Y_{13}^{-1}-171\cdot 13-180\cdot {13}^2Y_{13}-63\cdot {13}^3Y_{13}^2 \right. \\
&\qquad \qquad \qquad \qquad \left. -9\cdot {13}^4Y_{13}^3-{13}^4Y_{13}^4 \right), \\
&\left(\widetilde{E}_{4,13}Z_{13}^{-1}\right)\mid U_{13}=\widetilde{E}_{4,13}\left(124+18\cdot 13Y_{13} \right), \\
&\left(\widetilde{E}_{4,13}\right)\mid U_{13}=\widetilde{E}_{4,13}(1+19\cdot 13Y_{13}+20\cdot 13^2 Y_{13}^2+7\cdot 13^{3} Y_{13}^3 +12\cdot {13}^3Y_{13}^4), \\
&\left(\widetilde{E}_{4,13}Z_{13}\right)\mid U_{13}=\widetilde{E}_{4,13}(-46Y_{13}-20\cdot 13Y_{13}^2+0\cdot Y_{13}^3+{13}^3Y_{13}^4), \\
&\left(\widetilde{E}_{4,13}Z_{13}^2\right)\mid U_{13}=\widetilde{E}_{4,13}(-{13}^2Y_{13}^4), \\
&\left(\widetilde{E}_{4,13}Z_{13}^3\right)\mid U_{13}=\widetilde{E}_{4,13}(10Y_{13}^2+8\cdot 13 Y_{13}^3+{13}^2Y_{13}^4), \\
&\left(\widetilde{E}_{4,13}Z_{13}^4\right)\mid U_{13}=\widetilde{E}_{4,13}(-13Y_{13}^4), \\
&\left(\widetilde{E}_{4,13}Z_{13}^5\right)\mid U_{13}=\widetilde{E}_{4,13}(-13Y_{13}^4), \\
&\left(\widetilde{E}_{4,13}Z_{13}^6\right)\mid U_{13}=\widetilde{E}_{4,13}(-Y_{13}^4), \\
&\left(\widetilde{E}_{4,13}Z_{13}^7\right)\mid U_{13}=\widetilde{E}_{4,13}(-Y_{13}^4), \\
&\left(\widetilde{E}_{4,13}Z_{13}^8\right)\mid U_{13}=\widetilde{E}_{4,13}(19Y_{13}^5+20\cdot 13Y_{13}^6+7\cdot {13}^2Y_{13}^7+{13}^3Y_{13}^8).
\end{align*}

Group \uppercase\expandafter{\romannumeral4}:
\begin{align*}
&\widetilde{E}_{6,13}\mid U_{13}=\widetilde{E}_{6,13}\left(1-38\cdot 13 Y_{13}-122\cdot {13}^2Y_{13}^2-108\cdot {13}^3Y_{13}^3-46\cdot {13}^4Y_{13}^4 \right.\\
&\qquad \qquad \qquad \left. -10\cdot {13}^5Y_{13}^5-12\cdot {13}^5Y_{13}^6  \right), \\
&\left(\widetilde{E}_{6,13}Z_{13} \right) \mid U_{13}=\widetilde{E}_{6,13}\left(258Y_{13}+542\cdot 13Y_{13}^2+250\cdot {13}^2Y_{13}^3+4\cdot {13}^4Y_{13}^4 \right.\\
&\qquad \qquad \qquad \qquad \left.+0\cdot Y_{13}^5 -{13}^5Y_{13}^6   \right), \\
&\left(\widetilde{E}_{6,13}Z_{13}^2 \right) \mid U_{13}=\widetilde{E}_{6,13}\left({13}^4Y_{13}^6  \right), \\
&\left(\widetilde{E}_{6,13}Z_{13}^3 \right) \mid U_{13}=\widetilde{E}_{6,13}\left(-32Y_{13}^2-76\cdot 13Y_{13}^3-44\cdot {13}^2Y_{13}^4-14\cdot {13}^3Y_{13}^5-{13}^4Y_{13}^6  \right), \\
&\left(\widetilde{E}_{6,13}Z_{13}^4 \right) \mid U_{13}=\widetilde{E}_{6,13}\left({13}^3Y_{13}^6 \right), \\
&\left(\widetilde{E}_{6,13}Z_{13}^5 \right) \mid U_{13}=\widetilde{E}_{6,13}\left(-8Y_{13}^3-8\cdot {13}Y_{13}^4+0\cdot Y_{13}^5 +{13}^3Y_{13}^6  \right), \\
&\left(\widetilde{E}_{6,13}Z_{13}^6 \right) \mid U_{13}=\widetilde{E}_{6,13}\left( {13}^2Y_{13}^6 \right), \\
&\left(\widetilde{E}_{6,13}Z_{13}^7 \right) \mid U_{13}=\widetilde{E}_{6,13}\left(6Y_{13}^4+6\cdot 13Y_{13}^5+{13}^2Y_{13}^6  \right), \\
&\left(\widetilde{E}_{6,13}Z_{13}^8  \right) \mid U_{13}=\widetilde{E}_{6,13}\left(13Y_{13}^6  \right), \\
&\left(\widetilde{E}_{6,13}Z_{13}^9  \right) \mid U_{13}=\widetilde{E}_{6,13}\left(-2Y_{13}^5-13Y_{13}^6  \right), \\
&\left(\widetilde{E}_{6,13}Z_{13}^{10}\right) \mid U_{13}=\widetilde{E}_{6,13}\left( Y_{13}^6 \right), \\
&\left(\widetilde{E}_{6,13}Z_{13}^{11} \right) \mid U_{13}=\widetilde{E}_{6,13}\left(Y_{13}^6  \right), \\
&\left(\widetilde{E}_{6,13}Z_{13}^{12} \right) \mid U_{13}=\widetilde{E}_{6,13}\left(38Y_{13}^7+122\cdot 13Y_{13}^8+108\cdot {13}^2Y_{13}^9+46\cdot {13}^3Y_{13}^{10}\right. \\
&\quad \quad \qquad \qquad \qquad \qquad  \left. +10\cdot {13}^4Y_{13}^{11}+{13}^5Y_{13}^{12}  \right).
\end{align*}

\subsection{Formulas for the rank and crank moments}
\noindent \\

Group \uppercase\expandafter{\romannumeral5}:
\begin{align}
M_8(n)=&\frac{83}{2160}e_8(n)-\frac{2}{27}e_6(n)+\frac{4}{27}ne_6(n)+\frac{71}{1080}e_4(n) -\frac{5}{9}ne_4(n)+\frac{2}{3}n^2e_4(n)\nonumber \\
&-\frac{13}{432}p(n)+\frac{41}{27}np(n)-\frac{70}{3}n^2p(n)+112n^3p(n)-112n^4p(n), \label{M8-exp} \\
M_{10}(n)=& -\frac{2173}{25740}e_{10}(n)+\frac{83}{540}e_8(n)-\frac{83}{360}ne_8(n)-\frac{47}{468}e_6(n)+\frac{70}{117}ne_6(n)\nonumber \\
& -\frac{20}{39}n^2e_6(n)+\frac{61}{1188}e_4(n)-\frac{103}{132}ne_4(n)+\frac{30}{11}n^2e_4(n)-\frac{20}{11}n^3e_4(n)\nonumber \\
&-\frac{2}{99}p(n) +\frac{85}{72}np(n)-\frac{70}{3}n^2p(n)+180n^3p(n)-480n^4p(n)+288n^5p(n), \label{M10-exp} \\
M_{12}(n)=& \frac{1892286317}{6856799040}e_{12}(n) -\frac{2173}{4446}e_{10}(n)+\frac{2173}{3705}ne_{10}(n)\nonumber \\
&~ +\frac{42911}{146880}e_8(n)-\frac{913}{680}ne_8(n)+\frac{913}{1020}n^2e_8(n)-\frac{4565}{44226}e_{6}(n)\nonumber \\
&~ +\frac{407}{351}ne_6(n)-\frac{352}{117}n^2e_6(n)+ \frac{176}{117}n^3e_6(n)+\frac{5827}{157248}e_4(n)\nonumber \\
&~ -\frac{2749}{3276}ne_4(n)+ \frac{141}{26}n^2e_4(n)-\frac{140}{13}n^3e_4(n)+ \frac{60}{13}n^4e_4(n)\nonumber \\
&~ -\frac{8009}{606528}p(n)+ \frac{571}{648}np(n)-\frac{2299}{108}n^2p(n)+ \frac{6116}{27}n^3p(n)\nonumber \\
&~ -\frac{3124}{3}n^4p(n)+ 1760n^5p(n)-704n^6p(n)  -\frac{17147966}{26113581}p_{23}(n-1), \label{M12-exp}\\
M_{14}(n)=& -\frac{120667369}{96279840}e_{14}(n)+\frac{1892286317}{866518560}e_{12}(n)-\frac{1892286317}{866518560}ne_{12}(n) \nonumber \\
&~ -\frac{767069}{615600}e_{10}(n)+\frac{23903}{5130}ne_{10}(n) -\frac{2173}{855}n^2e_{10}(n)+\frac{12236939}{31395600}e_8(n) \nonumber \\
&~ -\frac{8059051}{2325600}ne_8(n)+\frac{83083}{11628}n^2e_8(n)-\frac{83083}{29070}n^3e_8(n)-\frac{12067}{132192}e_6(n) \nonumber \\
&~ + \frac{13123}{8262}ne_6(n) -\frac{3619}{459}n^2e_6(n)+\frac{616}{51}n^3e_6(n)-\frac{616}{153}n^4e_6(n)\nonumber \\
&~ +\frac{199}{7776}e_4(n) -\frac{2023}{2592}ne_4(n)+\frac{415}{54}n^2e_4(n) -\frac{259}{9}n^3e_4(n)+\frac{112}{3}n^4e_4(n)\nonumber \\
&~ -\frac{56}{5}n^5e_4(n)-\frac{31}{3645}p(n) +\frac{14917}{23328}np(n)-\frac{5915}{324}n^2p(n)+ \frac{4459}{18}n^3p(n) \nonumber \\
&~ -\frac{44408}{27}n^4p(n)+\frac{74984}{15}n^5p(n) -5824n^6p(n)+ 1664n^7p(n) \nonumber \\ &-\frac{240071524}{46200951}p_{23}(n-1)+\frac{240071524}{46200951}np_{23}(n-1). \label{M14-exp}
\end{align}

Group \uppercase\expandafter{\romannumeral6}:
\begin{align}
N_8(n)=&~ \frac{26}{495}e_8(n)-\frac{14}{99}e_6(n)+\frac{8}{11}ne_6(n)+\frac{14}{45}e_4(n)-\frac{16}{3}ne_4(n) +16n^2 e_4(n) \nonumber \\
&~ -\frac{2}{9}p(n)+12np(n)-228n^2p(n)+1728n^3p(n)-3888n^4p(n)+N_2(n)\nonumber \\
&~ -36nN_2(n)+360n^2N_2(n)-864n^3N_2(n), \label{N8-exp} \\
N_{10}(n)=&~ -\frac{227}{2145}e_{10}(n)+ \frac{13}{55}e_8(n)- \frac{52}{55}ne_8(n)- \frac{36}{143}e_6(n)+ \frac{480}{143}ne_6(n) \nonumber \\ &~ -\frac{1080}{143}n^2e_6(n)+ \frac{4}{11}e_4(n)- \frac{112}{11}ne_4(n)+ \frac{840}{11}n^2e_4(n)- \frac{1440}{11}n^3e_4(n) \nonumber \\
 &~- \frac{8}{33}p(n)+16np(n)-396n^2p(n)+4464n^3p(n)-21600n^4p(n)\nonumber \\
 &~ +31104n^5p(n)+N_2(n)-48nN_2(n)+756n^2N_2(n)-4320n^3N_2(n)\nonumber \\
 &~ +6480n^4N_2(n), \label{N10-exp} \\
N_{12}(n)=&~ \frac{145181}{440895}e_{12}(n)- \frac{2497}{3705}e_{10}(n)+ \frac{2724}{1235}ne_{10}(n)+ \frac{143}{255}e_{8}(n)- \frac{104}{17}ne_8(n)\nonumber \\
&~ + \frac{936}{85}n^2e_8(n)- \frac{33}{91}e_6(n)+ \frac{108}{13}ne_6(n)- \frac{648}{13}n^2e_6(n)+ \frac{864}{13}n^3e_6(n)\nonumber \\
&~ + \frac{110}{273}e_4(n) -\frac{1440}{91}ne_4(n)+ \frac{2592}{13}n^2e_4(n)- \frac{11520}{13}n^3e_4(n) \nonumber \\
&~ + \frac{12960}{13}n^4e_4(n)- \frac{10}{39}p(n) +20np(n)-612n^2p(n)+9216n^3p(n)\nonumber \\
&~ -69552n^4p(n)+233280n^5p(n) -233280n^6p(n)+N_2(n)-60nN_2(n)\nonumber \\
&~ +1296n^2N_2(n) -12096n^3N_2(n) +45360n^4N_2(n)-46656n^5N_2(n) \nonumber \\
&~ -\frac{2664576}{2901509}p_{23}(n-1), \label{N12-exp} \\
N_{14}(n)=&~-\frac{107637}{74290}e_{14}(n)+\frac{1887353}{668610}e_{12}(n)-\frac{290362}{37145}ne_{12}(n)-\frac{2951}{1425}e_{10}(n)\nonumber \\
&~ +\frac{1816}{95}ne_{10}(n) -\frac{2724}{95}n^2e_{10}(n)+\frac{24167}{24225}e_{8}(n) -\frac{156156}{8075}ne_8(n)\nonumber \\
&~ + \frac{156156}{1615}n^2e_8(n) -\frac{170352}{1615}n^3e_8(n) -\frac{143}{306}e_6(n)+\frac{264}{17}ne_6(n) \nonumber \\
 &~ -\frac{2772}{17}n^2e_6(n) + \frac{10080}{17}n^3e_6(n) -\frac{9072}{17}n^4e_6(n) + \frac{13}{30}e_4(n) \nonumber \\
 &~ -22ne_4(n)+396n^2e_4(n) -3024n^3e_4(n)+ 9072n^4e_4(n) -\frac{36288}{5}n^5e_4(n) \nonumber \\
&~ -\frac{4}{15}p(n)  + 24np(n) -876n^2p(n) +16560n^3p(n)-171072n^4p(n) \nonumber \\
&~ +\frac{4644864}{5}n^5p(n)-2286144n^6p(n)+1679616n^7p(n)+N_2(n) -72nN_2(n)\nonumber \\
&~ + 1980n^2N_2(n) -25920n^3N_2(n) + 163296n^4N_2(n) -435456n^5N_2(n) \nonumber \\
&~ + 326592n^6N_2(n)  -\frac{40412736}{5133439}p_{23}(n-1) \nonumber \\
&~ + \frac{111912192}{5133439}np_{23}(n-1). \label{N14-exp}
\end{align}

\end{document}